\documentclass[10pt]{article}
\usepackage[utf8]{inputenc}
\usepackage{amsmath,amsthm,graphicx}
\usepackage{subcaption}
\usepackage{amssymb}
\usepackage{tikz}
\usetikzlibrary{patterns,patterns.meta,arrows.meta,calc,3d,perspective,fadings,shadows,cd}
\usepackage{geometry}
\usepackage[colorlinks,allcolors=.]{hyperref}

\newenvironment{alcases}{\left\{\begingroup\arraycolsep=0pt\begin{array}{r@{\;}l@{\quad}l}}{\end{array}\endgroup\right.}

\theoremstyle{plain}
	\newtheorem{theo}{Theorem}
	\newtheorem*{thm1}{Theorem 1}
	\newtheorem*{thm3}{Theorem 2}
	\newtheorem{coro}{Corollary}
	\newtheorem*{cor1}{Corollary 1}
	\newtheorem*{cor2}{Corollary 2}
	\newtheorem{lemm}{Lemma}[section]
	\newtheorem{prop}[lemm]{Proposition}
	\newtheorem*{prop1}{Proposition 3.9.}

\theoremstyle{definition}
	\newtheorem{defi}{Definition}[section]
	\newtheorem*{nota}{Notations}

	\newtheorem{rema}{Remark}[section]
	\newtheorem*{remas}{Remarks}
	\newtheorem*{exam}{Example}

\newcommand{\bigslant}[2]{{\raisebox{.2em}{$#1$}\left/\raisebox{-.2em}{$#2$}\right.}}
\newcommand{\arr}{\buildrel{\pm 1}\over{\longrightarrow}}
\newcommand{\C}{\mathbb{C}}
\newcommand{\R}{\mathbb{R}}
\newcommand{\Q}{\mathbb{Q}}
\newcommand{\Z}{\mathbb{Z}}
\newcommand{\N}{\mathbb{N}}
\newcommand{\T}{\mathbb{T}}
\newcommand{\F}{\mathbb{F}}
\newcommand{\bbP}{\mathbb{P}}
\renewcommand{\P}{\mathrm{P}}
\renewcommand{\t}{\mathfrak{t}}
\DeclareMathOperator{\Sed}{Sed}

\DeclareMathOperator{\Sd}{Sd}
\DeclareMathOperator{\codim}{codim}
\DeclareMathOperator{\vol}{vol}
\DeclareMathOperator{\rk}{rk}
\DeclareMathOperator{\Res}{Res}
\DeclareMathOperator{\card}{card}
\DeclareMathOperator{\coker}{coker}
\DeclareMathOperator{\Hom}{Hom}
\DeclareMathOperator{\id}{id}
\DeclareMathOperator{\Tor}{Tor}
\DeclareMathOperator{\df}{d}
\DeclareMathOperator{\convhull}{Conv.Hull}
\DeclareMathOperator{\Conv}{Conv}

\newcommand{\Cell}{\mathbf{Cell}}
\newcommand{\Mod}{\mathbf{Mod}}
\newcommand{\Fbar}{\hat{F}}
\newcommand{\Hbar}{\hat{H}}
\newcommand{\ibar}{\hat{\imath}}
\newcommand{\smallslant}[2]{{#1}/{#2}}
\newcommand{\ori}{[\t(\Z)]}

\title{A Poincaré--Lefschetz Theorem for Cellular Cosheaves and an Application to the Tropical Homology of Orbifold Toric Varieties}

\author{Jules Chenal}
\date{}

\begin{document}

\maketitle

\noindent\textit{This is the Author Accepted Manuscript version of the article:}

\vspace{5pt}

Jules Chenal, A Poincaré--Lefschetz Theorem for Cellular Cosheaves and an Application to the Tropical Homology of Orbifold Toric Varieties, \emph{Annales Henri Lebesgue}, Volume 8 (2025), pp.277--327, DOI: \url{https://doi.org/10.5802/ahl.236}.

\vspace{5pt}

\noindent\textit{The article is published under license} CC BY 4.0.

\begin{abstract}
We begin by presenting a version of the Poincaré--Lefschetz theorem for certain cellular cosheaves on a particular subdivision of a CW-complex~$K$. To that end we construct a cellular sheaf on $K$ whose cohomology with compact support is isomorphic to the homology of the initial cosheaf. Thereafter we use the first result to generalise the tropical version of the Lefschetz Hyperplane Section Theorem to some singular tropical toric varieties and singular tropical hypersurfaces.
\end{abstract}

\section{Introduction}

In this article we study a property of duality between sheaves and their dual objects, cosheaves, and then apply the said property in order to extend a classical theorem of complex geometry to polyhedral subdivisions of polytopes. Philosophically, sheaves can be seen as a generalisation of the concept of functions. They are extensively used in algebraic geometry and their cohomology, at least over $\C$, can be used to recover some of the algebro-topological properties of the analytification of the variety supporting them. Our final goal here will be to generalise a theorem that mimics the Lefschetz Hyperplane Theorem for subdivided integral polytopes endowed with the tropical sheaves of I.~Itenberg, Katzarkov, G.~Mikhalkin and I.~Zharkov.
We want to provide some context about these objects and explain why a theorem of complex geometry might apply to them. Given an algebraic hypersurface $Z\subset(\C^\times)^n$, I.~Gelfand, M.~Kapranov, and A.~Zelevinsky introduced its amoeba $A\subset\R^n$. This amoeba is the image of $Z$ under the map $\log:(z_1;\dots\,;z_n)\mapsto (\log|z_1|;\dots\,;\log|z_n|)$. Its complementary set has convex connected components that are associated to specific properties of the polynomial that defines $Z$. Further, the renormalised limit sets $\lim_{t\to+\infty} \frac{\log Z_t}{\log(t)}$ of families of hypersurfaces $(Z_t)_{t\,>\,1}$ came into play and it birthed tropical geometry. At first glance, tropical geometry can be seen as the geometry of these limit sets. This theory has nice ties to convex affine geometry and can be interpreted as the algebraic geometry of varieties over the tropical ``field''. The latter is not a field per se but, if we continue the analogy, the role of polynomials would be played by some piecewise affine convex functions. The main protagonists of tropical geometry are polyhedral complexes endowed with sheaves. Here we will be given a polytope $P$ and a tropical hypersurface $X$ that respectively correspond to the ``tropical'' locus of a toric variety $Y$ and the limit set $\lim_{t\to+\infty} \frac{\log Z_t}{\log(t)}$ of a family of hypersurfaces $(Z_t)_{t\,>\,1}$ of $Y$. In favorable cases, the cohomology of these sheaves respectively recovers the Hodge numbers of $Y$ and of a generic member $Z_t$ of the family that converges to $X$. It rose the question: \emph{Which theorems of complex geometry extend to these objects?} C.~Arnal, A.~Renaudineau and K.~Shaw proved the first tropical version of the Lefschetz Hyperplane Theorem. We generalise it here. Ultimately, the objects of interest are CW-complexes endowed with a special kind of sheaves that are adapted to the cellular structure.

Given a regular CW-complex $K$, we define \emph{dihomologic cosheaves} on $K$ to be a mild generalisation of the concept of cellular cosheaves on a CW-complex. These objects can be seen as systems of coefficients that associate a group to every pair of adjacent cells of the complex $K$. For a broad variety of examples these cosheaves correspond to classical cellular cosheaves on a suitable subdivision of $K$. Let $F$ be a dihomologic cosheaf on a regular CW-complex $K$. We tackled the question: \emph{Can we compute the homology of $F$ from a reduced quantity of data?} Given a set of hypotheses on the \emph{local homology} of $F$, we are able to construct a cellular sheaf $H_n(F_*)$ on $K$ whose cohomology with compact support is isomorphic to the homology of $F$. This sheaf represents the homology of $F$ in neighbourhoods of cells of $K$. We went from data carried by adjacent pairs of cells in $K$ to data carried by individual cells.

\begin{thm1}[Cellular Poincaré--Lefschetz Theorem]
Let $K$ be a finite dimensional, locally finite and regular CW-complex, $n$ be a non-negative integer, and $F$ be a dihomologic cosheaf on $K$ whose sheaves of local homology\footnote{C.f. Definition~\ref{def:loc_hom_shea}.} $H_q(F_*)$ vanish for all $q\neq n$. Then, for all integers $k$, $H_k(K;F)$ and $H_c^{n-k}(K;H_n(F_*))$ are canonically isomorphic. In particular, $H_k(K;F)$ vanishes for all $k>n$. If in addition $K$ has dimension $n$, then this isomorphism comes from an injective quasi-isomorphism from the complex of compactly supported cellular cochains with values in $H_n(F_*)$ to the complex of dihomologic chains with values in $F$.
\end{thm1}

This statement reminded us of the Poincaré--Lefschetz duality which can be found as one of its direct corollaries.

\begin{cor1}
If $X$ is a homology $n$-manifold in the sense of Definition~\ref{definition:homo_man} then $H_k(X;\Z)\cong H^{n-k}_c(X;\partial X; o_\Z)$ where $o_\Z$ denotes the system of local orientations defined on $X\setminus\partial X$ by $x\mapsto H_n(X;X-x;\Z)$.
\end{cor1}

This is the application of Theorem~\ref{thm:cell_poinca_lef} to the constant cosheaf $\Z$. In this special case, the proof is the same as the one given by Zeeman in~\cite[Theorem~1 p.~159]{Zee_dih_iii}.
We want to emphasise that we chose the name \emph{dihomologic} in reference to Zeeman's theory of dihomology~\cite{Zee_dih_i,Zee_dih_ii,Zee_dih_iii}. A generalisation of Theorem~\ref{thm:cell_poinca_lef} could be stated without the assumptions\footnote{``$H_q(F_*)$ vanishes for all $q\neq n$''.} on the sheaves of local homology of $F$. However, in this case, we would associate a complex of cellular sheaves on $K$ to $F$ whose cohomology with compact support (or more precisely hypercohomology) would be isomorphic to the homology of $F$. If we follow this road and further assume that $F$ is the subdivision of a cosheaf on $K$ then the generalisation would be very close to the Verdier duality of cellular cosheaves given by J.~Curry in~\cite{Cur_she_cos}. Another corollary of Theorem~\ref{thm:cell_poinca_lef} is a version of Serre duality for flat vector bundles on homology manifolds.

\begin{cor2}
Let $X$ be a homology $n$-manifold in the sense of Definition~\ref{definition:homo_man}, $\F$ be a field, and $E$ be a flat bundle of $\F$-vector spaces of finite rank over $X$. For all integers $k$, we have:
\begin{equation*}
H^k(X;E)\cong \big(H^{n-k}_c\left(X;\partial X; o_\F\otimes_\F E^*\right)\big)^*.
\end{equation*}
\end{cor2}

Following the two first sections, we generalise the tropical version of the Lefschetz Hyperplane Theorem given by C.~Arnal, A.~Renaudineau and K.~Shaw in~\cite[Theorem~1.2.1349]{Arn-Ren-Sha_Lef_sec} and its extension by E.~Brugallé, L.~Lopez de Medrano and J.~Rau in~\cite[Proposition~3.2 p.~15]{Brug-LdM-Rau_Comb_pac}. Let $P$ be a full dimensional polytope of $\R^n$ whose vertices lie in $\Z^n$. Let $K$ be a polyhedral subdivision of $P$ whose vertices also lie in $\Z^n$. In tropical geometry, $K$ needs to be convex i.e. obtained by projection of the bottom faces of a polytope $Q$ living above $P$ in $\R^{n+1}$. The subdivision $K$ is determined by a tropical hypersurface of $P$ and contains all the information we need about this hypersurface. For instance, we can associate a subcomplex $X$ of the barycentric subdivision of $K$ that is isotopic to the tropical hypersurface. This subcomplex is called the \emph{dual hypersurface} of $K$. In this framework, the tropical homology of $X$ and $P$ is defined as the homology of certain dihomologic cosheaves $F^{X}_p\subset F^{P}_p$ on $K$, for all non-negative integers $p$. They were introduced in~\cite{Ite-Kat-Mik-Zha_tro_hom} by I.~Itenberg, Katzarkov, G.~Mikhalkin and I.~Zharkov. To keep the analogy with complex geometry going, they play the role of the sheaves of holomorphic forms on $X$ and $P$ respectively. We note that if we drop the convexity assumption on $K$ the definition of the dual hypersurface $X$, as well as the definition of I.~Itenberg et al.'s cosheaves are still making sense. However, without convexity the pair $(P;X)$ is beyond the scope of tropical geometry. Nonetheless, we will place ourselves here in this level of generality. The tropical version of the Lefschetz Hyperplane Theorem describes the nature of the morphisms induced in homology by the inclusions:
\begin{equation*}
H_{p,q}(X;\Z):=H_q\left(K;F^{X}_p\right)\rightarrow H_q\left(K; F^{P}_p\right)=:H_{p,q}(P;\Z).
\end{equation*}
C.~Arnal, A.~Renaudineau and K.~Shaw showed in~\cite{Arn-Ren-Sha_Lef_sec} that, when the toric variety associated with $P$ is smooth and $K$ is a convex unimodular triangulation\footnote{I.e. $K$ is made of primitive simplices. A simplex is primitive if its vertices belong to the lattice and its volume is minimal among such simplices.}, these morphisms are isomorphisms when $p+q<\dim P -1$ and surjective when $p+q=\dim P -1$. Later, E.~Brugallé, L.~Lopez de~Medrano and J.~Rau showed that this statement remains true when the convexity hypothesis is dropped. Using Theorem~\ref{thm:cell_poinca_lef} we are able to extend the statement to polytopes with orbifold toric varieties modulo a change of coefficients and of cosheaves $F^X_p$. The toric variety associated with $P$ is orbifold when the polytope $P$ is simple\footnote{I.e. every face of codimension $q$ of $P$ is the intersection of exactly $q$ faces of codimension~$1$.}. When $P$ is simple we define an integer $\delta(P)\geq 1$ that gives a ``measure'' of the singularities of the toric variety of $P$. Then, we introduce the \emph{saturated tropical cosheaves} of $X$:
\begin{equation*}
F^{X}_p\subset \Fbar^X_p\subset F^{P}_p,\,p\in\N,
\end{equation*}
for which the following theorem holds.

\begin{thm3}
Let $P$ be an $n$-dimensional integral polytope endowed with an integral polyhedral subdivision $K$, $X$ be the dual hypersurface of $K$, and $R$ be a ring in which $\delta(P)$ is invertible. The homological morphisms:
\begin{equation*}
\ibar_{p,q}\colon \Hbar_{p,q}(X;R):=H_q\left(K;\Fbar^X_p\otimes R\right) \rightarrow H_{p,q}(P; R):= H_q\left(K; F^P_p\otimes R\right)\,,
\end{equation*}
induced by the inclusions $\ibar_p\colon \Fbar_p^{X}\rightarrow F_p^{P}$ are isomorphisms for all $p+q<n-1$, and surjective for all $p+q=n-1$.
\end{thm3}

The number $\delta(P)$ equals $1$ if and only if the toric variety of $P$ is smooth. Moreover, we determine a number $\theta(K)\geq 1$ associated to the subdivision $K$ such that $F^X_p\otimes R$ equals $\Fbar^X_p\otimes R$ for all non-negative integers $p$ if and only if $\theta(K)$ is invertible in the ring $R$. The number $\theta(K)$ equals $1$ on every unimodular triangulations, the usual definition of smoothness for subdivisions, but the converse does not even imply that $K$ is a triangulation.


In addition of Theorem~\ref{thm:Lefschetz_hyp_sec} we recover the formul{\ae} giving the dimensions of the homology groups $H_{p,q}(P;\Q)$, corresponding to the rational Betti numbers of the toric variety of $P$.

\begin{prop1}
Let $P$ be a simple integral polytope and $R$ be a principal ideal domain in which $\delta(P)$ is invertible. For all $p\in\N$, the only non-trivial homology group of the cosheaf $F^{P}_p\otimes R$ is in dimension $p$. Moreover, this $R$-module is free and its rank is given by the following formula:
\begin{equation*}
\rk_R\, H_{p,p}(P;R) = \sum_{k=0}^p(-1)^{p-k} \binom{n-k}{p-k}f_{n-k}(P)\,,
\end{equation*}
where $f_k(P)$ denotes the number of $k$-faces of $P$.
\end{prop1}

\section{CW-Complexes and Cellular Homology}

CW-complexes were introduced by J.~H.~C.~Whitehead in \emph{Combinatorial homotopy.~{I}}, \cite{Whi_com_hom}. Their underlying topological spaces, their \emph{supports}, form a broad family of spaces usually considered well-behaved. Some of the sheaves defined on their support are particularly adapted to their structure. They can be described by a relatively small amount of data and their cohomology can be computed by the techniques of cellular cohomology. In the following paragraphs we give a succinct presentation of the objects at play in this text.

\subsection{CW-Complexes}

\begin{defi}[CW-Complex]
A \emph{CW-complex} $K$ is the data of a Hausdorff topological space $|K|$, called the \emph{support} of $K$, filtered by closed subsets $\varnothing=K^{(-1)}\subset K^{(0)}\subset\dots\subset K^{(k)}\subset\dots\subset |K|$ called the \emph{skeleta} of $K$ whose union covers $|K|$. Such filtration has to satisfy the additional properties:
\begin{enumerate}
\item For every $k\geq 0$ and every connected component $e^k$ of $ K^{(k)}\setminus K^{(k-1)}$, called an \emph{open $k$-cell}, there exists a surjective continuous map from the closed $k$-dimensional ball onto the closure $\bar{e}^k$ that carries homeomorphically the open ball onto $e^k$, such a map is called a \emph{characteristic map} of the open cell~$e^k$;
\item $|K|$ has the weak topology: \emph{a subset $A\subset |K|$ is closed if and only if its intersection~$A\cap\Bar{e}^k$ with every closed cell is closed};
\item Every skeleton $K^{(k)}$ has the weak topology in the same sense as in point 2.
\end{enumerate}
We call the dimension of $K$, $\dim K$, the smallest integer from which the filtration $(K^{(k)})_{k\,\geq\,-1}$ is stationary. It might be $\infty$. A \emph{sub-complex} $L$ of $K$ is determined by a closed subset $|L|$ for which the induced filtration:
\begin{equation*}
\varnothing=L^{(-1)}\subset L^{(0)}\subset\dots\subset L^{(k)}\subset\dots\subset |L| \textnormal{ where }L^{(k)}:=|L|\cap K^{(k)} \textnormal{ for all }k\in \N,
\end{equation*}
turns it into a CW-complex of its own right. The intersection of some sub-complexes is again a sub-complex. For all subsets $A$ of $|K|$ we set $K(A)$ to be the smallest sub-complex containing $A$ in its support i.e. the intersection of all the sub-complexes containing $A$ in their support.
\end{defi}

\begin{defi}[Simplicial Complex]\label{def:simplicial_cpx}
A \emph{simplicial complex} $S$ is a collection~$V$ of \emph{vertices} and a collection $T$ of finite subsets of $V$, called \emph{simplices}, such that:
\begin{equation*}
\forall \sigma\in T,\; \tau\subset\sigma \Rightarrow \tau\in T.
\end{equation*}
The \emph{geometric realisation} of $S$ is the set:
\begin{equation*}
\bigcup_{\sigma\,\in\,T} \convhull(v\in\sigma) \subset \bigoplus_V\R,
\end{equation*}
with the induced product topology.
\end{defi}

\begin{exam}
The most basic examples (of CW-complexes) are given by simplices and all the geometric realisations of simplicial complexes as defined in~\cite{Whi_sim_spa}, c.f. Definition~\ref{def:simplicial_cpx}. More generally, a polyhedral complex is an example of CW-complex. By a polyhedral complex we mean a collection $K$ of polytopes\footnote{A convex hull of a finite number of vertices.} in a real vector space that contains all the faces of its polytopes and in which two distinct polytopes intersect on a common face (which might be empty). In a polyhedral complex the open cell corresponding to a polytope is its \emph{relative interior}, that is to say, the topological interior of the polytope in the affine space it spans.
\end{exam}
\begin{exam}
Some extremely classical examples are given by the real projective spaces. They filter themselves ${\R \P^0\subset\R \P^1\subset\dots \subset \R \P^n}$ by inclusion on the first coordinates. The induced partition of $\R \P^n$ into open cells corresponds to a decomposition into affine spaces, one for every $0\leq k\leq n$. By extension, the inductive limit $\R \P^\infty$ is also a CW-complex for the induced filtration.
\end{exam}

\begin{defi}
A CW-complex in which every point has a neighbourhood that meets only finitely many open cells is called \emph{locally finite}.
\end{defi}

Any \emph{finite} (with finitely many cells) CW-complex is obviously locally finite. Among our examples, $\R \P^\infty$ is not locally finite as the neighbourhood of a point in the open cell $\R^k$ will meet all the open cells $\R^n$ for all $n\geq k$.

\begin{prop}[{\cite[(G), pp.~225-227, (M), pp.~230-231]{Whi_com_hom}}]
A CW-complex is a normal and locally contractible topological space.
\end{prop}

\begin{defi}
A \emph{regular} CW-complex is one that admits for each cell a characteristic map that is a homeomorphism over the entire closed ball.
\end{defi}

In particular, every closed cell of a regular CW-complex is homeomorphic to a closed ball. It excludes the CW-complex structure of the real projective spaces (apart from the trivial case $\R \P^0$) given by affine spaces as every positive dimensional closed cell is a projective space, different from a closed ball. An important example of regular CW-complex is given by geometric realisations of simplicial complexes in which every closed cell is a closed simplex, hence topologically a closed ball. Likewise, a polyhedral complex is necessarily a regular CW-complex.


Given a general CW-complex, the formula $e_1\leq e_2\iff \Bar{e}_1\subset\Bar{e}_2$ defines an order on its cells. When the CW-complex is regular this order shares the same properties as the inclusion of faces in a simplicial complex.

\begin{lemm}
For any two open cells $e_1,e_2$ of $K$, a regular CW-complex, $e_1$ meets the closure of $e_2$ if and only if it is fully contained in it:
\begin{equation*}
e_1\cap\Bar{e}_2\neq \varnothing \, \iff\, e_1\subset \Bar{e}_2.
\end{equation*}
\end{lemm}

One can find a proof in~\cite[pp.~229-230, R.R.1]{Coo-Fin_hom_cel}.
Therefore, in a regular CW-complex, whenever a cell $e_1$ meets the closure of another one $e_2$ we have $e_1\leq e_2$ and we say that $e_1$ is a \emph{face} of $e_2$. If $e_1$ is distinct from $e_2$ we say that $e_1$ is a \emph{proper face} of $e_2$ and denote it by $e_1<e_2$. Furthermore, if $e_1\leq e_2$ or $e_2\leq e_1$ we say that $e_1$ and $e_2$ are \emph{adjacent}.

\begin{lemm}
Let $K$ be a regular CW-complex. For all cells $e$ of $K$, the support of the sub-complex $K(e)$ is the closure $\bar{e}$.
\end{lemm}

\begin{lemm}\label{lemm:2_faces}
Let $k\in\N$ and $e^{k+2}$ be an open cell of a regular CW-complex $K$. For all faces of codimension~2, $e^k$ of $e^{k+2}$ there are exactly two cells of codimension~1 between $e^k$ and $e^{k+2}$:
\begin{equation*}
\card \left\{e^{k+1}\,\middle|\,e^k<e^{k+1}<e^{k+2}\right\}=2.
\end{equation*}
\end{lemm}

\begin{lemm}[Open Star]
Let $e$ be a cell of a regular CW-complex $K$. The union of all the cells having $e$ as a face, called the \emph{open star} of $e$, is an open subset of $|K|$.
\end{lemm}

Proofs of these statements are given in~\cite[Proposition~1.6, p.~30, Theorem~4.2, pp.~231-232, Lemma~4.1, p.~230]{Coo-Fin_hom_cel}.
In a locally finite CW-complex $K$, the open star of a cell is a finite union of cells, so its closure, the \emph{closed star} of the cell, is a finite sub-complex of $K$. The collection $K-e$ of all the cells whose closure avoids $e$ is the complement of the open star of $e$ and is a sub-complex of $K$. Its underlying topological space is a deformation retract of $|K|\setminus e$ and is the largest sub-complex of $K$ contained in the complement $|K|\setminus e$. For all subsets $A$ of $|K|$, we denote by $K-A$ the largest sub-complex of $K$ contained in $|K|\setminus A$.

\begin{defi}[Subdivisions]
A \emph{subdivision} $K'$ of a CW-complex $K$ is a CW-complex on the same support $|K'|=|K|$ in which every cell $e'\in K'$ is contained in a cell $e\in K$. Another way of saying it is that the partition of $|K|$ into open cells of $K'$ is finer than the partition given by the open cells of $K$.
\end{defi}

\begin{figure}[h]
\centering
\begin{tikzpicture}[scale=2]
\tikzfading[name=fade inside, inner color=transparent!100, outer color=transparent!50]
\coordinate (a) at (0,0);
\coordinate (b) at (2,0);
\coordinate (c) at ($(a)+(60:2)$);
\coordinate (a1) at ($(a)+(4,0)$);
\coordinate (b1) at ($(b)+(4,0)$);
\coordinate (c1) at ($(c)+(4,0)$);
\coordinate (center) at ($(a)!0.33333333!(a1)+(a)!0.33333333!(b1)+(a)!0.33333333!(c1)$);
\fill[pattern={Lines[angle=90, line width=.5pt]}] (a) -- (b)-- (c) -- cycle;
\draw[very thick](a) -- (b)-- (c) -- cycle;
\draw[very thick] (a1) -- (b1)-- (c1) -- cycle;
\filldraw[pattern={Lines[angle=90, line width=.5pt]},very thick] (a1) -- (center) -- ($(c1)!(b1)!(a1)$);
\filldraw[pattern={Lines[angle=120, yshift=-3.75pt, line width=.5pt]},very thick] (a1) -- (center) -- ($(b1)!(c1)!(a1)$);
\filldraw[pattern={Lines[angle=60, yshift=0.75pt, line width=.5pt]},very thick] (b1) -- (center) -- ($(b1)!(c1)!(a1)$);
\filldraw[pattern={Lines[angle=90, line width=.5pt]},very thick] (b1) -- (center) -- ($(b1)!(a1)!(c1)$);
\filldraw[pattern={Lines[angle=120, yshift=5pt, line width=.5pt]},very thick] (c1) -- (center) -- ($(c1)!(a1)!(b1)$);
\filldraw[pattern={Lines[angle=60, yshift=4pt, line width=.5pt]},very thick] (c1) -- (center) -- ($(c1)!(b1)!(a1)$);
\fill (a) circle (.05);
\fill (a1) circle (.05);
\fill (b) circle (.05);
\fill (b1) circle (.05);
\fill (c) circle (.05);
\fill (c1) circle (.05);
\fill ($(a1)!.5!(b1)$) circle (.05);
\fill ($(c1)!.5!(b1)$) circle (.05);
\fill ($(a1)!.5!(c1)$) circle (.05);
\fill (center) circle (0.05);
\draw[->,ultra thick] ($(b)+(0,1)$) -- ($(a1)+(0,1)$);
\end{tikzpicture}
\caption{The barycentric subdivision of the triangle.}\label{fig:bar_trig}
\end{figure}
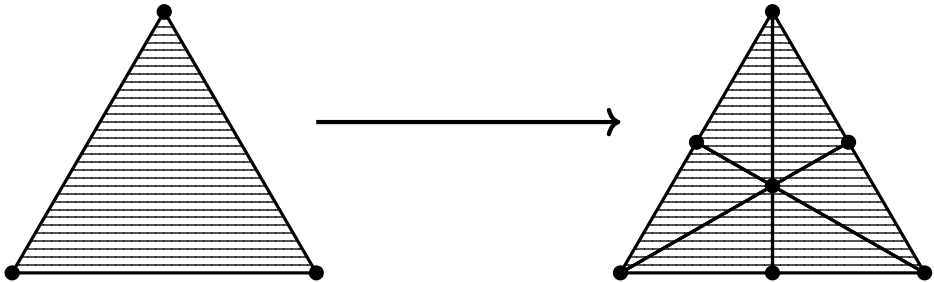

A common example of subdivision is given by the \emph{barycentric subdivision} $\Sd S$ of a simplicial complex $S$, see for instance Figure~\ref{fig:bar_trig}. It is described abstractly as follows: the vertices of $\Sd S$ are given by the simplices of $S$ and the simplices of $\Sd S$ by the flags of simplices of $S$. More concretely, a collection of simplices $\{\sigma_0,\dots,\sigma_k\}$ is a simplex of $\Sd S$ if and only if it can be totally ordered i.e. there is a permutation $\pi$ of the indices $\{0,\dots,k\}$ such that:
\begin{equation*}
\sigma_{\pi(0)} <\dots < \sigma_{\pi(n)}.
\end{equation*}
One can define a homeomorphism between the geometric realisation of $\Sd S$ and the geometric realisation of $S$ by sending each vertex of $\Sd S$ to the barycenter of its corresponding simplex in $S$ and extending this map by linearity. An example is depicted in Figure~\ref{fig:homeo_segment}.
\begin{figure}[h]
\centering
\begin{tikzpicture}[scale=2]
\draw[very thick] (0,0) -- (1,0);
\fill (0,0) circle (0.05);
\fill (1,0) circle (0.05);
\draw[->,ultra thick] (1.25,0) -- (1.75,0);
\draw[dashed] (2,-.5) -- (3.5,-.5);
\draw[dashed] (2.25,-.75) -- (2.25,.75);
\fill (2.25,-.5) circle (0.05);
\fill (3.25,-.5) circle (0.05);
\fill (2.25,.5) circle (0.05);
\draw[very thick] (2.25,-.5) -- (2.25,.5) -- (3.25,-.5);
\draw[thick,->] ($(2.25,.5)!.1!(2.75,-.5)$) -- ($(2.25,.5)!.95!(2.75,-.5)$);
\draw[->,ultra thick] (3.75,0) -- (4.25,0);
\draw[very thick] (4.5,0) -- (5.5,0);
\fill (4.5,0) circle (0.05);
\fill (5,0) circle (0.05);
\fill (5.5,0) circle (0.05);
\end{tikzpicture}
\caption{The homeomorphism from the barycentric subdivision of the segment to the initial segment.}\label{fig:homeo_segment}
\end{figure}
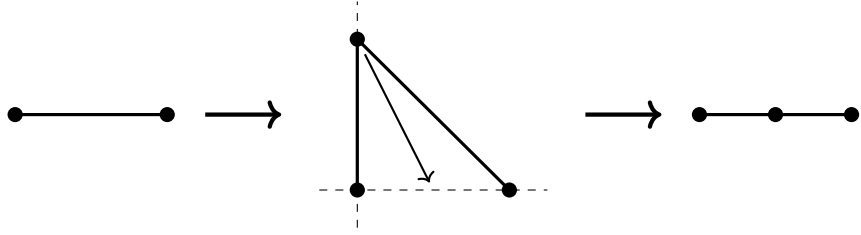

The image of the skeletal filtration of $\Sd S$ under this homeomorphism defines a subdivision of the CW-complex induced by $S$ in the sense of the previous definition. Note that if instead we chose to send each vertex of $\Sd S$ to an arbitrary point in the open cell defined by the corresponding simplex in $S$ (and then extending the map by linearity) we would also have defined a subdivision of $S$ that is equivalent, in a combinatorial way, to the previous one. We could say that the barycentric subdivision is only defined unequivocally on the abstract level. The same abstract procedure can be performed with a regular CW-complex $K$. The \emph{barycentric subdivision} $\Sd K$ of $K$ is defined to be the following simplicial complex:
\begin{enumerate}
\item Every cell of $K$ corresponds to a vertex of $\Sd K$;
\item A finite set of cells of $K$ corresponds to a simplex of $\Sd K$ if and only if it is totally ordered by adjacency.
\end{enumerate}

\begin{prop}\label{prop:bar_sub_cw}
Let $K$ be a regular CW-complex. The geometric realisation of $\Sd K$ is homeomorphic to $|K|$ in such a way that $\Sd K$ can be seen as a subdivision of $K$.
\end{prop}

This proposition, proven in~\cite[Theorem~1.7, pp.~80-81.]{Lun-Wei_top_cw},
allows us to see a regular CW-complex as ``a simplicial complex in which the simplexes are more efficiently combined into closed cells.''\footnote{Ibid. p.~77.} Another feature of simplicial complexes shared by regular complexes is the following:

\begin{prop}
In a regular CW-complex the open stars of cells are contractible.
\end{prop}

\begin{proof}
The geometric realisation of $\Sd K$ lives in the real vector space $V$ spanned by the cells of $K$. We use the same symbol to denote a cell $e^k$ and its associated generator in $V$. Hence, an element of $V$ is a formal finite linear combination of the cells of $K$. We endow this vector space with the norm $1$:
\begin{equation*}
\left\|\sum_{e\,\in\,K} x_ee\right\|_1=\sum_{e\,\in\,K}|x_e|.
\end{equation*}
The geometric realisation $|\Sd K|$ is the union of the convex hulls of the sets of cells $\{e^{k_0},\dots,e^{k_n}\}$ corresponding to barycentric simplices i.e. flags of cells. It is a subset of the intersection of the unit sphere with the positive ortant $V_+:=\{\sum_{e\,\in \,K}x_ee\in V\;|\;x_e\geq 0\}$. For a flag of cells $e^{k_0}<\dots<e^{k_n}$ let us denote here the corresponding open simplex by:
\begin{equation*}
\big(e^{k_0};\dots\,;e^{k_n}\big):=\left\{\sum_{i=0}^n x_ie^{k_i}\;\middle|\;\forall i,\, x_i>0\text{ and }\sum_{i=0}^n x_i=1\right\}.
\end{equation*}
An open cell $e^k$ of $K$ corresponds under the homeomorphism of Proposition~\ref{prop:bar_sub_cw} to the union of the open barycentric simplices $(e^{k_0};\dots\,;e^{k_n})$ for which $e^{k_n}=e^k$. Therefore, the open star $S$ of $e^k$ is in this context:
\begin{equation*}
S=\bigcup_{\substack{e^{k_0}\,<\,\dots\,<\,e^{k_n}\\e^k\,\leq\,e^{k_n}}}\big(e^{k_0};\dots\,;e^{k_n}\big).
\end{equation*}
We consider the family of bounded linear operators $(\Phi_t:V\rightarrow V)_{0\,\leq\,1\,\leq\,t}$ defined by $\Phi_t=(\id-\pi)+t\pi$, where $\pi$ denotes the projection to the sub-space spanned by the cells that do not contain $e^k$ parallely to the sub-space spanned by those that contain it. Let $U$ be the open set of $V_+$ of vectors that have at least one positive coordinate indexed by a cell that contains $e^k$. We have $S=|\Sd K|\cap U$. For all $t\in[0;1]$, the map:
\begin{equation*}
\begin{array}{rcl}
\Psi:[0;1]\times U& \longrightarrow & U \\%
u & \longmapsto & \frac{\|u\|_1}{\|\Phi_t(u)\|_1}\Phi_t(u)\,,
\end{array}
\end{equation*}
is continuous and every partial map $\Psi(t;-)$ stabilises $S$. The image $\Psi(0;S)$ is the union of the open barycentric simplices $(e^{k_0};\dots\,;e^{k_n})$ for which $e^k\leq e^{k_0}$. Note that the restriction of every map $\Psi(t;-)$ is constant on this set. Therefore $\Psi(0;S)$ is a deformation retract of $S$. Now this set retracts on the barycentre of $e^k$ by simple convex interpolation $(t;u)\in [0;1]\times\Psi(0;S)\mapsto (1-t)u+te^k$, thus $S$ is contractible.
\end{proof}

When we consider the barycentric subdivision of a finite regular CW-complex we increase considerably the number of cells. There is, however, a less expensive procedure that builds a ``pseudo-subdivision'' for any regular CW-complex $K$ that sits between $K$ and $\Sd K$. By pseudo-subdivision we mean a certain recombination of the barycentric simplices that looks a lot like a regular subdivision.

\begin{defi}[Dihomologic Pseudo-subdivision]
Let $K$ be a regular CW-complex. Let $e^p\leq e^q$ be a pair of adjacent cells, we define its associated \emph{dihomologic pseudo-cell} to be the union of the open barycentric simplices\footnote{The open simplex on a vertex set $v_0,\dots,v_n$ is the set
\[
\left\{\sum_{i=0}^nt_iv_i\;\middle|\; \forall i,\,t_i>0\text{ and }\sum_{i=0}^nt_i=1\right\}.
\]
} associated with the flags $e^{k_1}<\dots<e^{k_n}$ for which $e^p= e^{k_1}$ and $e^{k_n}=e^q$. These ``open'' pseudo-cells form a partition of the simplicial complex $\Sd K$. We call such partition the \emph{dihomologic pseudo-subdivision} of $K$. We say that the dihomologic pseudo-cell associated with the pair $e^p\leq e^q$ has dimension $q-p$. Figure~\ref{fig:dih_sub_disc} illustrates this procedure on a disc.
\end{defi}

\begin{figure}[h]
\centering
\begin{tikzpicture}[scale=2]
\filldraw[pattern={Lines[angle=90, line width=.5pt]},very thick] (0,0) circle (1);
\fill (-1,0) circle (.05);
\fill (1,0) circle (.05);
\fill (3,0) circle (.05);
\fill (5,0) circle (.05);
\fill (4,1) circle (.05);
\fill (4,-1) circle (.05);
\fill (4,0) circle (.05);
\filldraw[pattern={Lines[angle=90, line width=.5pt, xshift=1pt]},very thick] (4,-1) -- (4,1) arc [start angle=90, end angle=270, radius=1] (4,-1) ;
\filldraw[pattern={Lines[angle=0, line width=.5pt]},very thick] (4,1) -- (4,-1) arc [start angle=270, end angle=450, radius=1] (4,1) ;
\draw[->,ultra thick] (1.25,0) -- (2.75,0);
\end{tikzpicture}
\caption{The dihomologic pseudo-subdivision of a regular CW-complex structure of the disc.}\label{fig:dih_sub_disc}
\end{figure}
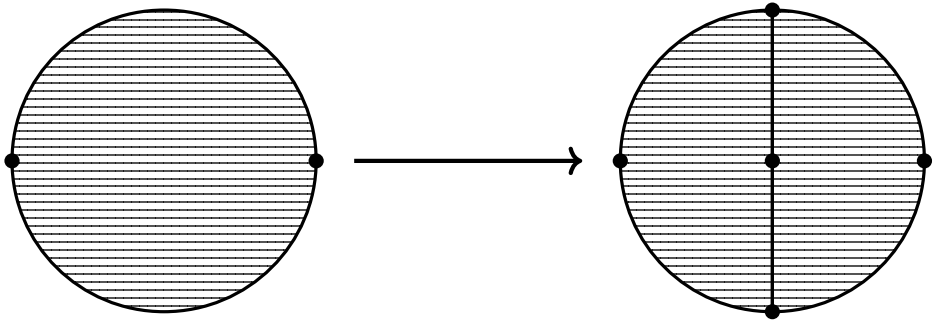

As every open dihomologic pseudo-cell is a union of open barycentric simplices their closures are naturally supports of sub-complexes of the barycentric subdivision. We always endow the closed dihomologic pseudo-cell with this particular simplicial structure. The dihomologic pseudo-subdivision shares many properties with a regular subdivision of $K$ but may fail to define a CW-complex structure on $|K|$. We do not know if closures of dihomologic pseudo-cells are always homeomorphic to closed balls. However, for a broad variety of examples as we will see, this is indeed the case and the dihomologic pseudo-subdivision is a regular subdivision of $K$.

\begin{rema}
We chose to use the terminology ``dihomologic'' because Zeeman's dihomology bicomplex, c.f. \cite{Zee_dih_i,Zee_dih_ii,Zee_dih_iii}, would be the cellular chain complex of this pseudo-subdivision. This bicomplex was latter (dually) rediscovered by Forman in~\cite{For_comb_nov} under the name of ``combinatorial differential forms''.
\end{rema}

\begin{defi}[Collapse]
Let $S=(V;T)$ be a simplicial complex.~A \emph{free simplex} of $S$ is a simplex $\sigma\in T$ strictly contained in exactly one maximal simplex of $S$. An \emph{elementary collapse} of $S$ is the operation:
\begin{equation*}
S\mapsto \big(V;T\setminus\{\tau\,:\,\sigma\subset \tau\}\big),
\end{equation*}
where $\sigma$ is a free simplex of $S$. A \emph{collapse} is a composition of elementary collapses. A complex is called \emph{collapsible} if it can be collapsed to a single vertex.
\end{defi}

\begin{prop}
Let $e^p\leq e^q$ be a pair of adjacent cells of $K$. The closure of the pseudo-cell associated with $e^p\leq e^q$, is a collapsible simplicial complex of pure dimension $q-p$. Moreover, any codimension~1 simplex of the open pseudo-cell is exactly contained in two of its maximal simplices.
\end{prop}

\begin{proof}
Let $e^p\leq e^q$ be a pair of adjacent cells. If $e^p=e^q$ the associated pseudo-cell is a single vertex and the statement of the proposition is true. Now suppose $e^p$ is a proper face of $e^q$. The closure of the associated pseudo-cell is the union of the barycentric simplices whose flags $e^{k_1}<\dots<e^{k_n}$ satisfy $e^p\leq e^{k_1}<\dots<e^{k_n}\leq e^q $. This is a simplicial sub-complex of $\Sd K$. A maximal simplex of this closed pseudo-cell is given by a maximal flag of adjacent cells of $K$ starting with $e^p$ and ending with~$e^q$. Since $K$ is regular such flag has necessarily length $q-p+1$, so the corresponding simplex has dimension $q-p$. Hence, the closed pseudo-cell is a pure $(q-p)$-dimensional simplicial complex. If $q-p=1$, this closed pseudo-cell corresponds to the closed barycentric edge $e^p<e^q$. This is a collapsible complex. If $q-p>1$, the maximal simplices of the closed pseudo-cell correspond to flags $e^p<e^{p+1}<\dots<e^{q-1}<e^q$. They all have a codimension~2 free face, associated with $e^{p+1}<\dots<e^{q-1}$. If we collapse all such simplices we obtain the union of the closed $(q-p-1)$-simplices associated with the flags of the form $e^p<e^{k_1}<\dots<e^{k_{q-p-2}}<e^q$. All such simplices are now maximal and contain each a free face of codimension~2, namely $e^{k_1}<\dots<e^{k_{q-p-2}}$. We can recursively perform such collapses to end up with the barycentric edge $e^p<e^q$. Thus, the closed pseudo-cell of the pair $e^p\leq e^q$ is collapsible. For the final part, a simplex of codimension~1 of the open pseudo-cell is given by a flag $e^p<e^{k_1}<\dots<e^{k_{q-p-2}}<e^q$ of length $q-p$ and has the form $e^p<\dots<e^i<e^{i+2}<\dots<e^q$. Since $K$ is regular there are exactly two $(i+1)$-cells between $e^i$ and $e^{i+1}$, hence two $(q-p)$-simplices.
\end{proof}

\begin{prop}
The open pseudo-cell associated with the pair $e^p\leq e^q$ meets the closed pseudo-cell indexed by the pair $e^{p'}\leq e^{q'}$ if and only if $e^{p'}\leq e^{p}$ and $e^q\leq e^{q'}$. This means that every open barycentric simplex of the former is contained in the latter. Moreover, if $\epsilon_0$ is an open pseudo-cell of dimension $k$ included in a closed pseudo-cell $\bar{\epsilon}_2$ of dimension $k+2$ then there are exactly two open pseudo-cells $\epsilon_1$ of dimension $k+1$ such that $\epsilon_0\subset\bar{\epsilon}_1$ and $\epsilon_1\subset\bar{\epsilon}_2$.
\end{prop}

\begin{proof}
The first part is a consequence of the definition. For the second part, we choose $\epsilon_0$ corresponding to a pair $e^{p+a}\leq e^{p+a+k}$ and $\epsilon_2$ to a pair $e^p\leq e^{p+a+k+b}$ satisfying $e^p\leq e^{p+a}\leq e^{p+a+k}\leq e^{p+a+k+b}$. By assumption, we have $a+b=2$ and three different cases can occur: one of the two numbers $a,b$ is $2$ and the other $0$ or both equal $1$. The two first cases are symmetric. If it's $a$ that equals $2$ we have $e^p\leq e^{p+2}\leq e^{p+2+k}\leq e^{p+2+k}$ and the two $(k+1)$-pseudo-cells between $\epsilon_0$ and $\epsilon_1$ are those associated with the two pairs $e^{p+1}\leq e^{p+2+k}$ with $e^p\leq e^{p+1}\leq e^{p+2}$. If both $a$ and $b$ equal $1$ then we have $e^p\leq e^{p+1}\leq e^{p+1+k}\leq e^{p+2+k}$ and the two $(k+1)$-pseudo-cells between $\epsilon_0$ and $\epsilon_2$ are those associated with the two pairs $e^p\leq e^{p+1+k}$ and $e^{p+1}\leq e^{p+2+k}$.
\end{proof}

In the light of this property it makes sense to talk about \emph{adjacent} pseudo-cells as we do for cells of regular CW-complexes. As in the regular case, if $\epsilon$ and $\epsilon'$ are adjacent pseudo-cells with $\epsilon\subset\overline{\epsilon}'$ we say that $\epsilon$ is a \emph{face} of $\epsilon'$. If in addition $\epsilon\neq\epsilon'$ we say that $\epsilon$ is a \emph{proper face} of $\epsilon'$. Also we see here that two dihomologic pseudo-cells meet on a common face if their intersection is non-empty. Note that this property might not be true in $K$, two closed cells meet on the union of their common faces if the intersection is non-empty. Let $e^p< e^q$ be a proper adjacent pair of cells of $K$ and $\epsilon$ denote its associated pseudo-cell. The simplicial complex supported on the closure, $(\Sd K)(\epsilon)$, is the join of the barycentric edge $e^p<e^q$ and a sub-complex $A(e^p;e^q)\subset\Sd K$. This sub-complex is the collection of all the barycentric simplices indexed by the flags $e^{k_1}<\dots<e^{k_{n}}$ for which $e^p<e^{k_1}$ and $e^{k_{n}}<e^q$.

\begin{defi}[Homology Manifold]\label{definition:homo_man}
A \emph{homology manifold} of dimension $n\in\N$ is the support $X$ of a regular, finite dimensional, locally finite CW-complex for which the graded local homology group $H_*(X;X\setminus\{x\};\Z)$ of every point $x\in X$ is isomorphic to either $H_*(\R^n;\R^n\setminus \{0\};\Z)$ or $0$. The \emph{boundary} of $X$, denoted by $\partial X$, is the set of points $x\in X$ for which $H_*(X;X\setminus\{x\};\Z)=0$.
\end{defi}

\begin{prop}\label{prp:A_homo}
The support of $A(e^p;e^q)$ is a connected homology manifold of dimension $(q-p-2)$ whose $(p+1)$-fold suspension is homeomorphic to a $(q-1)$-sphere.
\end{prop}

\begin{proof}
Let $B$ denote the simplicial complex $(\Sd K)(\bar{e}^q\setminus e^q)$ and $e^0<\dots<e^p$ be a complete flag of cells of $K$. The simplicial complex $A(e^p;e^q)$ is the link in $B$ of the barycentric simplex associated with $e^0<\dots<e^p$. $B$ is a simplicially triangulated $(q-1)$-sphere thus in application of~\cite[Proposition~1.3 p.~5]{Gal-Ste_Cla_Sim} the $(p+1)$-fold suspension of $A(e^p;e^q)$ is homeomorphic to a $(q-1)$-sphere. The Mayer-Vietoris long exact sequence in singular homology implies that, if $\Sigma A$ is the suspension of a topological space $A$, we have $H_0(\Sigma A;\Z)=\Z$, $H_{k}(\Sigma A;\Z)\cong H_{k-1}(A;\Z) \text{ for all }k\geq 2$ and the exact sequence:
\begin{equation*}
0 \rightarrow H_1(\Sigma A;\Z) \rightarrow H_0(A;\Z) \rightarrow \Z \rightarrow 0\;.
\end{equation*}
So $A$ has the homology of a $k$-sphere if and only if the $l$-fold suspension $\Sigma^l A$ has the homology of an $(l+k)$-sphere. Therefore, the simplicial complex $A(e^p;e^q)$ has the integral homology of a $(q-p-2)$-sphere. For the remaining part of the proposition we note that if the barycentric $n$-simplex with indexing flag $e^{k_0}<\dots<e^{k_{n}}$ belongs to $A(e^p;e^q)$ then its link $L$ in this complex is the join:
\begin{equation*}
A\left(e^p;e^{k_0}\right)*A\left(e^{k_0};e^{k_1}\right)*\dots*A\left(e^{k_{n-1}};e^{k_n}\right)*A\left(e^{k_n};e^{q}\right).
\end{equation*}
Hence, its $(p+n+2)$-fold suspension is homeomorphic to a $(q-1)$-sphere and $L$ has the integral homology of a $(q-p-n-3)$-sphere. Now $A(e^p;e^q)$ is a simplicial complex of pure dimension $(q-p-2)$ in which the link of every $n$-dimensional simplex has the homology of a $(q-p-n-3)$-sphere. This is a homology manifold of dimension $(q-p-2)$. Indeed, if $x$ is a point of $|A(e^p;e^q)|$ that belongs to the relative interior of the barycentric $n$-simplex $\sigma$, then, by excision, $H_k(|A(e^p;e^q)|;|A(e^p;e^q)|\setminus\{x\};\Z)$ equals $H_k(|S|;|S|\setminus\{x\};\Z)$ for all $k$ where $S$ is the closed star of $\sigma$. Note that $|S|$ is contractible and $|S|\setminus\{x\}$ is non-empty. Hence $H_0(|S|;|S|\setminus\{x\};\Z)$ vanishes, $H_k(|S|;|S|\setminus\{x\};\Z)$ equals $H_{k-1}(|S|\setminus\{x\};\Z)$ for all $k\geq2$, and:
\begin{equation*}
0 \rightarrow H_1(|S|;|S|\setminus\{x\};\Z) \rightarrow H_0(|S|\setminus\{x\};\Z) \rightarrow \Z \rightarrow 0\;.
\end{equation*}
Since $|S|$ is homeomorphic to the topological join $\sigma * |L|$ where $L$ is the link of $\sigma$, $|S|\setminus\{x\}$ is homotopic to the $n$-fold suspension of $|L|$. By assumption $|L|$ has the homology of a $(q-p-n-3)$-sphere so $|S|\setminus\{x\}$ has the homology of a $(q-p-3)$-sphere. Combining this with the relation observed by the relative homology of $(|S|;|S|\setminus\{x\})$ with the homology of $|S|\setminus\{x\}$ we find that $H_k(|A(e^p;e^q)|;|A(e^p;e^q)|\setminus\{x\};\Z)$ is isomorphic to $H_k(\R^{q-p-2};\R^{q-p-2}\setminus\{0\};\Z)$ for all $k$, and that $A(e^p;e^q)$ is a compact $(q-p-2)$-homology manifold without boundary.
\end{proof}

As a direct consequence we get that:

\begin{prop}\label{prop:initial_cells}
The closed pseudo-cells associated with the adjacent pairs of the form $e^0\leq e^p$ are homeomorphic to closed balls.
\end{prop}

\begin{proof}
It follows from the previous observation that the support of this closed pseudo-cell is homeomorphic to the topological join $[0;1]*|A(e^0;e^p)|$ which is the cone over the suspension of $|A(e^0;e^p)|$. By the last proposition, this suspension is a $(p-1)$-sphere so the closed pseudo-cell is actually a closed ball.
\end{proof}

Let $\epsilon$ be a dihomologic pseudo-cell associated with a pair $e^p\leq e^q$, its ``boundary'' $\bar{\epsilon}\setminus \epsilon$ is the support of the simplicial join of the union of the barycenters of $e^p$ and $e^q$ with $A(e^p;e^q)$ (so the suspension of $A(e^p;e^q)$). From that description, we see that it is the union of dihomologic pseudo-cells that are faces of $\epsilon$.


We will see further that the dihomologic pseudo-subdivision also shares a lot of homological features with a regular subdivision. Now we state a condition that ensures the regularity of this pseudo-complex.

\begin{prop}\label{prop:dih_reg_shel}
If $K$ is not only regular but also satisfies that the induced CW-complex on every closed cell $K(e)$ is shellable in the sense of~\cite{Bjo_pos_reg} then every closed dihomologic pseudo-cell is also shellable. As a consequence, the geometric realisation of every closed pseudo-cell is actually homeomorphic to a closed ball making the dihomologic pseudo-subdivision a regular subdivision of $K$.
\end{prop}

\begin{proof}
Let $e^q$ be a cell of $K$. From the first part of~\cite[Proposition~4.4 p.~12]{Bjo_pos_reg}
we know that the barycentric subdivision of $K(e^q)$ is a shellable simplicial complex. Let $e^p$ be a face of $e^q$, we expressed the associated closed dihomologic pseudo-cell $\bar{\epsilon}$ as the simplicial join of a closed interval and the simplicial complex $A(e^p;e^q)$. As in the proof of Proposition~\ref{prp:A_homo}, we can write this complex as the link of a barycentric simplex $\sigma$ with indexing flag $e^0<e^1<\dots<e^p<e^q$ in $\Sd(K(e^q))$. The link of $\sigma$ is shellable by~\cite[Lemma~8.7 p.~237]{Zie_lec_pol}.
 The complex $A(e^p;e^q)$ is a pure $(q-p-2)$-dimensional shellable simplicial complex in which every codimension~1 simplex belongs to exactly two maximal simplices, hence it is homeomorphic to a sphere by~\cite[Proposition~4.3 p.~12]{Bjo_pos_reg}.
Finally, the closed dihomologic pseudo-cell $\bar{\epsilon}$ is the support of a shellable simplicial complex homeomorphic to a closed ball.
\end{proof}

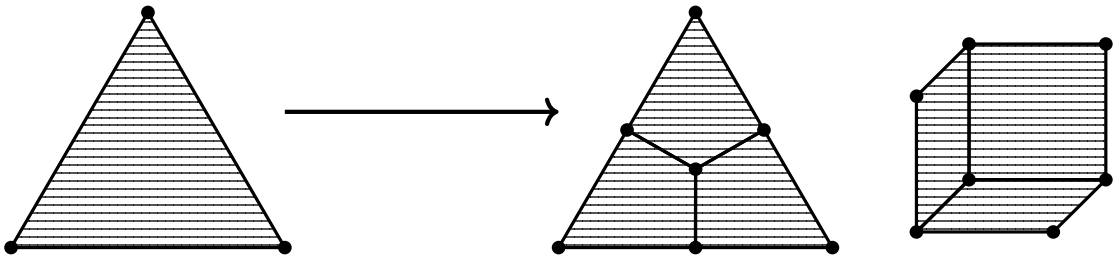
\begin{figure}[h]
\centering
\begin{tikzpicture}[scale=1.8]
\coordinate (a) at (0,0);
\coordinate (b) at (2,0);
\coordinate (c) at ($(a)+(60:2)$);
\coordinate (a1) at ($(a)+(4,0)$);
\coordinate (b1) at ($(b)+(4,0)$);
\coordinate (c1) at ($(c)+(4,0)$);
\coordinate (center) at ($(a)!0.33333333!(a1)+(a)!0.33333333!(b1)+(a)!0.33333333!(c1)$);
\filldraw[pattern={Lines[angle=90, line width=.5pt]},very thick] (a) -- (b)-- (c) -- cycle;
\filldraw[pattern={Lines[angle=120, yshift=-3.75pt, line width=.5pt]},very thick] (a1) -- ($(a1)!(b1)!(c1)$) -- (center) -- ($(b1)!(c1)!(a1)$)-- cycle;
\filldraw[pattern={Lines[angle=60, line width=.5pt]},very thick] (b1) -- ($(b1)!(c1)!(a1)$) -- (center) -- ($(b1)!(a1)!(c1)$) -- cycle;
\filldraw[pattern={Lines[angle=0, yshift=4pt, line width=.5pt]},very thick] (c1) -- ($(b1)!(a1)!(c1)$) -- (center) -- ($(c1)!(b1)!(a1)$) -- cycle;
\fill (a) circle (.05);
\fill (a1) circle (.05);
\fill (b) circle (.05);
\fill (b1) circle (.05);
\fill (c) circle (.05);
\fill (c1) circle (.05);
\fill ($(a1)!.5!(b1)$) circle (.05);
\fill ($(c1)!.5!(b1)$) circle (.05);
\fill ($(a1)!.5!(c1)$) circle (.05);
\fill (center) circle (0.05);
\draw[->,ultra thick] ($(b)+(0,1)$) -- ($(a1)+(0,1)$);

\begin{scope}[canvas is xy plane at z=0]
\filldraw[pattern={Lines[angle=150, yshift=4pt, line width=.5pt]},very thick] (7,.5) -- (8,.5) -- (8,1.5) -- (7,1.5) -- cycle;
\fill (7,.5) circle (.05);
\fill (8,.5) circle (.05);
\fill (8,1.5) circle (.05);
\fill (7,1.5) circle (.05);
\end{scope}

\begin{scope}[canvas is xz plane at y=0.5]
\filldraw[pattern={Lines[angle=22.5, yshift=4pt, line width=.5pt]},very thick] (7,0) -- (8,0) -- (8,1) -- (7,1) -- cycle;
\coordinate (w) at (8,1);
\coordinate (x) at (7,1);
\end{scope}

\fill (w) circle (.05);
\fill (x) circle (.05);

\begin{scope}[canvas is yz plane at x=7]
\filldraw[pattern={Lines[angle=60, yshift=4pt, line width=.5pt]},very thick] (0.5,0) -- (1.5,0) -- (1.5,1) -- (0.5,1) -- cycle;
\coordinate (y) at (1.5,1);
\end{scope}

\fill (y) circle (.05);
\end{tikzpicture}
\caption{The dihomologic subdivision of the triangle.}\label{fig:cubical_triang}
\end{figure}

When $K$ is a polyhedral complex, the theorem of Bruggesser and P.~Mani~\cite[Corollary p.~203]{Bru-Man_she_dec}
ensures that it satisfies the hypotheses of the last proposition and the dihomologic pseudo-subdivision is an actual regular subdivision. This is especially the case when $K$ is a simplicial complex. In this particular case, the associated dihomologic subdivision even has the structure of a cubical complex, c.f. Figure~\ref{fig:cubical_triang}. It comes from the following triangulation of the cube $[0;1]^n$: order its vertex set $\{0;1\}^n$ with the product order\footnote{$(x_i)_{1\,\leq\,i\,\leq \,n}\leq(y_i)_{1\,\leq\,i\,\leq\,n}$ if and only if $x_i\leq y_i$, for all $1\leq i\leq n$.} and consider the convex hulls of the flags of such vertices as the simplices of the triangulation. The triangulation of the 3-dimensional cube is illustrated in Figure~\ref{fig:sub_cube}.

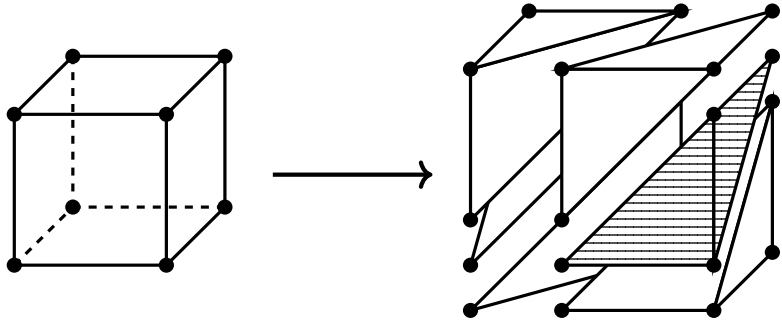
\begin{figure}[h]
\centering
\begin{tikzpicture}[scale=2]
\fill[white] (0,0,0) -- (1,0,-1) -- (1,1,-1) -- cycle ;
\draw[draw=black, very thick] (0,0,0) -- (1,0,-1) -- (1,1,-1) -- cycle ;
\fill[white] (0.6,0,0) -- (1.6,0,0) -- (1.6,1,-1) -- cycle;
\draw[draw=black, very thick] (0.6,0,0) -- (1.6,0,0) -- (1.6,1,-1) -- cycle;
\fill[white] (1.6,0,0) -- (1.6,1,-1) -- (1.6,0,-1) -- cycle;
\draw[draw=black, very thick] (1.6,0,0) -- (1.6,1,-1) -- (1.6,0,-1) -- cycle;
\fill[white] (0,0.3,0) -- (0,1.3,-1) -- (1,1.3,-1) -- cycle ;
\draw[draw=black, very thick] (0,0.3,0) -- (0,1.3,-1) -- (1,1.3,-1) -- cycle ;
\fill[white] (0.6,0.3,0) -- (1.6,0.3,0) -- (1.6,1.3,0) -- cycle;
\filldraw[pattern={Lines[angle=22.5, yshift=4pt, line width=.5pt]}, draw=black, very thick] (0.6,0.3,0) -- (1.6,0.3,0) -- (1.6,1.3,0) -- cycle;
\fill[white] (1.6,0.3,0) -- (1.6,1.3,-1) -- (1.6,1.3,0) -- cycle;
\filldraw[pattern={Lines[angle=60, yshift=4pt, line width=.5pt]}, draw=black, very thick] (1.6,0.3,0) -- (1.6,1.3,-1) -- (1.6,1.3,0) -- cycle;
\filldraw[fill = white, draw=black, very thick] (0,0.6,0) -- (0,1.6,0) -- (1,1.6,-1) -- cycle;
\filldraw[fill = white, draw=black, very thick] (0,1.6,0) -- (0,1.6,-1) -- (1,1.6,-1) -- cycle;
\fill[white] (0.6,0.6,0) -- (0.6,1.6,0) -- (1.6,1.6,0) -- cycle;
\draw[draw=black, very thick] (0.6,0.6,0) -- (0.6,1.6,0) -- (1.6,1.6,0) -- cycle;
\fill[white] (1.6,1.6,0) -- (1.6,1.6,-1) -- (0.6,1.6,0) -- cycle;
\draw[draw=black, very thick] (1.6,1.6,0) -- (1.6,1.6,-1) -- (0.6,1.6,0) -- (1.6,1.6,0);
\fill (0.6,0,0) circle (.05);
\fill (1.6,0,0) circle (.05);
\fill (1.6,0,-1) circle (.05);
\fill (0,0.3,0) circle (.05);
\fill (0.6,0.3,0) circle (.05);
\fill (1.6,0.3,0) circle (.05);
\fill (1.6,1.3,0) circle (.05);
\fill (1.6,1.3,-1) circle (.05);
\fill (1.6,1,-1) circle (.05);
\fill (1.6,1.6,-1) circle (.05);
\fill (0.6,0.6,0) circle (.05);
\fill (0.6,1.6,0) circle (.05);
\fill (1.6,1.6,0) circle (.05);
\fill (0,0,0) circle (.05);
\fill (0,0.6,0) circle (.05);
\fill (0,1.6,0) circle (.05);
\fill (0,1.6,-1) circle (.05);
\fill (1,1.6,-1) circle (.05);
\fill (-3,0.3,0) circle (0.05);
\fill (-2,0.3,0) circle (0.05);
\fill (-2,1.3,0) circle (0.05);
\fill (-3,1.3,0) circle (0.05);
\draw[very thick] (-3,0.3,0) -- (-2,0.3,0) -- (-2,1.3,0) -- (-3,1.3,0) -- cycle;
\fill (-2,.3,-1) circle (0.05);
\fill (-2,1.3,-1) circle (0.05);
\draw[very thick] (-2,.3,0) -- (-2,0.3,-1) -- (-2,1.3,-1) -- (-2,1.3,0) ;
\fill (-3,1.3,-1) circle (0.05);
\draw[very thick] (-3,1.3,0) -- (-3,1.3,-1) -- (-2,1.3,-1) ;
\fill[very thick] (-3,.3,-1) circle (0.05);
\draw[dashed, very thick] (-3,0.3,0) -- (-3,.3,-0.95);
\draw[dashed, very thick] (-2,0.3,-1) -- (-2.96,.3,-1);
\draw[dashed, very thick] (-3,1.3,-1) -- (-3,.35,-1);
\draw[->,ultra thick] (-1.3,.90,0) -- (-.25,.90,0);
\end{tikzpicture}
\caption{The subdivision of a cube into six tetrahedra, the convex hull of $\{(0;0;0);(1;0;0);(1;0;1);(1;1;1)\}$ is marked.}\label{fig:sub_cube}
\end{figure}

It produces a triangulation of the $n$-cube into $n!$ simplices. Observe now that the ordered set of vertices of $[0;1]^n$ is naturally isomorphic to the lattice of subsets of a set with $n$ elements. Moreover, if $\sigma\leq\tau$ is a pair of adjacent simplices of relative codimension $n$, the lattice of intermediary simplices $\{\sigma\leq \nu\leq \tau\}$ is the same as the lattice of faces of the link of $\sigma$ in $\tau$ (empty face included) i.e. the lattice of subsets of a set with $\dim(\tau)+1-(\dim(\sigma)+1)=n$ elements. For more general polyhedral complexes, the shapes of the pseudo-cells can be different, as shown in Figure~\ref{fig:oct}.

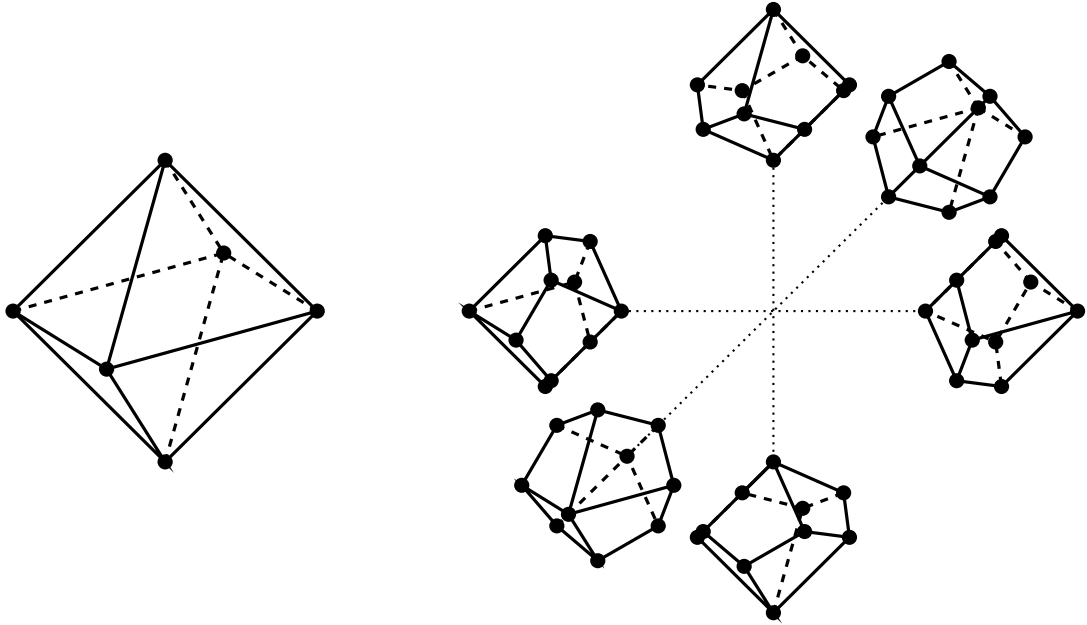
\begin{figure}[h]
\centering
\begin{tikzpicture}[scale=2]
\coordinate (U) at (0,1,0) ;
\coordinate (L) at (-1,0,0);
\coordinate (R) at (1,0,0);
\coordinate (F) at (0,0,1);
\coordinate (B) at (0,0,-1);
\coordinate (D) at (0,-1,0);
\coordinate (v) at (4,0,0);
\fill (U) circle (0.05);
\fill (L) circle (0.05);
\fill (R) circle (0.05);
\fill (F) circle (0.05);
\fill (B) circle (0.05);
\fill (D) circle (0.05);
\draw[very thick] (0,1,0) -- (0,0,1) -- (-1,0,0) -- cycle ;
\draw[dashed,very thick] (0,1,0) -- (0,0,-1) -- (-1,0,0) ;
\draw[very thick] (0,0,1) -- (1,0,0) -- (0,1,0) ;
\draw[dashed, very thick] (1,0,0) -- (0,0,-1);
\draw[very thick] (-1,0,0) -- (0,-1,0) -- (0,0,1) ;
\draw[very thick] (0,-1,0) -- (1,0,0) ;
\draw[dashed, very thick] (0,-1,0) -- (0,0,-1);
\coordinate (vU) at ($(v)+(0,1,0)$);
\fill ($(vU)+(U)$) circle (.05);
\fill ($(vU)+1/2*(U)+1/2*(L)$) circle (.05);
\fill ($(vU)+1/2*(U)+1/2*(F)$) circle (.05);
\fill ($(vU)+1/2*(U)+1/2*(R)$) circle (.05);
\fill ($(vU)+1/2*(U)+1/2*(B)$) circle (.05);
\fill ($(vU)+1/3*(U)+1/3*(L)+1/3*(F)$) circle (.05);
\fill ($(vU)+1/3*(U)+1/3*(L)+1/3*(B)$) circle (.05);
\fill ($(vU)+1/3*(U)+1/3*(R)+1/3*(F)$) circle (.05);
\fill ($(vU)+1/3*(U)+1/3*(R)+1/3*(B)$) circle (.05);
\fill ($(vU)$) circle (.05);
\draw[very thick] ($(vU)+(U)$) -- ($(vU)+1/2*(U)+1/2*(L)$) -- ($(vU)+1/3*(U)+1/3*(L)+1/3*(F)$) -- ($(vU)+1/2*(U)+1/2*(F)$) -- cycle;
\draw[very thick] ($(vU)+(U)$) -- ($(vU)+1/2*(U)+1/2*(R)$) -- ($(vU)+1/3*(U)+1/3*(R)+1/3*(F)$) -- ($(vU)+1/2*(U)+1/2*(F)$);
\draw[dashed, very thick] ($(vU)+(U)$) -- ($(vU)+1/2*(U)+1/2*(B)$) -- ($(vU)+1/3*(U)+1/3*(R)+1/3*(B)$) -- ($(vU)+1/2*(U)+1/2*(R)$);
\draw[dashed, very thick] ($(vU)+1/2*(U)+1/2*(B)$) -- ($(vU)+1/3*(U)+1/3*(L)+1/3*(B)$) -- ($(vU)+1/2*(U)+1/2*(L)$);
\draw[very thick] ($(vU)+1/3*(U)+1/3*(L)+1/3*(F)$) -- (vU) -- ($(vU)+1/3*(U)+1/3*(R)+1/3*(F)$);
\draw[dashed, very thick] ($(vU)+1/3*(U)+1/3*(L)+1/3*(B)$) -- (vU) -- ($(vU)+1/3*(U)+1/3*(R)+1/3*(B)$);
\coordinate (vD) at ($(v)+(0,-1,0)$);
\fill ($(vD)+(D)$) circle (.05);
\fill ($(vD)+1/2*(D)+1/2*(L)$) circle (.05);
\fill ($(vD)+1/2*(D)+1/2*(F)$) circle (.05);
\fill ($(vD)+1/2*(D)+1/2*(R)$) circle (.05);
\fill ($(vD)+1/2*(D)+1/2*(B)$) circle (.05);
\fill ($(vD)+1/3*(D)+1/3*(L)+1/3*(F)$) circle (.05);
\fill ($(vD)+1/3*(D)+1/3*(L)+1/3*(B)$) circle (.05);
\fill ($(vD)+1/3*(D)+1/3*(R)+1/3*(F)$) circle (.05);
\fill ($(vD)+1/3*(D)+1/3*(R)+1/3*(B)$) circle (.05);
\fill ($(vD)$) circle (.05);
\draw[very thick] ($(vD)+(D)$) -- ($(vD)+1/2*(D)+1/2*(L)$) -- ($(vD)+1/3*(D)+1/3*(L)+1/3*(F)$) -- ($(vD)+1/2*(D)+1/2*(F)$) -- cycle;
\draw[very thick] ($(vD)+(D)$) -- ($(vD)+1/2*(D)+1/2*(R)$) -- ($(vD)+1/3*(D)+1/3*(R)+1/3*(F)$) -- ($(vD)+1/2*(D)+1/2*(F)$);
\draw[dashed, very thick] ($(vD)+(D)$) -- ($(vD)+1/2*(D)+1/2*(B)$) -- ($(vD)+1/3*(D)+1/3*(R)+1/3*(B)$) ;
\draw[very thick] ($(vD)+1/3*(D)+1/3*(R)+1/3*(B)$) -- ($(vD)+1/2*(D)+1/2*(R)$);
\draw[dashed, very thick] ($(vD)+1/2*(D)+1/2*(B)$) -- ($(vD)+1/3*(D)+1/3*(L)+1/3*(B)$) -- ($(vD)+1/2*(D)+1/2*(L)$);
\draw[very thick] ($(vD)+1/3*(D)+1/3*(L)+1/3*(F)$) -- (vD) -- ($(vD)+1/3*(D)+1/3*(R)+1/3*(F)$);
\draw[very thick] ($(vD)+1/3*(D)+1/3*(L)+1/3*(B)$) -- (vD) -- ($(vD)+1/3*(D)+1/3*(R)+1/3*(B)$);
\coordinate (vF) at ($(v)+(0,0,2.5)$);
\fill ($(vF)+(F)$) circle (.05);
\fill ($(vF)+1/2*(F)+1/2*(L)$) circle (.05);
\fill ($(vF)+1/2*(F)+1/2*(D)$) circle (.05);
\fill ($(vF)+1/2*(F)+1/2*(R)$) circle (.05);
\fill ($(vF)+1/2*(F)+1/2*(U)$) circle (.05);
\fill ($(vF)+1/3*(F)+1/3*(L)+1/3*(D)$) circle (.05);
\fill ($(vF)+1/3*(F)+1/3*(L)+1/3*(U)$) circle (.05);
\fill ($(vF)+1/3*(F)+1/3*(R)+1/3*(D)$) circle (.05);
\fill ($(vF)+1/3*(F)+1/3*(R)+1/3*(U)$) circle (.05);
\fill ($(vF)$) circle (.05);
\draw[very thick] ($(vF)+(F)$) -- ($(vF)+1/2*(F)+1/2*(L)$) -- ($(vF)+1/3*(F)+1/3*(L)+1/3*(D)$) -- ($(vF)+1/2*(F)+1/2*(D)$) -- cycle;
\draw[very thick] ($(vF)+(F)$) -- ($(vF)+1/2*(F)+1/2*(R)$) -- ($(vF)+1/3*(F)+1/3*(R)+1/3*(D)$) -- ($(vF)+1/2*(F)+1/2*(D)$);
\draw[very thick] ($(vF)+(F)$) -- ($(vF)+1/2*(F)+1/2*(U)$) -- ($(vF)+1/3*(F)+1/3*(R)+1/3*(U)$) -- ($(vF)+1/2*(F)+1/2*(R)$);
\draw[very thick] ($(vF)+1/2*(F)+1/2*(U)$) -- ($(vF)+1/3*(F)+1/3*(L)+1/3*(U)$) -- ($(vF)+1/2*(F)+1/2*(L)$);
\draw[dashed,very thick] ($(vF)+1/3*(F)+1/3*(L)+1/3*(D)$) -- (vF) -- ($(vF)+1/3*(F)+1/3*(R)+1/3*(D)$);
\draw[dashed, very thick] ($(vF)+1/3*(F)+1/3*(L)+1/3*(U)$) -- (vF) -- ($(vF)+1/3*(F)+1/3*(R)+1/3*(U)$);
\coordinate (vB) at ($(v)+(0,0,-2.5)$);
\fill ($(vB)+(B)$) circle (.05);
\fill ($(vB)+1/2*(B)+1/2*(L)$) circle (.05);
\fill ($(vB)+1/2*(B)+1/2*(D)$) circle (.05);
\fill ($(vB)+1/2*(B)+1/2*(R)$) circle (.05);
\fill ($(vB)+1/2*(B)+1/2*(U)$) circle (.05);
\fill ($(vB)+1/3*(B)+1/3*(L)+1/3*(D)$) circle (.05);
\fill ($(vB)+1/3*(B)+1/3*(L)+1/3*(U)$) circle (.05);
\fill ($(vB)+1/3*(B)+1/3*(R)+1/3*(D)$) circle (.05);
\fill ($(vB)+1/3*(B)+1/3*(R)+1/3*(U)$) circle (.05);
\fill ($(vB)$) circle (.05);
\draw[dashed, very thick] ($(vB)+(B)$) -- ($(vB)+1/2*(B)+1/2*(L)$) ;
\draw[dashed, very thick] ($(vB)+(B)$) -- ($(vB)+1/2*(B)+1/2*(D)$) ;
\draw[very thick] ($(vB)+1/2*(B)+1/2*(L)$) -- ($(vB)+1/3*(B)+1/3*(L)+1/3*(D)$) -- ($(vB)+1/2*(B)+1/2*(D)$);
\draw[dashed, very thick] ($(vB)+(B)$) -- ($(vB)+1/2*(B)+1/2*(R)$);
\draw[very thick] ($(vB)+1/2*(B)+1/2*(R)$) -- ($(vB)+1/3*(B)+1/3*(R)+1/3*(D)$) -- ($(vB)+1/2*(B)+1/2*(D)$);
\draw[dashed, very thick] ($(vB)+(B)$) -- ($(vB)+1/2*(B)+1/2*(U)$);
\draw[very thick] ($(vB)+1/2*(B)+1/2*(U)$) -- ($(vB)+1/3*(B)+1/3*(R)+1/3*(U)$) ;
\draw[very thick] ($(vB)+1/3*(B)+1/3*(R)+1/3*(U)$) -- ($(vB)+1/2*(B)+1/2*(R)$);
\draw[very thick] ($(vB)+1/2*(B)+1/2*(U)$) -- ($(vB)+1/3*(B)+1/3*(L)+1/3*(U)$) -- ($(vB)+1/2*(B)+1/2*(L)$);
\draw[very thick] ($(vB)+1/3*(B)+1/3*(L)+1/3*(D)$) -- (vB) -- ($(vB)+1/3*(B)+1/3*(R)+1/3*(D)$);
\draw[very thick] ($(vB)+1/3*(B)+1/3*(L)+1/3*(U)$) -- (vB) -- ($(vB)+1/3*(B)+1/3*(R)+1/3*(U)$);
\coordinate (vL) at ($(v)+(-1,0,0)$);
\fill ($(vL)+(L)$) circle (.05);
\fill ($(vL)+1/2*(L)+1/2*(D)$) circle (.05);
\fill ($(vL)+1/2*(L)+1/2*(F)$) circle (.05);
\fill ($(vL)+1/2*(L)+1/2*(U)$) circle (.05);
\fill ($(vL)+1/2*(L)+1/2*(B)$) circle (.05);
\fill ($(vL)+1/3*(L)+1/3*(D)+1/3*(F)$) circle (.05);
\fill ($(vL)+1/3*(L)+1/3*(D)+1/3*(B)$) circle (.05);
\fill ($(vL)+1/3*(L)+1/3*(U)+1/3*(F)$) circle (.05);
\fill ($(vL)+1/3*(L)+1/3*(U)+1/3*(B)$) circle (.05);
\fill ($(vL)$) circle (.05);
\draw[very thick] ($(vL)+(L)$) -- ($(vL)+1/2*(L)+1/2*(D)$) -- ($(vL)+1/3*(L)+1/3*(D)+1/3*(F)$) -- ($(vL)+1/2*(L)+1/2*(F)$) -- cycle;
\draw[very thick] ($(vL)+(L)$) -- ($(vL)+1/2*(L)+1/2*(U)$) -- ($(vL)+1/3*(L)+1/3*(U)+1/3*(F)$) -- ($(vL)+1/2*(L)+1/2*(F)$);
\draw[dashed, very thick] ($(vL)+(L)$) -- ($(vL)+1/2*(L)+1/2*(B)$) -- ($(vL)+1/3*(L)+1/3*(U)+1/3*(B)$) ;
\draw[very thick] ($(vL)+1/3*(L)+1/3*(U)+1/3*(B)$) -- ($(vL)+1/2*(L)+1/2*(U)$);
\draw[dashed, very thick] ($(vL)+1/2*(L)+1/2*(B)$) -- ($(vL)+1/3*(L)+1/3*(D)+1/3*(B)$) -- ($(vL)+1/2*(L)+1/2*(D)$);
\draw[very thick] ($(vL)+1/3*(L)+1/3*(D)+1/3*(F)$) -- (vL) -- ($(vL)+1/3*(L)+1/3*(U)+1/3*(F)$);
\draw[dashed, very thick] ($(vL)+1/3*(L)+1/3*(D)+1/3*(B)$) -- (vL);
\draw[very thick] (vL) -- ($(vL)+1/3*(L)+1/3*(U)+1/3*(B)$);
\coordinate (vR) at ($(v)+(1,0,0)$);
\fill ($(vR)+(R)$) circle (.05);
\fill ($(vR)+1/2*(R)+1/2*(U)$) circle (.05);
\fill ($(vR)+1/2*(R)+1/2*(F)$) circle (.05);
\fill ($(vR)+1/2*(R)+1/2*(D)$) circle (.05);
\fill ($(vR)+1/2*(R)+1/2*(B)$) circle (.05);
\fill ($(vR)+1/3*(R)+1/3*(U)+1/3*(F)$) circle (.05);
\fill ($(vR)+1/3*(R)+1/3*(U)+1/3*(B)$) circle (.05);
\fill ($(vR)+1/3*(R)+1/3*(D)+1/3*(F)$) circle (.05);
\fill ($(vR)+1/3*(R)+1/3*(D)+1/3*(B)$) circle (.05);
\fill ($(vR)$) circle (.05);
\draw[very thick] ($(vR)+(R)$) -- ($(vR)+1/2*(R)+1/2*(U)$) -- ($(vR)+1/3*(R)+1/3*(U)+1/3*(F)$) -- ($(vR)+1/2*(R)+1/2*(F)$) -- cycle;
\draw[very thick] ($(vR)+(R)$) -- ($(vR)+1/2*(R)+1/2*(D)$) -- ($(vR)+1/3*(R)+1/3*(D)+1/3*(F)$) -- ($(vR)+1/2*(R)+1/2*(F)$);
\draw[dashed, very thick] ($(vR)+(R)$) -- ($(vR)+1/2*(R)+1/2*(B)$) -- ($(vR)+1/3*(R)+1/3*(D)+1/3*(B)$) -- ($(vR)+1/2*(R)+1/2*(D)$);
\draw[dashed, very thick] ($(vR)+1/2*(R)+1/2*(B)$) -- ($(vR)+1/3*(R)+1/3*(U)+1/3*(B)$) -- ($(vR)+1/2*(R)+1/2*(U)$);
\draw[very thick] ($(vR)+1/3*(R)+1/3*(U)+1/3*(F)$) -- (vR) -- ($(vR)+1/3*(R)+1/3*(D)+1/3*(F)$);
\draw[dashed, very thick] ($(vR)+1/3*(R)+1/3*(U)+1/3*(B)$) -- (vR) -- ($(vR)+1/3*(R)+1/3*(D)+1/3*(B)$);
\draw[dotted, thick] (vF) -- (vB);
\draw[dotted, thick] (vU) -- (vD);
\draw[dotted, thick] (vL) -- (vR);
\end{tikzpicture}
\caption{The dihomologic subdivision of an octahedron.}\label{fig:oct}
\end{figure}

However, even when $K$ doesn't satisfy the hypotheses of the Proposition~\ref{prop:dih_reg_shel} all the 2-dimensional dihomologic pseudo-cells are squares because of Lemma~\ref{lemm:2_faces}, c.f. Figure~\ref{fig:dih_squares}. Finally, in low dimension the pseudo subdivision is always regular:

\begin{prop}\label{prp:low_dim}
Let $e^p\leq e^q$ be a pair of cells of $K$. If, $2\leq q-p\leq 4$, then $A(e^p;e^q)$ is homeomorphic to a sphere. If $q=p+5$, $A(e^p;e^q)$ is a $3$-dimensional integral homology sphere.
\end{prop}

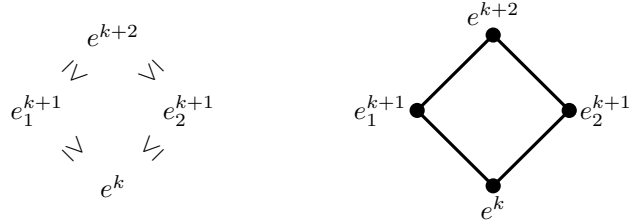
\begin{figure}[h]
\centering
\begin{tikzpicture}[scale=2]
\draw (0,0) node{$e^k$};
\draw (-.5,.5) node{$e^{k+1}_1$};
\draw (.5,.5) node{$e^{k+1}_2$};
\draw (0,1) node{$e^{k+2}$};
\draw ($(0,0)!.5!(-.5,.5)$) node[rotate=-45]{$\geq$};
\draw ($(0,0)!.5!(.5,.5)$) node[rotate=45]{$\leq$};
\draw ($(0,1)!.5!(-.5,.5)$) node[rotate=-135]{$\geq$};
\draw ($(0,1)!.5!(.5,.5)$) node[rotate=135]{$\leq$};
\draw[very thick] (2.5,0) node[anchor=north]{$e^k$} -- (2,.5) node[anchor=east]{$e^{k+1}_1$} -- (2.5,1) node[anchor=south]{$e^{k+2}$} -- (3,0.5) node[anchor=west]{$e^{k+1}_2$} -- cycle;
\fill (2.5,0) circle (.05);
\fill (2,.5) circle (.05);
\fill (2.5,1) circle (.05);
\fill (3,0.5) circle (.05);
\end{tikzpicture}
\caption{The lattice of faces and the dihomologic square associated with a pair of relative codimension~2.}\label{fig:dih_squares}
\end{figure}

\begin{proof}
As shown in Figure~\ref{fig:dih_squares}, the simplicial complex $A(e^p;e^{p+2})$ consists of two vertices and is therefore a $0$-sphere. If we look at $A(e^p;e^{p+3})$ the link of every simplex is either empty or a $A(e^k;e^{k+2})$, so $A(e^p;e^{p+3})$ is actually a manifold by Proposition~{1.3} of~\cite[Proposition~1.3 p.~5]{Gal-Ste_Cla_Sim}.
Therefore, by Proposition~\ref{prp:A_homo}, it is a $1$-dimensional integral homology sphere, so a circle. Now for $A(e^k;e^{k+4})$ we have from the proof of Proposition~\ref{prp:A_homo} that the link of every simplex is either empty or a join of a $A(e^k;e^{k+2})$ with a $A(e^k;e^{k+3})$ which we have just shown to be spheres. Therefore, $A(e^k;e^{k+4})$ is a $2$-dimensional integral homology sphere. By classification of compact orientable~$2$-dimensional manifolds it is homeomorphic to a $2$-sphere. For the last part our previous arguments show that the $A(e^p;e^{p+5})$ are $3$-dimensional integral homology spheres.
\end{proof}

\begin{rema}
The $3$-dimensional closed pseudo-cells are not only closed balls but even \emph{trapezohedra} i.e. similar to Figure~\ref{fig:trape}. The family of such polyhedra is indexed by integers $n$ at least equal to $3$ (for which we find ``the cube'').
\end{rema}

\begin{figure}[h]
\centering
\begin{tikzpicture}[scale=2]
\foreach \a in {0,60,120,180,240,300}{

\begin{scope}[canvas is xz plane at y=0]
\coordinate (p) at (\a:1);
\end{scope}

\fill (p) circle (.05);}

\foreach \a in {0,60,120,300}{

\begin{scope}[canvas is xz plane at y=0]
\draw[very thick] (\a:1) -- ($(0,0)!1!60:(\a:1)$);
\end{scope}}

\foreach \a in {180,240}{

\begin{scope}[canvas is xz plane at y=0]
\draw[dashed,very thick] (\a:1) -- ($(0,0)!1!60:(\a:1)$);
\end{scope}}

\foreach \a in {0,60,120,180,300}{

\begin{scope}[canvas is xz plane at y=0]
\coordinate (p) at (\a:1);
\end{scope}

\draw[very thick] (p) -- (0,1,0);}

\foreach \a in {240}{

\begin{scope}[canvas is xz plane at y=0]
\coordinate (p) at (\a:1);
\end{scope}

\draw[dashed, very thick] (p) -- (0,1,0);}

\fill (0,1,0) node[above=2pt]{$e^0$} circle (.05);
\draw (-.2,0,0) node{$e^3$};
\coordinate (v) at (4,-.2,0);
\draw[ultra thick, ->] (1.3,.7,0) -- (2.3,.7,0);
\fill ($(v)+2*(0,1,0)$) node[above=2pt]{$e^0$} circle (.05);

\foreach \a in {0,60,120,180,240,300}{

\begin{scope}[canvas is xz plane at y=0]
\coordinate (p) at (\a:1);
\end{scope}

\fill ($(v)+(p)+(0,1,0)$) circle (.05);}

\foreach \a in {0,60,120,180,300}{

\begin{scope}[canvas is xz plane at y=0]
\coordinate (p) at (\a:1);
\end{scope}

\draw[very thick] ($(v)+(p)+(0,1,0)$) -- ($(v)+2*(0,1,0)$);}

\foreach \a in {240}{

\begin{scope}[canvas is xz plane at y=0]
\coordinate (p) at (\a:1);
\end{scope}

\draw[dashed,very thick] ($(v)+(p)+(0,1,0)$) -- ($(v)+2*(0,1,0)$);}

\foreach \a in {0,60,120,180,240,300}{

\begin{scope}[canvas is xz plane at y=0]
\coordinate (p) at (\a:1);
\coordinate (q) at ($(0,0)!1!60:(\a:1)$);
\end{scope}

\fill ($(v)+2/3*(p)+2/3*(0,1,0)+2/3*(q)$) circle (.05);}

\foreach \a in {0,60,120,300}{

\begin{scope}[canvas is xz plane at y=0]
\coordinate (p) at (\a:1);
\coordinate (q) at ($(0,0)!1!60:(\a:1)$);
\end{scope}

\draw[very thick] ($(v)+(p)+(0,1,0)$) -- ($(v)+2/3*(p)+2/3*(0,1,0)+2/3*(q)$) -- ($(v)+(q)+(0,1,0)$);}

\foreach \a in {240,180}{

\begin{scope}[canvas is xz plane at y=0]
\coordinate (p) at (\a:1);
\coordinate (q) at ($(0,0)!1!60:(\a:1)$);
\end{scope}

\draw[dashed, very thick] ($(v)+(p)+(0,1,0)$) -- ($(v)+2/3*(p)+2/3*(0,1,0)+2/3*(q)$) -- ($(v)+(q)+(0,1,0)$);}

\foreach \a in {0,60,120,300}{

\begin{scope}[canvas is xz plane at y=0]
\coordinate (p) at (\a:1);
\coordinate (q) at ($(0,0)!1!60:(\a:1)$);
\end{scope}

\draw[very thick] ($(v)+(0,-.7,0)$) -- ($(v)+2/3*(p)+2/3*(0,1,0)+2/3*(q)$);}

\foreach \a in {180,240}{

\begin{scope}[canvas is xz plane at y=0]
\coordinate (p) at (\a:1);
\coordinate (q) at ($(0,0)!1!60:(\a:1)$);
\end{scope}

\draw[dashed, very thick] ($(v)+(0,-.7,0)$) -- ($(v)+2/3*(p)+2/3*(0,1,0)+2/3*(q)$);}

\fill ($(v)+(0,-.7,0)$) node[below=2pt]{$e^3$} circle (.05);
\end{tikzpicture}
\caption{The shape of the dihomologic cell associated with the adjacent pair $e^0\leq e^3$.}\label{fig:trape}
\end{figure}
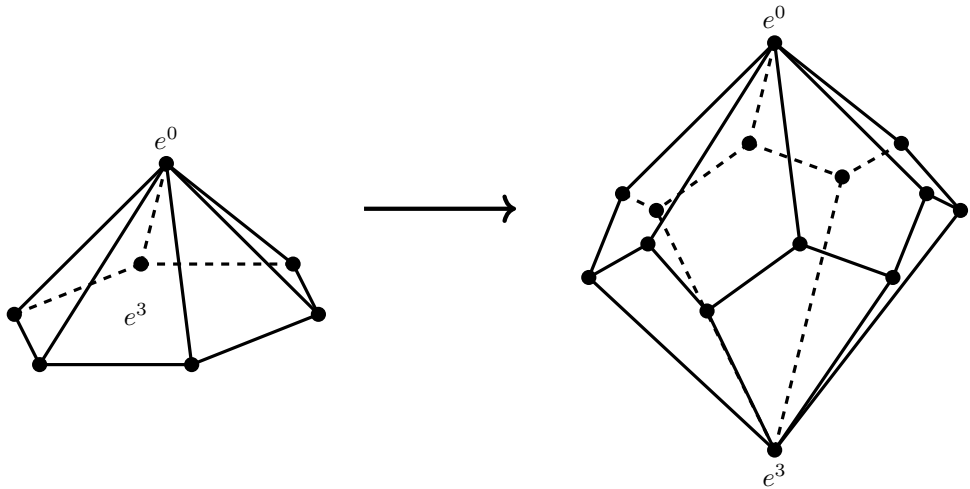

A direct consequence of Proposition~\ref{prp:low_dim} and Proposition~\ref{prop:initial_cells} is the proposition below.

\begin{prop}
If $\dim K$ is at most 5 then its dihomologic pseudo-subdivi\-sion is a regular subdivision.
\end{prop}

\subsection{Cellular Sheaves and Cosheaves}

In this paragraph $K$ denotes a regular CW-complex.

\begin{defi}[Constructible Sheaves]
A sheaf $F$ on $|K|$ is called \emph{constructible} with respect to the CW-complex structure if its restriction to every open cell is constant.
\end{defi}

Such sheaves are among the ``simplest ones'' on $|K|$ as they are reducible to combinatorial data. The knowledge of their section groups above every open star as well as the restrictions morphisms between these stars is enough to characterise the sheaf completely up to isomorphism. The key fact about these sheaves is the following: for every cell $e$ of $K$, if $S$ denotes its open star and $x\in e$ then the two following morphisms are isomorphisms:
\begin{equation*}
\begin{tikzcd}
F(S)\arrow[r,"\text{rest.}" below] & F|_e(e) \ar[r, "\text{stalk}" below] & F_x.
\end{tikzcd}
\end{equation*}
See for instance~\cite[Proposition~1.3 p.~323]{Kas_rie_hil}. Since $e$ is connected and locally connected, the fact that $F|_e(e) \rightarrow F_x$ is an isomorphism follows from the constance of $F|_e$. Let $F(e)$ denote the group $F|_e(e)$ of values of $F|_e$. From the previous observation, we gain a ``restriction map'' from $F(e^p)$ and $F(e^q)$ for all pairs of adjacent cells $e^p\leq e^q$ that comes from the following commutative diagram:
\begin{equation*}
\begin{tikzcd}
F(e^p) \arrow[rr] & & F(e^q) \\
F(S_{e^p}) \arrow[u,"\text{rest. to }F|_{e^p}(e^p)" left, "\cong " right] \arrow[rr, "S_{e^q}\,\subset\,S_{e^p}" below, "\text{rest.}" above] & & F(S_{e^q}) \ar[u, "\text{rest. to }F|_{e^q}(e^q)" right, "\cong" left]
\end{tikzcd}
\end{equation*}
where $S_{e^p}$ and $S_{e^q}$ are the respective open stars of $e^p$ and $e^q$. The data of the groups $F(e^p)$ together with the morphisms connecting them, defined from the topological sheaf $F$, is called a cellular sheaf:

\begin{defi}[Cellular Sheaf]
A \emph{cellular sheaf} on $K$ is the data of a covariant functor:
\begin{equation*}
F:\Cell\;K\rightarrow \Mod_R,
\end{equation*}
from the category of cells of $K$ with arrows given by adjacency to the category of $R$-modules (for some commutative ring $R$). We call the images of the arrows by such functor its \emph{restriction morphisms}. For two adjacent cells $e^p\leq e^q$ and $f\in F(e^p)$ we will denote by $f|^{e^q}_{e^p}$ the image of $f$ in $F(e^q)$ by the restriction morphism.
\end{defi}

\begin{defi}[Cellular Cosheaf]
A \emph{cellular cosheaf} on $K$ is the data of a contravariant functor:
\begin{equation*}
F:\Cell\;K\rightarrow \Mod_R,
\end{equation*}
from the category of cells of $K$ with arrows given by adjacency to the category of $R$-modules (for some commutative ring $R$). We call the images of the arrows by such functor its \emph{extension morphisms}. For two adjacent cells $e^p\leq e^q$ and $f\in F(e^q)$ we will denote by $f|^{e^q}_{e^p}$ the image of $f$ in $F(e^p)$ by the extension morphism.
\end{defi}

Every functorial operation performed on Abelian groups, or more generally on modules over a given commutative ring, such as direct sums, products, tensor products, etc. can be performed as well on cellular sheaves and cosheaves by performing it group by group over every cell. Also, we can construct a cosheaf from a sheaf $F$ by considering, for $G$ a fixed group, the contravariant functor $e\mapsto \Hom(F(e);G)$ with adjoint arrows. This construction also goes the other way around when one starts with a cosheaf.

\begin{defi}[Morphisms of Sheaves and Cosheaves]
A \emph{morphism of cellular sheaves} (or \emph{cellular cosheaves}) $f:F\rightarrow F'$ is a natural transformation. Such morphism is said to be \emph{injective} (resp. \emph{surjective}, resp. \emph{invertible}) if the associated morphisms $f_e:F(e)\rightarrow F'(e)$ are injective (resp. surjective, resp. invertible) for all cells $e$. The \emph{kernel}, \emph{image}, and \emph{cokernel} of such morphism $f$ are the ``cell-wise'' kernel, image, and cokernel. They are sheaves/cosheaves themselves with the induced restriction/extension morphisms because $f$ is a natural transformation.
\end{defi}

The most basic examples of such objects are given by \emph{local systems of coefficients}. We see such local systems as fibre bundles of discrete groups above $|K|$. Since every cell is connected and contractible the restriction of its sheaf of continuous sections to any cell is constant. Therefore, it satisfies the hypothesis of the definition and induces a cellular sheaf. It has the property that all its restriction morphisms are invertible. This is even a way to characterise such local systems. This special property also allows us to see it as a cellular cosheaf by inverting every arrows. Indeed, the commutativity conditions on the composition of such morphisms are automatically satisfied from the ones given by the cellular sheaf structure. Another family of examples is given by the \emph{characteristic cosheaves} associated with sub-complexes:

\begin{defi}
Let $K'$ be a sub-complex of $K$ and $G$ be an Abelian group (or a module over a commutative ring) we denote by $[K';G]$ the cellular cosheaf defined by:
\begin{equation*}
e\in K \longmapsto \left\{
\begin{array}{cl} G & \text{if }e\in K' \\ 0 & \text{otherwise}
\end{array}\right..
\end{equation*}
Its extension morphisms are given either by the identity of $G$ or by the zero morphism whenever one of the two groups involved is trivial. If a cell belongs to $K'$ then all of its faces belong to it too. As a consequence the commutativity conditions are satisfied for every triplet of adjacent cells gives rise to one of the following commutative diagrams:
\begin{equation*}
\begin{tikzcd}
G \ar[dd,"\id"] \ar[dr,"\id"] & &  0 \ar[dd,"0"] \ar[dr,"0"] & &  0 \ar[dd,"0"] \ar[dr,"0"] & &  0 \ar[dd,"0"] \ar[dr,"0"] & & \\%
& G \ar[dl,"\id"] & &  G \ar[dl,"\id"] & &  0 \ar[dl,"0"] & &  0 \ar[dl,"0"] & \\%
G & &  G & &  G & &  0 & &
\end{tikzcd}
\end{equation*}
Whenever $K''$ is a sub-complex of $K'$, we have a natural injective morphism of cosheaves $[K'';G]\rightarrow [K'\,;G]$. Over a cell it is either given by the $0$ morphism or by the identity of $G$. We will denote the resulting quotient by $[K';K'';G]$. It is $G$ on the cells of $K'$ not contained in $K''$ and $0$ elsewhere. A sub-family of these examples will be of particular interest. They are the ``local'' cosheaf $[K;K-e\,;G]$, for all cells $e$ of $K$. Its value is $G$ only on the cells containing $e$. For all cells $e$, the cosheaves $[K(e)\,;G]$ and $[K;K-e\,;G]$ are the dual constructions of the elementary cellular sheaves considered by A. Shepard in his thesis~\cite{She_cel_des}. They were also considered later by J. Curry (in~\cite{Cur_she_cos} for instance).
\end{defi}

\begin{defi}[Localisation of a Cellular Cosheaf]
Let $F$ be a cellular cosheaf on $K$ and $e$ be a cell. We denote by $F_e$ the tensor product $F\otimes_\Z[K;K-e\,;\Z]$ and call it the \emph{localisation} of $F$ at $e$. For all cells $e'$, $F_e(e')$ is $F(e')$ if $e'\geq e$ and $0$ otherwise, its extension morphisms are then appropriately given by the extension morphisms of $F$ or $0$. Moreover, the natural projection $[K;K-e\,;\Z]\rightarrow [K;K-e'\,;\Z]$ for adjacent cells $e\leq e'$ induces a surjective localisation morphism $F_e\rightarrow F_{e'}$.
\end{defi}

\begin{defi}[Subdivision]\label{definition:subd}
If $K'$ is a subdivision of $K$, there is a subdivision functor from the category of cellular cosheaves on $K$ (resp. cellular sheaves) to the category of cellular cosheaves on $K'$ (resp. cellular sheaves). If $F$ is a cosheaf (resp. sheaf) on $K$, its \emph{subdivision} $F'$ is given, for all cells $e'\in K'$, by:
\begin{equation*}
F'(e')=F(e),
\end{equation*}
where $e$ denotes the only cell of $K$ containing $e'$. The extension (resp. restriction) morphisms are then adequately derived from those of $F$. If $e'_0\leq e'_1$ are contained in the same cell, the morphism is the identity. If this is not the case, then it is given by the morphism associated with the only pair of cells $e_0\leq e_1$ of $K$ satisfying $e'_0\subset e_0$ and $e'_1\subset e_1$.
\end{defi}

\begin{defi}[Dihomologic Cellular Sheaves and Cosheaves]\label{definition:dih_cosheaves}
A \emph{dihomologic cellular cosheaf} (resp. \emph{sheaf}) on $K$ is the data of a contravariant (resp. covariant) functor $F$ from the category associated with the set of dihomologic pseudo-cells of $K$ ordered by adjacency to the category of $R$-modules. As in the case of cellular sheaves and cosheaves a morphism of such objects is defined to be a natural transformation of functors. The notions of injectivity, surjectivity and invertibility are also defined ``cell-wise'' and so are the kernels, images and cokernels.


Formally, a dihomologic cellular cosheaf consists of an assignment of a module $F(e^p;e^q)$ to every pair of adjacent cells $e^p\leq e^q$ and morphisms connecting them. Because of the commutativity conditions on the compositions of such morphisms, it is only necessary to define them on elementary adjacency relations. By that, we mean that if the dihomologic pseudo-cell of the pair $e^p\leq e^q$ is a face of the pseudo-cell $e^{p'}\leq e^{q'}$ then we have $e^{p'}\leq e^{p}\leq e^{q}\leq e^{q'}$ and the commutative diagram of extension morphisms:
\begin{equation*}
\begin{tikzcd}
& F(e^{p'};e^{q'}) \ar[dd, "(5)" description] \ar[dr,"(3)" description] \ar[dl,"(1)" description] & \\
F(e^{p};e^{q'}) \ar[dr,"(2)" description] & & F(e^{p'};e^{q}) \ar[dl,"(4)" description] \\
& F(e^{p};e^{q}) &
\end{tikzcd}
\end{equation*}
Knowing the morphism $(5)$ only amounts to knowing the composition of $(1)$ and $(2)$ or $(3)$ and $(4)$. So to describe such $F$ completely we can only provide the groups and the extension morphisms when we ``increase the first coordinate'' and ``decrease the second one'' and verify that these satisfy the commutative diagram:
\begin{equation*}
\begin{tikzcd}
& F(e^{p'};e^{q'}) \ar[dr] \ar[dl] & \\
F(e^{p};e^{q'}) \ar[dr] & & F(e^{p'};e^{q}) \ar[dl] \\
& F(e^{p};e^{q}) &
\end{tikzcd}
\end{equation*}
\end{defi}

\begin{defi}[Dihomologic Subdivision of Cellular Sheaves and Cosheaves]
Let $F$ be a cellular cosheaf (resp. sheaf) on $K$, its \emph{dihomologic subdivision} $F'$ is the dihomologic cellular cosheaf (resp. sheaf) that associates to every pair of adjacent cells $e^p\leq e^q$ the module:
\begin{equation*}
F'(e^p;e^q):=F(e^q),
\end{equation*}
with extension (resp. restriction) morphisms coming from those of $F$ and illustrated in the following commutative diagram (resp. with opposite arrows) for the elementary adjacency relations $e^{p'}\leq e^{p}\leq e^{q}\leq e^{q'}$:
\begin{equation*}
\begin{tikzcd}
& & F(e^{q'}) \arrow[ddll,"\id" above left] \arrow[ddrr,"\big|^{e^{q'}}_{e^q}"] \arrow[d,equal] & & \\%
& & F'(e^{p'};e^{q'}) \arrow[dr] \arrow[dl] & & \\%
F(e^{q'})\arrow[ddrr,"\big|^{e^{q'}}_{e^q}" below left] \arrow[r,equal] & F'(e^{p};e^{q'}) \arrow[dr] & & F'(e^{p'};e^{q}) \arrow[dl] \arrow[r,equal] & F(e^q) \arrow[ddll,"\id"] \\%
& & F'(e^{p};e^{q}) \arrow[d,equal] & & \\%
& & F(e^q) & &
\end{tikzcd}
\end{equation*}
Whenever the dihomologic pseudo-subdivision of $K$ is a regular subdivision this construction corresponds to the usual subdivision of cosheaves (resp. sheaves) of Definition~\ref{definition:subd}. The open cell $e^q$ is covered by the open dihomologic (pseudo)-cells associated with the adjacent pairs of the form $e^p\leq e^q$.
\end{defi}

\begin{defi}[Localisation by fixing the first coordinate]
Let $F$ be a dihomologic cosheaf on $K$ and $e$ be a cell of $K$. We define the local cellular cosheaf $F_e$ on $K$ by the formula:
\begin{equation*}
e'\in K \longmapsto \left\{
\begin{array}{cl} F(e;e') & \text{if }e'\geq e \\ 0 & \text{otherwise}
\end{array}\right.,
\end{equation*}
with extension morphisms either $0$ or given by $F$. We call it local as it is invariant by the operation of localisation at $e$: $(F_e)_e=F_e$. Moreover, if we apply this process to a dihomologic cosheaf $F'$ obtained by subdividing a cellular cosheaf $F$, we recover the localisation operation previously defined. The situation is illustrated in the following commutative diagram:
\begin{equation}
\tag{D1}\label{diag:comm_rel_loc}
\begin{tikzcd}
\big\{\textnormal{Cosheaves of }K\big\} \arrow[dd,"\textnormal{loc. at }e" left] \arrow[rr,"\textnormal{subd.}"] & & \big\{\textnormal{Dihomologic cosheaves of }K\big\} \arrow[ddll,"\textnormal{fix. loc. at }e" description] \arrow[dd,"\textnormal{loc. at }(e\,\leq\,e)"] \\%
& \\%
\big\{\textnormal{Cosheaves of }K\big\} \ar[rr,"\textnormal{subd.}" below] & & \big\{\textnormal{Dihomologic cosheaves of }K\big\}
\end{tikzcd}
\end{equation}
\end{defi}

\subsection{Cellular Homology and Cohomology}

In this paragraph, $K$ denotes a locally finite regular CW-complex. When one computes the homology of the CW-complex $K$ cellularilly by filtering, the singular chain complex for instance, by its skeleta, one ends up on the $E^1$-page with the cellular chain complex of $K$. The $k$-th group of this complex is given by the direct sum of free Abelian groups of rank $1$, one for each $k$-cell. These groups, that we redefine below, are a key ingredient in cellular homology. They all have two generators that corresponds to the two orientations of the cell.

\begin{defi}[{\cite[\S 39. pp.~222-231]{Munk_ele_alg}}]
Let $e$ be a $k$-cell of $K$. We call an \emph{orientation} of $e$ a generator of the group $\Z(e):=H_k(|K|;|K|\setminus e;\Z)=H_k(\bar{e};\bar{e}\setminus e;\Z)$ computed by singular homology. We will call the latter group the \emph{group of oriented coefficients} of $e$ and say that $[e]$ is an \emph{oriented $k$-cell} when $[e]$ is an orientation of the $k$-cell $e$. Whenever $e^{p-1}$ is a codimension~$1$ face of $e^{p}$, we have a \emph{boundary morphism} $\Z(e^p)\rightarrow\Z(e^{p-1})$ defined by the following composition:
\begin{equation*}
\begin{tikzcd}
H_k(\bar{e}\,^{p};\bar{e}\,^{p}\setminus e^{p}) \arrow[r,"(1)" description] & H_{k-1}(\bar{e}\,^{p}\setminus e^{p}) \arrow[ld,"(2)" description] \\ H_{k-1}\big(\bar{e}\,^{p}\setminus e^{p} ; \bar{e}\,^{p}\setminus \left(e^{p}\cup e^{p-1}\right)\big) \arrow[r,"(3)" description] & H_{k-1}\left(\bar{e}\,^{p-1} ; \bar{e}\,^{p-1}\setminus e^{p-1}\right)\;,
\end{tikzcd}
\end{equation*}
where all four homology groups are computed with integer coefficients. The morphism $(1)$ is the connection morphism of the homological long exact sequence associated with the pair $(\bar{e}\,^{p}\setminus e^{p})\subset \bar{e}\,^{p}$, $(2)$ is the reduction modulo $\bar{e}\,^{p}\setminus (e^{p}\cup e^{p-1})$, and $(3)$ is the inverse of the excision isomorphism. The image of an orientation $[e^p]$ under this morphism is by definition the $\Z(e^{p-1})$-component of its boundary. It is a generator of $\Z(e^{p-1})$. The first map, the connection morphism, comes from the boundary operator of the singular homology chain complex and therefore relies on the canonical orientation of $\R^n$. Let $\sigma:\Conv(\{0,\dots,n\})\rightarrow |K|$ be a singular simplex, we have the following formula:
\begin{equation*}
\partial \sigma = \sum_{i=0}^n (-1)^i\sigma_i\,,
\end{equation*}
where $\sigma_i$ denotes the restriction of $\sigma$ to $\Conv(\{0,\dots,n\}\setminus\{i\})$. The convention on the orientation of such restriction $\sigma_i$ is then given by ``outward pointing normal vector'' as illustrated in Figure~\ref{fig:orientation}.
\begin{figure}[h]
\centering
\begin{tikzpicture}[scale=2]
\tikzfading[name=fade outside,inner color=transparent!30,outer color=transparent!100]
\fill[pattern={Lines[angle=90, yshift=4pt, line width=.5pt]}](0,0) -- ++(60:1) -- ++(-60:1) -- cycle;
\fill[white,path fading=fade outside](0,0) -- ++(60:1) -- ++(-60:1) -- cycle;
\draw[very thick] (0,0) node[below left=.3pt]{$0$} -- ++(60:1) node[above=2pt]{$2$} -- ++(-60:1)node[below right=.3pt]{$1$} -- cycle;
\draw[very thick,->] (0.52,.2) arc[start angle=-80, end angle=250, radius=4.5pt];
\fill (0,0) circle (.05) -- ++(60:1) circle (.05) -- ++(-60:1) circle (.05);
\draw[ultra thick,->] (1.25,0.5) -- (2.75,.5) node[midway,above]{$\partial$} ;
\draw[very thick] (3,0) node[below left=.3pt]{$0$} -- ++(60:1) node[above=2pt]{$2$} -- ++(-60:1)node[below right=.3pt]{$1$} -- cycle ;
\fill (3,0) circle (.05) -- ++(60:1) circle (.05) -- ++(-60:1) circle (.05);
\draw[very thick,->] (3,0) -- (3.5,0);
\draw[very thick,->] (4,0) -- +(120:.5);
\draw[very thick,->] ($(3,0)+(60:1)$) -- ($(3,0)+(60:.5)$);

\begin{scope}[>=Triangle]
\draw[very thick,->] (3.25,0) -- (3.25,-.5);
\draw[very thick,->] ($(4,0)+(120:.25)$) -- ($(4,0)+(120:.25)+(30:.5)$);
\draw[very thick,->] ($(4,0)+(120:1)+(-120:.25)$) -- ($(4,0)+(120:1)+(-120:.25)+(150:.5)$);
\end{scope}

\draw[dashed,->] (3.25,-.25) arc[start angle=-90, end angle=-30, radius=.25];
\draw[dashed,->] ($(4,0)+(120:.25)+(30:.25)$) arc[start angle=30, end angle=90, radius=.25];
\draw[dashed,->] ($(4,0)+(120:1)+(-120:.25)+(150:.25)$) arc[start angle=150, end angle=210, radius=.25];
\end{tikzpicture}
\caption{The orientation of the boundary.}\label{fig:orientation}
\end{figure}
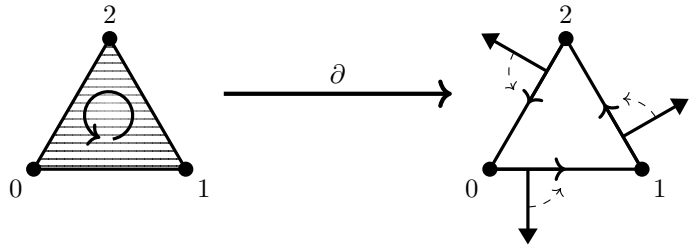
\end{defi}

These boundary morphisms satisfy the well known property that if $e^p$ is a codimension~$2$ face of some cell $e^{p+2}$ and $e^{p+1}_1,e^{p+1}_2$ denote the two codimension~1 faces of $e^{p+2}$ adjacent to $e^p$ then the two compositions of boundary morphisms $\Z(e^{p+2})\rightarrow \Z(e_1^{p+1}) \rightarrow \Z(e^p)$ and $\Z(e^{p+2})\rightarrow \Z(e_2^{p+1}) \rightarrow \Z(e^p)$ are opposite of each other. In particular, composing boundary morphisms between oriented coefficients does not define a cellular cosheaf.

\begin{defi}[Dihomologic Orientations]\label{definition:rel_or}
Let $\epsilon$ be the dihomologic pseu\-do-cell associated with an adjacent pair $e^p\leq e^q$, we define its \emph{group of oriented coefficients} to be:
\begin{equation*}
\Z(\epsilon)=\Z(e^p;e^q):=\Hom(\Z(e^p);\Z(e^q)).
\end{equation*}
We call a generator of such group an \emph{orientation} of $\epsilon$ or a \emph{relative orientation} of the pair $e^p\leq e^q$ and denote such an element by the symbol $[e^p; e^q]$. If $\epsilon'$ is a codimension~1 face of the pseudo-cell $\epsilon$, we also have a \emph{boundary morphism} $\Z(\epsilon)\rightarrow \Z(\epsilon')$ defined by the boundary morphisms between the groups of oriented coefficients of cells of $K$ and, up to sign, by the functorial properties of $\Hom$. They are shown in the diagram below.
\begin{equation}
\tag{D2}\label{diag:bound_mor}
\begin{tikzcd}
& \Z\left(e^p;e^q\right) \arrow[ld,"(-1)^{p+1}\Hom(\partial\,;\,\id)" above left] \arrow[rd,"(-1)^p\Hom(\id\,;\,\partial')" above right] & & \Z\left(e^{p+1}\right)\arrow[d,"\partial" left] & \Z\left(e^q\right) \arrow[d,"\partial'" right] \\ \Z\left(e^{p+1};e^q\right) & & \Z\left(e^p;e^{q-1}\right) & \Z\left(e^p\right) & \Z\left(e^{q-1}\right)
\end{tikzcd}
\end{equation}
\end{defi}

Note that if $\epsilon$ is a dihomologic pseudo-cell of dimension $k$, its closure can be expressed as the cone over a space that has the homology of a $(k-1)$-sphere. In this description, $\epsilon$ corresponds to the open cone. Therefore, the singular homology of $\bar{\epsilon}$ relatively to $\bar{\epsilon}\setminus\epsilon$ is isomorphic to the singular homology of a $k$-ball relatively to its boundary. Thus, we could equivalently define $\Z(\epsilon)$ to be $H_k(\bar{\epsilon};\bar{\epsilon}\setminus\epsilon;\Z)$ as in the case of a regular CW-complex. Indeed, if $e^p\leq e^q$ is the indexing pair of the pseudo-cell $\epsilon$, there is a canonical isomorphism between the groups $H_{q-p}(\bar{\epsilon};\bar{\epsilon}\setminus\epsilon;\Z)$ and $\Hom(\Z(e^p);\Z(e^q))$. If we represent a cell $e^p$ by a $p$-dimensional real vector space inside a $q$-dimensional vector space (representing $e^q$) then the dihomologic pseudo-cell $\epsilon$ indexed by $e^p\leq e^q$ represents a supplementary sub-space of the former in the latter. An orientation of such supplementary space allows by wedge product to orient the $q$-dimensional vector space from an orientation of the $p$-dimensional one. The situation is explained by the cohomological bilinear cup product:
\begin{equation*}
\cup : H^p\left(\bar{e}\,^p;\bar{e}\,^p\setminus e^p\right)\otimes H^{q-p}\left(\bar{\epsilon};\bar{\epsilon}\setminus\epsilon\right) \rightarrow H^q\left(\bar{e}\,^q;\bar{e}\,^q\setminus e^q\right).
\end{equation*}
It is non-degenerate. To see it one can notice that all the spaces considered are supports of simplicial complexes so each of these groups can be computed simplicially. A generator of $H^p(\bar{e}\,^p;\bar{e}\,^p\setminus e^p)$ is represented by the simplicial cocycle whose value is $1$ on the barycentric simplex indexed by a complete flag $e^0<\dots<e^p$ and $0$ elsewhere. Likewise, a generator of $H^{q-p}(\bar{\epsilon};\bar{\epsilon}\setminus\epsilon)$ is represented by the simplicial cocycle whose value is $1$ on the barycentric simplex indexed by a complete flag $e^p<\dots<e^q$ and $0$ elsewhere. Their cup product is the simplicial cocycle whose value is $1$ on the barycentric simplex indexed by the complete flag $e^0<\dots<e^p<\dots<e^q$ and $0$ elsewhere. It represents a generator of $H^q(\bar{e}\,^q;\bar{e}\,^q\setminus e^q)$. It gives rise to an isomorphism:
\begin{equation*}
H^{q-p}(\bar{\epsilon};\bar{\epsilon}\setminus\epsilon)\cong \Hom\Big(H^p\left(\bar{e}\,^p;\bar{e}\,^p\setminus e^p\right);H^q\left(\bar{e}\,^q;\bar{e}\,^q\setminus e^q\right)\Big).
\end{equation*}
Using the universal coefficients theorem~\cite[Theorem~3.3 p.~113]{Car-Eil_hom_alg}
three times this isomorphism becomes:
\begin{equation*}
\Z(\epsilon) \cong \Hom\Big(\Hom\Big(\Hom(\Z(e^p);\Z);\Hom(\Z(e^q);\Z)\Big);\Z\Big).
\end{equation*}
Since all the groups involved are free, we can compose it with the musical isomorphism of the trace scalar product:
\begin{multline*}
\Hom\Big(\Hom\Big(\Hom(\Z(e^p);\Z);\Hom(\Z(e^q);\Z)\Big);\Z\Big)\\
\cong \Hom\Big(\Hom(\Z(e^q);\Z);\Hom(\Z(e^p);\Z)\Big),
\end{multline*}
and then with the transposition:
\begin{equation*}
\Hom\Big(\Hom(\Z(e^q);\Z);\Hom(\Z(e^p);\Z)\Big) \cong \Hom\Big(\Z(e^p);\Z(e^q)\Big),
\end{equation*}
to finally find our desired isomorphism $\Z(\epsilon)\cong\Hom(\Z(e^p);\Z(e^q))$. Note that the signs in the Definition~\ref{definition:rel_or}, diagram~\eqref{diag:bound_mor}, of the boundary morphisms between the groups of oriented coefficients of dihomologic pseudo-cells come from this description and the relations:
\begin{equation*}
d_1\alpha\cup\beta+(-1)^p\alpha\cup d_2\beta = 0 \textnormal{ and } d_3(\alpha\cup \gamma)=0+(-1)^p\alpha\cup d_4\gamma,
\end{equation*}
that come from the graded Leibniz rule. They are satisfied by all $\alpha\in H^p(\bar{e}\,^p;\bar{e}\,^p\setminus e^p)$, $\beta\in H^{q-p-1}(\bar{\epsilon}_1;\bar{\epsilon}_1\setminus\epsilon_1)$ and $\gamma\in H^{q-p-1}(\bar{\epsilon}_2;\bar{\epsilon}_2\setminus\epsilon_2)$, where:
\begin{enumerate}
\item $\epsilon_1$ and $\epsilon_2$ are adjacent dihomologic pseudo-cells respectively indexed by the pairs $e^{p+1}\leq e^q$ and $e^p\leq e^{q-1}$;
\item $d_1$ and $d_3$ are respectively transpose of the orientation boundary morphisms $\partial_1:\Z(e^{p+1})\rightarrow\Z(e^p)$ and ${\partial_3:\Z(e^{q})\rightarrow\Z(e^{q-1})}$;
\item $d_2:H^{q-p-1}(\bar{\epsilon}_1;\bar{\epsilon}_1\setminus\epsilon_1)\rightarrow H^{q-p}(\bar{\epsilon}_3;\bar{\epsilon}_3\setminus\epsilon_3)$ and $d_4:H^{q-p-1}(\bar{\epsilon}_2;\bar{\epsilon}_2\setminus\epsilon_2)\rightarrow H^{q-p}(\bar{\epsilon}_3;\bar{\epsilon}_3\setminus\epsilon_3)$ where $\epsilon_3$ is indexed by the pair $e^p\leq e^q$.
\end{enumerate}
If $[e^p]\in\Z(e^p)$ is an orientation of $e^p$ and $[e^p;e^q]\in\Z(e^p;e^q)$ is a relative orientation, we denote by $[e^p][e^p;e^q]$ the associated orientation of $e^q$. For all pairs $e^p\leq e^{p+1}$ of relative codimension~$1$, the inverse of the boundary morphism $\partial : \Z(e^{p+1})\rightarrow\Z(e^{p})$ defines a canonical relative orientation. We will always denote it by the symbol $[e^p;e^{p+1}]$ but we should emphasise that for any other positive dimensional dihomologic pseudo-cell the similar notation denotes an arbitrary orientation, possibly subject to conditions, as there are no canonical orientation for them. We can compose relative orientations in their morphism representations, we adopt the convention $[e^p;e^q][e^q;e^r]$ to denote $[e^q;e^r]\circ[e^p;e^q]$. In these notations, one can rewrite the anti-commutativity of the boundary morphisms as follows: for all adjacent pairs $e^p\leq e^{p+2}$ of relative codimension~$2$ we have:
\begin{equation*}
\sum_{e^p\,\leq\,e^{p+1}\,\leq\,e^{p+2}}\left[e^p;e^{p+1}\right]\left[e^{p+1};e^{p+2}\right]=0.
\end{equation*}

\begin{defi}[Cellular Chain and Cochain Complexes]
Let $F$ be a cellular cosheaf on $K$ and $k\in \N$. We define the group of \emph{cellular $k$-chains} of $K$ with coefficients in $F$ to be:
\begin{equation*}
C_k(K;F):=\bigoplus_{\dim e = k} F(e)\otimes_\Z\Z(e).
\end{equation*}
Moreover, if $[e]$ is an oriented $k$-cell of $K$ and $c$ is a $k$-chain with coefficients in $F$ we denote by $\langle c,[e]\rangle$ the unique element of $F(e)$ satisfying $c_e=\langle c,[e]\rangle\otimes [e]$. We also define the \emph{boundary operator} $\partial : C_k(K;F) \rightarrow C_{k-1}(K;F)$ by the formula:
\begin{equation*}
\partial f :=\sum_{e^{k-1}\,<\,e^k} \left\langle f,[e^k]\right\rangle\big|^{e^k}_{e^{k-1}}\otimes \left[e^{k-1}\right],
\end{equation*}
for all $f\in F(e^k)\otimes_\Z\Z(e^k)$ where $[e^{k-1}]$ is the image of $[e^k]$ by the boundary morphism. This relation can be written as ${[e^{k-1}][e^{k-1};e^k]=[e^k]}$ where $[e^{k-1};e^k]$ is the canonical relative orientation. The anti-commutativity of the boundary morphisms between oriented coefficients implies that $\partial^2=0$. Thus $(C_k(K;F);\partial)_{k\,\geq\,0}$ is a chain complex.


Dually, when $F$ is a cellular sheaf, we have a \emph{cochain complex} with coefficients in $F$. For all $k\in \N$, the group of \emph{cellular $k$-cochains} is given by:
\begin{equation*}
C^k(K;F):= \prod_{\dim e = k} \Hom(\Z(e);F(e)).
\end{equation*}
For all $k$-cochains $\alpha$ with coefficients in $F$ and all oriented $k$-cells $[e]$, we denote by $\alpha[e]\in F(e)$ the value of the $e$-component of $\alpha$ evaluated at $[e]$. The \emph{coboundary operator} or \emph{differential} $\df:C^k(K;F)\rightarrow C^{k+1}(K;F)$ is given for all $k$-cochains $\alpha$ and all oriented $(k+1)$-cells $[e^{k+1}]$ by:
\begin{equation*}
\df\alpha\left[e^{k+1}\right]=\sum_{e^k\,<\,e^{k+1}}\alpha\left[e^k\right]\big|^{e^{k+1}}_{e^k},
\end{equation*}
where $[e^k]$ denotes the only orientation of $e^k$ whose image by the boundary morphism is $[e^{k+1}]$, i.e. $[e^k][e^k;e^{k+1}]=[e^{k+1}]$.


The last cochain complex we define here is the complex of \emph{cellular cochains with compact support}. It is a sub-complex of $(C^k(K;F);\df)_{k\,\geq \,0}$ whose groups are given for all $k\in \N$, by:
\begin{equation*}
C_c^k(K;F):=\bigoplus_{\dim e = k} \Hom(\Z(e);F(e)).
\end{equation*}
The image of $C_c^k(K;F)$ under $\df$ is contained in $C_c^{k+1}(K;F)$ only because we assumed $K$ to be locally finite. A morphism of cosheaves or sheaves gives rise to a morphism of chain or cochain complexes, respectively. In a more categorical language, the association $F\mapsto (C_k(K;F);\partial)_{k\,\geq\,0}$ for $F$ a cosheaf and $G\mapsto (C_{(c)}^k(K;G);\df)_{k\,\geq\,0}$ for $G$ a sheaf are covariant functors.
\end{defi}

\begin{defi}[Dihomologic chain complex]
Let $F$ be a dihomologic cellular cosheaf on $K$. We define the group of \emph{dihomologic cellular $k$-chains} of $K$ with coefficients in $F$ to be:
\begin{equation*}
\Omega_k(K;F):=\bigoplus_{\dim \epsilon = k} F(\epsilon)\otimes_\Z\Z(\epsilon),
\end{equation*}
If $[\epsilon]$ is an oriented $k$-pseudo-cell of $K$ and $c$ a $k$-chain with coefficients in $F$ we denote by $\langle c,[\epsilon]\rangle$ the unique element of $F(\epsilon)$ satisfying $c_\epsilon=\langle c,[\epsilon]\rangle\otimes [\epsilon]$. We also define the \emph{boundary operator} $\partial : \Omega_k(K;F) \rightarrow \Omega_{k-1}(K;F)$ by the formula:
\begin{equation*}
\partial \big(f\otimes[\epsilon^k]\big) :=\sum_{\epsilon^{k-1}\,<\,\epsilon^k} \left\langle f,[\epsilon^k]\right\rangle\big|^{\epsilon^k}_{\epsilon^{k-1}}\otimes \left[\epsilon^{k-1}\right],
\end{equation*}
for all $f\otimes[\epsilon^k] \in F(\epsilon^k)\otimes_\Z\Z(\epsilon^k)$, where $[\epsilon^{k-1}]$ is the image of $[\epsilon^k]$ by the boundary morphism.
\end{defi}

\begin{rema}
If the dihomologic pseudo-subdivision of $K$ is a regular subdivision of $K$ the last definition is identical to the definition of the cellular chain complex with coefficients in a cellular cosheaf on a CW-complex.
\end{rema}

\begin{defi}
Let $F$ be a cellular cosheaf (resp. a cellular sheaf). We define its \emph{homology} (resp. \emph{cohomology}, resp. \emph{cohomology with compact support}) to be the homology of its cellular chain complex (resp. cohomology of its cellular cochain complex, resp. cohomology of its cellular cochain complex with compact support). We also define the homology of a dihomologic cosheaf to be the homology of the associated chain complex.
\end{defi}

\begin{exam}
Given two sub-complexes $K_1\leq K_2$ of $K$ and an Abelian group~$G$, the homology of $[K_2;K_1;G]$ is exactly the same as the cellular homology of $K_2$ relatively to $K_1$ with coefficients in $G$.
\end{exam}

The following propositions illustrate the usefulness of these cellular constructions:

\begin{prop}
If $F$ is a sheaf of Abelian groups on $|K|$, constructible with respect to the skeletal filtration, then the sheaf cohomology of $F$ is isomorphic to the cellular cohomology of the cellular sheaf associated to $F$.
\end{prop}

The proof of this fact is easily derived from the fact that the open cover of $|K|$ by the open stars of its vertices is a Leray cover for $F$ and that the cellular cochain complex of its associated cellular sheaf is the complex of its \v Cech cochains, see~\cite[Théorème~5.10.1 p.~228]{God_top_alg}
for instance.

\begin{prop}
Whenever $F$ is the cellular cosheaf arising from a local system of coefficients $L$ on $|K|$ there is a canonical isomorphism from the cellular homology of $F$ to the singular homology of $|K|$ with coefficients in $L$.
\end{prop}

The result follows directly from the description in~\cite[Section~3.H. Local Coefficients pp.~327-337.]{Hat_alg_top}
of the complex of singular chains with coefficients in $L$ and an adaptation of the classical proof relating singular homology and cellular homology.

\begin{prop}
Let $K'$ be a subdivision of $K$ and $F$ be a cellular cosheaf on $K$. If $F'$ denotes the subdivision of $F$ we have the following injective morphism of chain complexes:
\begin{equation*}
\begin{array}{rcl}
C_k(K;F) & \longrightarrow & C_k(K';F') \\%
c & \longmapsto & \displaystyle \sum_{e^k\,\in\,K} \left(\;\sum_{\substack{e'^k\,\subset\,e^k\\ e'^k\,\in\,K'}} \left\langle c\,;\,[e^k]\right\rangle\otimes [e'^k]\; \right)
\end{array}
\end{equation*}
where the orientations are defined as follows: if $[e^k]$ is an orientation of $e^k$ then $[e'^k]$ is the image of such orientation under the isomorphism $H_k(\bar{e}\,^k;\bar{e}\,^k\setminus e^k)\rightarrow H_k(\bar{e}'^k;\bar{e}'^k\setminus e'^k)$ inverse of the map induced by the inclusion $({\bar{e}'^k\setminus e^k};\bar{e}'^k)\subset(\bar{e}\,^k\setminus e'^k;\bar{e}\,^k)$. This is a quasi-isomorphism of chain complexes.
\end{prop}

A proof of this can be found in~\cite[Theorem~1.5.2 p.~31]{She_cel_des}
in the dual context of cellular sheaves. The same proposition holds for dihomologic subdivisions of cellular cosheaves on $K$. A proof of the specific case of the constant sheaf on the dihomologic subdivision is given in~\cite[Theorem~1.2 p.~12]{For_comb_nov}.

\begin{prop}
Let $F$ be a cellular cosheaf on $K$. If $F'$ denotes the dihomologic subdivision of $F$ we have the following injective morphism of chain complexes:
\begin{align*}
C_k(K;F) & \longrightarrow  \Omega_k(K;F') \\
c & \longmapsto  \sum_{e^0\,\leq\,e^k} \left\langle c\,;\,[e^k]\right\rangle\otimes \left[e^0;e^k\right]
\end{align*}
where the orientations are defined as follows: if $[e^k]$ is an orientation of $e^k$ then $[e^0;e^k]$ is the relative orientation defined by the relation $[e^0][e^0;e^k]=[e^k]$ where $[e^0]$ is the canonical orientation of the vertex $e^0$ of $K$. This is a quasi-isomorphism of chain complexes.
\end{prop}

This is an adaptation of the proof given by R. Forman.

\begin{proof}
Let us denote by $f$ the morphism of chain complexes given in the statement of the proposition. For all $q\in\N$, we have natural inclusions of the dihomologic complexes associated with the restrictions of $F'$ to the skeleta of $K$:
\begin{equation*}
\Omega_*\left(K^{(q)};F'\right)\subset \Omega_*\left(K^{(q+1)};F'\right).
\end{equation*}
We note that for every $k\in\N$, $f$ maps $C_k(K;F)$ to $\Omega_k(K^{(k)};F')$. Now we consider the spectral sequence associated with this filtration:
\begin{equation*}
E^r_{p,q}:=\frac{Z^r_{p,q} + \Omega_{q-p}\left(K^{(q-1)};F'\right)}{\partial Z^{r-1}_{p+r-2,q+r-1}+\Omega_{q-p}\left(K^{(q-1)};F'\right)}\;,
\end{equation*}
where $Z^r_{p,q}=\{c\in \Omega_{q-p}(K^{(q)};F') \,|\, \partial c \in \Omega_{q-p-1}(K^{(q-r)};F') \}$. The boundary operator $\partial^r_{p,q}:E^r_{p,q}\rightarrow E^r_{p-r+1,q-r}$ applied to an element $c\in E^r_{p,q}$ is given by computing $\partial c'\in \Omega_{q-p-1}(K;F')$ for any lift $c'\in Z^r_{p,q}$ and then projecting to $E^r_{p-r+1,q-r}$. The term $E^0_{p,q}$ is written as:
\begin{equation*}
E^0_{p,q}=\bigoplus_{e^p\,\leq\,e^q} F'(e^p;e^q)\otimes\Z(e^p;e^q),
\end{equation*}
where $F'(e^p;e^q)=F(e^q)$ by assumption. The boundary $\partial^0_{p,q}$ acts on the element $v\otimes[e^p;e^q]$ of the group $F'(e^p;e^q)\otimes\Z(e^p;e^q)=F(e^q)\otimes\Z(e^p;e^q)$ as follows:
\begin{equation*}
\partial^0_{p,q} \left(v\otimes[e^p;e^q]\right) =\sum_{e^p\,\leq\,e^{p+1}\,\leq \,e^q}v\otimes \left[e^{p+1};e^q\right] \in \bigoplus_{e^p\,\leq\,e^{p+1}\,\leq \,e^q}F'\left(e^{p+1};e^q\right)\otimes\Z\left(e^{p+1};e^q\right),
\end{equation*}
where $[e^{p+1};e^q]$ is the image of $[e^p;e^q]$ by the boundary $\partial:\Z(e^{p};e^q)\rightarrow \Z(e^{p+1};e^q)$. This morphism is given by $(-1)^{p+1}\Hom(\partial';\id)$ with $\partial':\Z(e^{p+1})\rightarrow\Z(e^q)$. Therefore, if we decide to write $[e^p;e^q]$ in the form $\alpha\otimes[e^q]\in \Hom(\Z(e^p);\Z)\otimes \Z(e^q)$ we have $[e^{p+1};e^q]=(-1)^{p+1}(\alpha\circ\partial')\otimes[e^q]$. As a consequence, we see that the line of index $q$ of the $E^0$-page of the spectral sequence splits into the direct sum of the cellular cochain complexes of the $K(e^q)$'s with coefficients in the constant cellular sheaves $F(e^q)\otimes\Z(e^q)$ with coboundary operator $\df$ twisted by $(-1)^{p+1}$:
\begin{equation*}
\begin{tikzcd}
\cdots \longrightarrow E^0_{p,q} \ar[rr,"\partial^0"] & & E^0_{p+1,q} \longrightarrow \cdots \\%
\displaystyle \bigoplus_{e^q\,\in\,K} C^p\big(K(e^q);F(e^q)\otimes\Z(e^q)\big) \ar[rr,"\oplus(-1)^{p+1}\df" below] & & \displaystyle \bigoplus_{e^q\,\in\,K} C^{p+1}\big(K(e^q);F(e^q)\otimes\Z(e^q)\big)
\end{tikzcd}
\end{equation*}
Since the support of $K(e^q)$ is the closure of $e^q$, the $E^1$-page of the spectral sequence is concentrated on the $0^{\rm th}$ column and satisfies:
\begin{equation*}
E^1_{0,q}=\bigoplus_{e^q\,\in\,K} H^0\big(K(e^q);F(e^q)\otimes\Z(e^q)\big).
\end{equation*}
We note that $E^1_{0,q}\subset E^0_{0,q}$ and is precisely the image of $f$. It defines an isomorphism between this $0^{\rm th}$ column and $C_*(K;F)$ and therefore between the homology of $F$ and the homology of its subdivision $F'$. Moreover, we deduce that our filtration is adapted to the cokernel of $f$ and that computing the first page of the induced spectral sequence amounts to replace in our computations the cohomology groups of the $K(e^q)$'s with coefficients in the group $F(e^q)\otimes\Z(e^q)$ with their reduced cohomology groups. Since the $|K(e^q)|$'s are contractible, they all vanish. Thus, the cokernel of $f$ has trivial homology and $f$ is a quasi-isomorphism.
\end{proof}

\section{A Poincaré--Lefschetz Theorem for Dihomologic Cellular Cosheaves}

Let $K$ be a locally finite regular CW-complex and $F$ be a dihomologic cosheaf on~$K$. The chain complex $(\Omega_k(K;F);\partial)_{k\,\geq\,0}$ is actually the total complex of a bicomplex. The $k$-dimensional dihomologic pseudo-cells are represented by adjacent pairs of cells $e^p\leq e^q$ with $q-p=k$. Hence, their set is partitioned into sets of pseudo-cells of different types. If we say that a dihomologic pseudo-cell indexed by $e^p\leq e^q$ has type $(p,q)$, the set of $k$-dimensional dihomologic pseudo-cells is the disjoint union of the set of pseudo-cells of type $(p,q)$ for all $q-p=k$. Similarly the group $\Omega_k(K;F)$ splits into the following direct sum:
\begin{equation*}
\Omega_k(K;F)=\bigoplus_{q-p=k}\Omega_{p,q}(K;F),
\end{equation*}
where:
\begin{equation*}
\Omega_{p,q}(K;F):=\bigoplus_{e^p\,\leq \,e^q}F(e^p;e^q)\otimes\Z(e^p;e^q),
\end{equation*}
for all $p,q\in\N$. A codimension~$1$ face of $e^p\leq e^q$ either starts with $e^p$ or ends with $e^q$ and therefore the restriction of the boundary operator to $\Omega_{p,q}(K;F)$ takes its values in the sum $\Omega_{p+1,q}(K;F)\oplus\Omega_{p,q-1}(K;F)$. The operator $\partial$ is the sum of an operator $\partial_1$ of bidegree $(+1;0)$ and an operator $\partial_2$ of bidegree $(0;-1)$. With this additional structure, we can consider the two canonical filtrations and associated spectral sequences. We will only look at the (decreasing) horizontal filtration, that is to say, the one filtered by the index $p$. The filtering pieces are, for all $k,l\in\N$:
\begin{equation*}
\Omega_k^{(l)}(K;F):=\bigoplus_{\substack{q-p=k\\ p\,\geq\,l}}\Omega_{p,q}(K;F).
\end{equation*}
The associated spectral sequence is given for all $p,q\in\Z$ by:
\begin{equation*}
E^r_{p,q}:=\frac{Z^r_{p,q} + \Omega_{q-p}^{(p+1)}(K;F)}{\partial Z^{r-1}_{p-r+1,q-r+2}+\Omega_{q-p}^{(p+1)}(K;F)},
\end{equation*}
where $Z^r_{p,q}$ denotes $\{c\in \Omega_{q-p}^{(p)}(K;F) \,|\, \partial c \in \Omega_{q-p-1}^{(p+r)}(K;F) \}$. The boundary operator $\partial^r_{p,q}:E^r_{p,q}\rightarrow E^r_{p+r,q+r-1}$ applied to an element $c\in E^r_{l,k}$ is given by computing $\partial c'\in \Omega^{(p+r)}_{q-p-1}(K;F)$ for any lift $c'\in Z^r_{p,q}$ and then projecting to $E^r_{p+r,q+r-1}$.

\begin{prop}\label{prp:degenerate_E2_lef}
If there is an $n\in\N$ such that for every cell $e$ of $K$ the local cosheaf $F_e$ has its homology concentrated in dimension $n$ then the horizontal spectral sequence of $\Omega(K;F)$ degenerates at the second page.
\end{prop}

\begin{proof}
We have:
\begin{equation*}
E^0_{p,q}\cong \Omega_{p,q}(K;F),
\end{equation*}
and $\partial^0_{p,q}$ corresponds to the vertical component, $\partial_2$, of the total boundary operator, $\partial$, of $(\Omega_{k}(K;F);\partial)_{k\,\in\,\N}$. The following page $(E^1_{p,q})_{p,q\,\in\,\N}$ is given by the homology groups of the column complexes of the bicomplex $(\Omega_{p,q}(K;F);\partial_1;\partial_2)_{p,q\,\in\,\N}$. Let $p,q\in\N$ and let's consider the morphism:
\begin{equation*}
\Phi_{p,q} : \Omega_{p,q}(K;F) \rightarrow \bigoplus_{e^p\,\in\,K}\Hom\big(\Z(e^p);C_q(K;F_{e^p})\big),
\end{equation*}
for which the $(e^p)$-component of the image of an element $c$ is given by the linear map:
\begin{equation*}
[e^p]\mapsto (-1)^{\frac{p(p+1)}{2}+\frac{q(q+1)}{2}}\sum_{e^q\,>\,e^p}\langle c,[e^p;e^q]\rangle\otimes[e^q],
\end{equation*}
where $[e^p;e^q],[e^q]$ are some choices of orientations satisfying $[e^p][e^p;e^q]=[e^q]$ (the map does not depend on such choices). For a fixed $p\in\N$, the collection $(\Phi_{p,q})_{q\,\in\,\N}$ is almost a chain complex isomorphism between the $p^{\rm th}$ column of the bicomplex of $F$-valued dihomologic chains $(\Omega_{p,q}(K;F);\partial_2)_{q\,\geq\,0}$ and the direct sum over the $p$-cells $e^p$ of the complexes $(\Hom(\Z(e^p);C_q(K;F_{e^p}));\Hom(\Z(e^p);\partial))_{q\,\geq\,0}$. We have $\partial\Phi_{p,q}=(-1)^{q-p}\Phi_{p,q-1}\partial_2$ so $(\Phi_{p,q})_{q\,\in\,\N}$ still defines an isomorphism in homology. Each of the $\Phi_{p,q}$ being individually an isomorphism is a matter of bookkeeping. On the left hand side we sum the groups $F(e^p,e^q)\otimes \Z(e^p,e^q)$ over all ordered pairs of cells $e^p\leq e^q$ and on the right hand side we sum the groups $\Hom(\Z(e^p);F(e^p,e^q)\otimes\Z(e^q))$ over the same index set but in a different order, $\Phi_{p,q}$ then sends bijectively each one of the former summands to one of the latter summands by means of the composition:
\begin{multline*}
F(e^p,e^q)\otimes \Z(e^p,e^q)\\
=
 F(e^p,e^q)\otimes \Hom(\Z(e^p);\Z(e^q))
\arr\Hom(\Z(e^p);F(e^p,e^q)\otimes\Z(e^q)).
\end{multline*}
Let us now prove the claim about the ``commutativity'' relation with the boundary operators. Let $c$ be a chain of type $(p;q)$. On the one hand, for all $p$-cells $e^p$ of $K$, the value of the $(e^p)$-component of $\Phi_{p,q-1}(\partial_2 c)$ on an orientation $[e^p]$ is given by the following formula:
\begin{equation*}
\Phi_{p,q-1}(\partial_2c)[e^p]=(-1)^{\frac{p(p+1)}{2}+\frac{q(q-1)}{2}}\sum_{e^{q-1}\,>\,e^p}\left(\sum_{e^q\,>\,e^{q-1}} \left\langle c, [e^p;e^q]\right\rangle \Big|^{e^p,e^q}_{e^p,e^{q-1}}\right)\otimes\left[e^{q-1}\right],
\end{equation*}
where $[e^p][e^p,e^{q-1}]=[e^{q-1}]$ and $[e^p,e^q]=(-1)^p[e^p,e^{q-1}][e^{q-1},e^q]$. That is to say:
\begin{equation*}
\Phi_{p,q-1}(\partial_2c)[e^p]=(-1)^{\frac{p(p-1)}{2}+\frac{q(q-1)}{2}}\sum_{e^{q-1}\,>\,e^p}\left(\sum_{e^q\,>\,e^{q-1}} \left\langle c, [e^p;e^q]\right\rangle \Big|^{e^p,e^q}_{e^p,e^{q-1}}\right)\otimes\left[e^{q-1}\right],
\end{equation*}
where $[e^p][e^p,e^q]=[e^{q-1}][e^{q-1},e^q]$. Since we changed the equation defining the orientation, we had to compensate with the multiplication by a factor $(-1)^p$ which turned $(-1)^{\frac{p(p+1)}{2}}$ into $(-1)^{\frac{p(p-1)}{2}}$. On the other hand, the value of the $(e^p)$-component of $\partial\Phi_{p,q}(c)$ on $[e^p]$ is the same as the boundary of $\Phi_{p,q}(c)[e^p]$:
\begin{align*}
\partial\Phi_{p,q}(c)[e^p] & =(-1)^{\frac{p(p+1)}{2}+\frac{q(q+1)}{2}} \sum_{e^q\,>\,e^p} \left(\sum_{e^{q-1}\,<\,e^q} \left\langle c,[e^p,e^q]\right\rangle \Big|^{e^q}_{e^{q-1}}\otimes\left[e^{q-1}\right] \right),
\end{align*}
where $[e^p][e^p;e^q]=[e^q]$ and $[e^{q-1}][e^{q-1};e^q]=[e^q]$, i.e. $[e^p][e^p,e^q]=[e^{q-1}][e^{q-1},e^q]$. The element $\langle c,[e^p,e^q]\rangle$ of $F(e^p,e^q)$ in the last formula is understood as an element of $F_{e^p}(e^q)$ hence we have to apply the extension morphisms of $F_{e^p}$ and this is why we wrote $\langle c,[e^p,e^q]\rangle |^{e^q}_{e^{q-1}}$. However, these extension morphisms are zero whenever $e^{q-1}$ doesn't contain $e^p$ and identical to those of $F$ in the opposite case. Therefore, we can write:
\begin{align*}
\partial\Phi_{p,q}(c)[e^p] & =(-1)^{\frac{p(p+1)}{2}+\frac{q(q+1)}{2}} \sum_{e^q\,>\,e^p} \left(\sum_{e^p\,<\,e^{q-1}\,<\,e^q} \left\langle c,[e^p,e^q]\right\rangle \Big|^{e^p,e^q}_{e^p,e^{q-1}}\otimes\left[e^{q-1}\right] \right) \\
& = (-1)^{\frac{p(p+1)}{2}+\frac{q(q+1)}{2}} \sum_{e^{q-1}\,>\,e^p}\left(\sum_{e^q\,>\,e^{q-1}} \left\langle c, [e^p;e^q]\right\rangle \Big|^{e^p,e^q}_{e^p,e^{q-1}}\right)\otimes\left[e^{q-1}\right] \\ & = (-1)^{q-p}\Phi_{p,q-1}(\partial_2c)[e^p].
\end{align*}
For each $p$-cell, the group $\Z(e^p)$ is free so the Universal Coefficient Theorem ensures that $E^1_{p,q}$ is isomorphic to the direct sum of the $\Hom(\Z(e^p);H_q(K;F_{e^p}))$'s. By assumptions, $H_q(K;F_{e^p})$ is trivial as soon as $q$ is not $n$. As a consequence, all pages following $E^1$ are concentrated on the horizontal line $\{q=n\}$. Since $\partial^r_{p,q}$ has bidegree $(r,r-1)$, the spectral sequence degenerates at the second page.
\end{proof}

\begin{defi}\label{def:loc_hom_shea}
Let $F$ be a dihomologic cosheaf of a regular CW-complex $K$ and $q$ be a non-negative integer. The $q^{\rm th}$ \emph{sheaf of local homology} of $F$ is denoted by $H_q(F_*)$. It is a cellular sheaf on $K$ and it is given by the groups:
\begin{equation*}
H_q(F_*):e^p\mapsto H_q(K;F_{e^p}).
\end{equation*}
Its restriction morphisms are induced in homology by the localisation morphisms $F_{e^{p_1}}\rightarrow F_{e^{p_2}}$ for all cells $e^{p_1}\leq e^{p_2}$.
\end{defi}

\begin{theo}[Cellular Poincaré--Lefschetz Theorem]\label{thm:cell_poinca_lef}
Let $K$ be a finite dimensional, locally finite and regular CW-complex, $n$ be a non-negative integer, and $F$ be a dihomologic cosheaf on $K$ whose sheaves of local homology $H_q(F_*)$ vanish for all $q\neq n$. Then for all integers $k$, $H_k(K;F)$ and $H_c^{n-k}(K;H_n(F_*))$ are canonically isomorphic. In particular, $H_k(K;F)$ vanishes for all $k>n$. If in addition $K$ has dimension $n$, then this isomorphism comes from an injective quasi-isomorphism:
\begin{equation*}
\big(C^{n-k}_c\big(K;H_n(F_*)\big);\df\big)_{k\,\geq\,0}\rightarrow \big(\Omega_k(K;F);\partial\big)_{k\,\geq\,0}.
\end{equation*}
\end{theo}

\begin{proof}
Let $d$ denote $\dim K$. We consider the horizontal filtration of the homology of $F$:
\begin{equation*}
0\subset H_k(K;F)^{(d)} \subset \cdots \subset H_k(K;F)^{(n-k)} \subset \cdots \subset H_k(K;F)^{(0)},
\end{equation*}
whose graded pieces:
\begin{equation*}
E^\infty_{d,k+d} \hspace{0.5cm} \cdots \hspace{0.5cm} E^\infty_{n-k,n} \hspace{0.5cm} \cdots \hspace{0.5cm} E^\infty_{0,k},
\end{equation*}
satisfy, in light of the last proposition, $E^\infty_{p,q}=E^2_{p,q}=0$ as soon as $q$ is different from $n$. Therefore, for all integers $k$, $H_k(K;F)=H_k(K;F)^{(n-k)}=E^\infty_{n-k,n}=E^2_{n-k,n}$. Now because of the isomorphisms $(\Phi_{p,q})_{p,q\,\in\,\N}$ given in the last proof, we recognise that $E^1_{p,q}\cong C^p_c(K;H_q(F_*))$. Then it remains only to show that the boundary operator $\partial^1_{p,q}$ of the spectral sequence is mapped to the coboundary operator $\df$. If $c'\in\Omega_{p,q}(K;F)$ is a $\partial_2$-cycle representing an element $c\in E^1_{p,q}$ then $\partial_1 c'$ is a $\partial_2$-cycle representing $\partial^1_{p,q} c$. We have:
\begin{equation*}
\left\langle \partial_1 c',[e^{p+1};e^q]\right\rangle = \sum_{e^p\,<\,e^{p+1}} \left\langle c',[e^{p};e^q]\right\rangle \Big|^{e^p,e^q}_{e^{p+1};e^q},
\end{equation*}
where $[e^{p};e^q]=(-1)^{p+1}[e^p;e^{p+1}][e^{p+1};e^q]$. Hence the image of $\partial_1 c'$ under $\Phi$ satisfies:
\begin{equation*}
\Phi_{p+1,q}(\partial_1 c')\left[e^{p+1}\right]=(-1)^{\frac{(p+2)(p+1)}{2}+\frac{q(q+1)}{2}} \sum_{e^q\,>\,e^{p+1}}\left(\sum_{e^p\,<\,e^{p+1}} \left\langle c',[e^{p};e^q]\right\rangle \Big|^{e^p,e^q}_{e^{p+1};e^q} \right)\otimes[e^q],
\end{equation*}
where $[e^{p+1}][e^{p+1};e^q]=[e^q]$ and $[e^{p};e^q]=(-1)^{p+1}[e^p;e^{p+1}][e^{p+1};e^q]$. Which can be written as:
\begin{equation*}
\begin{split}
\Phi_{p+1,q}(\partial_1 c')\left[e^{p+1}\right] & =(-1)^{\frac{p(p+1)}{2}+\frac{q(q+1)}{2}}\sum_{e^p\,<\,e^{p+1}}\left(\sum_{e^q\,>\,e^{p+1}} \left\langle c',[e^{p};e^q]\right\rangle \Big|^{e^p,e^q}_{e^{p+1};e^q} \otimes[e^q]\right) \\
& = (-1)^{\frac{p(p+1)}{2}+\frac{q(q+1)}{2}}\sum_{e^p\,<\,e^{p+1}} \Psi_{e^p}^{e^{p+1}} \left(\sum_{e^q\,>\,e^{p}} \left\langle c',[e^{p};e^q]\right\rangle \otimes[e^q]\right),
\end{split}
\end{equation*}
where $[e^{p+1}][e^{p+1};e^q]=[e^q]$, $[e^{p};e^q]=[e^p;e^{p+1}][e^{p+1};e^q]$, and $\Psi_{e^p}^{e^{p+1}}$ is the chain complex morphism associated with the cosheaf morphism $F_{e^p}\rightarrow F_{e^{p+1}}$. Therefore, $\Phi_{p+1,q}(\partial^1_{p,q} c)$ is the $(p+1)$-cochain with compact support that associates to an oriented cell $[e^{p+1}]$ the sum, over its codimension~1 faces, of the images in $H_q(K;F_{e^{p+1}})$ of the homology classes $[(-1)^{\frac{p(p+1)}{2}+\frac{q(q+1)}{2}} \sum_{e^q\,>\,e^{p}} \langle c',[e^{p};e^q]\rangle \otimes[e^q]]\in H_q(K;F_{e^{p}})$ where $[e^{p};e^q]=[e^p;e^{p+1}][e^{p+1};e^q]$ and $[e^{p};e^q]=[e^p;e^{p+1}][e^{p+1};e^q]$. Finally, this is precisely $\df\Phi_{p,q}(c)$ for:
\begin{equation*}
\df\Phi_{p,q}(c)[e^{p+1}]=\sum_{e^p\,<\,e^{p+1}}\Phi_{p,q}(c)[e^p],
\end{equation*}
where $[e^{p}][e^{p};e^{p+1}]=[e^{p+1}]$ and $\Phi_{p,q}(c)[e^p]=[(-1)^{\frac{p(p+1)}{2}+\frac{q(q+1)}{2}} \sum_{e^q\,>\,e^{p}} \langle c',[e^{p};e^q]\rangle \otimes[e^q]]$.

To prove the second part of the statement, when $\dim K=n$, we need to remember that $\Phi$ provided us with an isomorphism between the chain complexes $(E^1_{n-k,n};\partial^1)_{k\,\geq\,0}$ and $(C^{n-k}_c(K;H_n(F_*));\df)_{k\,\geq\,0}$. Also, in the special context of spectral sequences of bicomplexes we know that $E^1_{n-p,n}$ is the $n^{\rm th}$ homology group of the $(n-k)^{\rm th}$ column of $(\Omega_{p,q}(K;F);\partial_1;\partial_2)_{p,q\,\geq\,0}$. With the dimensional assumption, this is the homology group of highest dimension and therefore the same as the group of cycles. We have an inclusion of $E^1_{n-k,n}$ in $\Omega_{n-k,n}(K;F)$ as the kernel of the vertical part of the boundary operator, namely $\partial_2$. With this description, it is clearly an injective morphism of chain complexes whose cokernel inherits a bicomplex structure. By construction, all the columns of this bicomplex are exact and so is its total complex. Our injective quasi-isomorphism is then the composition of $\Phi$ with the inclusion of $(E_{n-k,n};\partial^1)_{k\,\geq\,0}$ in the total complex $(\Omega_k(K;F);\partial)_{k\,\geq\,0}$.
\end{proof}

\begin{rema}
In the proof of Theorem~\ref{thm:cell_poinca_lef} we did not actually used that $c'$ was a $\partial_2$-cycle and have actually proved that $\Phi$ is a bicomplex isomorphism. Indeed, we have $(\sum_{q-p=k}\Phi_{p,q})\circ(\partial_1+\partial_2)=(\df+(-1)^{q-p}\partial)\circ(\sum_{q-p=k+1}\Phi_{p,q})$. So if we no longer assume the vanishing hypotheses on the local homology of $F$ what we get instead is a complex of cellular sheaves $e\mapsto C^*(K;F_e)$ whose cohomology (or hypercohomology) with compact support corresponds to the homology of $F$.
\end{rema}

A direct consequence of the last corollary is the already known Poincaré-Lefschetz theorem:

\begin{coro}
If $X$ is a homology $n$-manifold in the sense of Definition~\ref{definition:homo_man} then $H_k(X;\Z)\cong H^{n-k}_c(X;\partial X; o_\Z)$ where $o_\Z$ denotes the system of local orientations defined on $X\setminus\partial X$ by $x\mapsto H_n(X;X-x;\Z)$.
\end{coro}

\begin{proof}
We consider the local system $o_\Z$ on $X\setminus \partial X$ given by:
\begin{equation*}
x\mapsto H_n(X;X\setminus\{x\};\Z).
\end{equation*}
We denote by $K$ a locally finite, regular, finite dimensional CW-complex whose support is $X$. We have, by hypotheses, that for every cell $e$ of $K$, $H_k(K;K-e;\Z)$ vanishes as soon as $k$ does not equal $n$. Then by Theorem~\ref{thm:cell_poinca_lef} we have:
\begin{equation*}
H_k(K;\Z)\cong H^{n-k}_c\big(X;H_n(\Z_*)\big).
\end{equation*}
The cellular sheaf $H_n(\Z_*)$ vanishes on the boundary and corresponds to the local system $o_\Z$ elsewhere, hence the corollary follows.
\end{proof}

An interesting corollary is a version of Serre duality for flat vector bundles over a field $\F$ where $o_\F=o_\Z\otimes\F$ plays the role of the canonical line bundle:

\begin{coro}
Let $X$ be a homology $n$-manifold in the sense of Definition~\ref{definition:homo_man}, $\F$ be a field, and $E$ be a flat bundle of $\F$-vector spaces of finite rank over $X$. For all integers $k$, we have:
\begin{equation*}
H^k(X;E)\cong \big(H^{n-k}_c(X;\partial X; o_\F\otimes_\F E^*)\big)^*.
\end{equation*}
\end{coro}

\begin{proof}
Let $K$ denote a locally finite, regular, finite dimensional CW-complex whose support is $X$. By the Universal Coefficients Theorem (c.f. \cite[Theorem~3.3 p.~113]{Car-Eil_hom_alg}) 
we have, after noticing that $(C^k(K;E);\df)_{k\geq 0}$ is dual to the complex $(C_k(K;E^*);\partial)_{k\,\geq\,0}$, that $H^k(X;E)\cong(H_k(X;E^*))^*$. Let $e$ be a cell of $K$, we define an isomorphism of cosheaves $\phi_e:[K;K-e;E^*(e)]\rightarrow E^*_e$ given by the inverses of the extension morphisms, i.e. its restriction morphisms, $E^*(e')\rightarrow E^*(e)$ for all cells $e'\geq e$. This being done, we have the isomorphism:
\begin{equation*}
H_k(K;K-e;E^*)\cong H_k(K;K-e;E^*(e)) \cong H_k(K;K-e;\F)\otimes_\F E^*(e).
\end{equation*}
By hypotheses these groups are all $0$ as soon as $k$ does not equal $n$, hence the dihomologic pseudo-subdivision of the cosheaf $E^*$ satisfies the conditions of Proposition~\ref{prp:degenerate_E2_lef}. Therefore, Theorem~\ref{thm:cell_poinca_lef} applies. In addition, for all $e^p\leq e^q$, we have the following commutative square:
\begin{equation*}
\begin{tikzcd}
\big[K;K-e^p;E^*(e^p)\big]\arrow[d] \arrow[r,"\phi_{e^p}"] & E^*_{e^p} \arrow[d,"\text{loc. at }e^q" right] \\%
\big[K;K-e^q;E^*(e^q)\big]\arrow[r,"\phi_{e^q}" below] & E^*_{e^q}
\end{tikzcd}
\end{equation*}
where the unlabelled morphism is given by the tensor product of the extension $E^*(e^p)\rightarrow E^*(e^q)$ with the localisation morphism $[K;K-e^p;\F]\rightarrow [K;K-e^q;\F]$. It appears that the cellular sheaf $H_n(E^*_*)$ is isomorphic to the tensor product ${H_n(\F_*)\otimes_\F E^*}$. By Theorem~\ref{thm:cell_poinca_lef}, we get:
\begin{equation*}
H^k(X;E)\cong \Big(H^{n-k}_c\big(X;H_n(E^*_*)\big)\Big)^* \cong \Big(H^{n-k}_c(X; H_n(\F_*)\otimes_\F E^*)\Big)^*.
\end{equation*}
The sheaf $H_n(\F_*)$ vanishes on $\partial X$ and its restriction to $X\setminus \partial X$ is given by the local system $o_\F$, so finally:
\begin{equation*}
H^k(X;E) \cong \Big(H^{n-k}_c(X;\partial X; o_\F\otimes_\F E^*)\Big)^*.\qedhere
\end{equation*}
\end{proof}

\section{Application to Tropical Homology: Lefschetz Hyperplane Section Theorem}

In this section we apply Theorem~\ref{thm:cell_poinca_lef} to the cosheaves arising from tropical geometry. In a first paragraph, we will state and prove three preliminary lemmata. Then, we will define the objects at play. Finally, we will state and prove Theorem~\ref{thm:Lefschetz_hyp_sec}.

\subsection{Three Lemmata}

\begin{defi}\label{definition:resolution_ext_alg}
Let $V$ be an $R$-module, $G\subset V$ be a finite subset, and $p\geq 0$ be an integer. We define the complex of $R$-modules $C(V;G;p):= (C_k,\partial_k)_{k\,\geq\,0}$ as follows:
\begin{equation*}
C_k:= \bigoplus_{\substack{F\,\subset\,G \\
|F|=k}} \bigwedge^{p-k} V_F\,,
\end{equation*}
where $V_F$ denotes the quotient of $V$ by the sub-module spanned by $F$. The boundary operator $\partial_k:C_k\rightarrow C_{k-1} $ is the sum of the following maps:
\begin{equation*}
\begin{array}{rcl}
\displaystyle \bigwedge^{p-k} V_{F} & \longrightarrow & \displaystyle\bigoplus_{f\,\in\,F}\bigwedge^{p-(k-1)} V_{F\setminus f} \\
v &\longmapsto & \displaystyle \sum_{f\,\in\,F} f\wedge v\,.
\end{array}
\end{equation*}
Because of the antisymmetry of the wedge product, $\partial^2$ vanishes. Moreover, it is worth noticing that when $G'$ is a subset of $G$ the complex $C(V;G';p)$ is naturally a sub-complex of $C(V;G;p)$.
\end{defi}

\begin{lemm}\label{lemm:resolution_ext_alg}
Let $V$ be a free $R$-module of finite rank and $p\geq 0$ be an integer. If $G$ is a linearly independent finite subset of $V$ spanning a free summand of $V$, then the only non-trivial homology group of $C(V;G;p)$ is in dimension~0. Moreover, this $H_0$ is equal to the $p^{\rm th}$ exterior power of $V_G$. In other words:
\begin{equation*}
0\leftarrow \bigwedge^{p}V_G \leftarrow C_0 \leftarrow\dots \leftarrow C_{|G|}\leftarrow 0\,,
\end{equation*}
is a free resolution of $\bigwedge^{p}V_G$. (The augmentation morphism is the reduction modulo the module spanned by $G$.)
\end{lemm}

\begin{proof}
By hypotheses, one can find a subset $G'\subset V$ disjoint from $G$ such that $G\cup G'$ is a basis of $V$. Let us denote by $g_1,\dots,g_n$ (resp. $g'_1,\dots,g'_{n'}$) the elements of $G$ (resp. $G'$). Let $F$ be a subset $\{g_i:i\in I\}\subset G$ where $|I|$ equals $k$. A basis of the $(p-k)^{\rm th}$ exterior power of $V_F$ is given by the elements $(g_{P\setminus I}\wedge g'_Q)$ for all ${I\subset P\subset \{1,\dots,n\}}$, and all ${Q\subset \{1,\dots,n'\}}$ such that $|P|+|Q|=p$. Moreover, the image of the generator $g_{P\setminus I}\wedge g'_Q$ under the boundary map is:
\begin{equation*}
\partial (g_{P\setminus I}\wedge g'_Q) = \sum_{i\,\in\,I} g_i\wedge g_{P\setminus I}\wedge g'_Q \in \bigoplus_{i\,\in\,I}\bigwedge^{p-(k-1)} V_{F\setminus g_i}\,.
\end{equation*}
Therefore, if we fix $P$ and $Q$ with $|P|+|Q|=p$ and see $P$ as an abstract simplex, we have an injective morphism of chain complexes from the reduced simplicial chain complex of $P$ to $C(V;G;p)$ given by:
\begin{equation*}
\begin{array}{rcl}
\widetilde{C}_k(P;R) & \longrightarrow & C_{k+1}\\
I & \longmapsto & g_{P\setminus I}\wedge g'_Q\,.
\end{array}
\end{equation*}
By construction, $C(V;G;p)$ is the direct sum of the images of these complexes, hence it only has homology in dimension~0. Furthermore, the only summands that contribute are the ones for which $P$ is empty: there is exactly one free summand of rank 1 for every basis element of $\bigwedge^{p}V_G$, although a very quick computation shows that $B_0=\langle G\rangle \wedge \bigwedge^{p-1}V$.
\end{proof}

\begin{lemm}[Homology Shift]\label{lemm:hom_shi}
Let $n\geq 0$ be an integer and:
\begin{equation}\label{eq:hypo_hom_shi}
\tag{S}
\begin{tikzcd}[cramped]
0 \ar[r] & C(n) \ar[r,"a_n"] & C(n-1) \ar[r,"a_{n-1}"] & \cdots \ar[r,"a_2"] & C(1) \ar[r,"a_1"] & C(0) \ar[r] & 0\,,
\end{tikzcd}
\end{equation}
be an exact sequence of chain complexes. If $C(i)$ is exact for all $1 \leq i \leq n-1$, then the homology of $C(0)$ is the one of $C(n)$ shifted by $1-n$:
\begin{equation*}
H(0)=H(n)[1-n]\,.
\end{equation*}
\end{lemm}

\begin{proof}
For all $1 \leq i \leq n$, let $A(i)$ denote the kernel complex of $a_i$. Since~\eqref{eq:hypo_hom_shi} is an exact sequence, all the columns of the following diagram are exact:
\begin{equation*}
\begin{tikzcd}[column sep=14]
& & 0 \ar[d] & & 0 \ar[d] & & 0 \ar[d] & 0 \ar[d] \\%
& 0 \ar[d] & A(n-1) \ar[d] & \cdots & A(i) \ar[d] & \cdots & A(1) \ar[d] & A(0) \ar[d] \\%
0 \ar[r] & C(n) \ar[d,"a_n"] \ar[r,"a_n" below] & C(n-1) \ar[d,"a_{n-1}"] \ar[r,"a_{n-1}" below] & \cdots \ar[r,"a_{i+1}" below] & C(i) \ar[d,"a_i"] \ar[r,"a_i" below] & \cdots \ar[r,"a_2" below] & C(1) \ar[d,"a_1"] \ar[r,"a_1"below] & C(0) \ar[d] \ar[r] & 0 \\%
& A(n-1) \ar[d] & A(n-2) \ar[d] & \cdots & A(i-1) \ar[d] & \cdots & A(0) \ar[d] & 0 \\%
& 0 & 0 & & 0 & & 0
\end{tikzcd}
\end{equation*}
By assumption, $H(i)$ vanishes for all $1\leq i\leq n-1$. Hence, by the homological long exact sequence of the $i^{\rm th}$ column, $H(A(i-1))$ equals $H(A(i))[-1]$. A finite recursion implies that:
\begin{equation*}
H(0)=H(A(0))=H(A(n-1))[1-n]=H(n)[1-n]\,.\qedhere
\end{equation*}
\end{proof}

\begin{lemm}[Double Localisation]\label{lemm:double_loc}
Let $F$ be a cellular cosheaf on a regular CW-complex $K$, and $K'$ be a subdivision of $K$. If $e$ is a cell of $K$ and $e'$ is a cell of $K'$ contained in $e$ then the local homology of $F$ at $e$ is identical to the local homology of the subdivision of $F$ at $e'$. More precisely, if $F'$ is the subdivision of $F$ there is a canonical quasi-isomorphism between $(C_k(K\,; F_e))_{k\,\geq\,0}$ and $(C_k(K' ;F'_{e'}))_{k\,\geq\,0}$.
\end{lemm}

\begin{proof}
We denote by $(F_e)'$ the subdivision of $F_e$. We have the following commutative square of cellular cosheaves on $K'$:
\begin{equation*}
\begin{tikzcd}
F' \ar[r] \ar[d] & (F')_{e'} \ar[d] \\%
(F_e)' \ar[r] & (F_e)'_{e'}
\end{tikzcd}
\end{equation*}
where all four morphisms come from localisations and subdivisions. The two cosheaves on the right hand column are actually equal and the arrow is rigorously the identity between them. Indeed, every cell $\tilde{e}'$ of $K'$ that contains $e'$ has to be contained in a cell $\tilde{e}$ of $K$ that contains $e$. Hence, $(F')_{e'}(\tilde{e}')$ equals $F(\tilde{e})$. We have the following commutative diagram of chain complexes:
\begin{equation*}
\begin{tikzcd}
& & \big(C_k(K' ;(F')_{e'}\big)_{k\,\geq\,0} \arrow[d,equal] \\%
\big(C_k(K\,;F_e)\big)_{k\,\geq\,0}\arrow[r,"\text{subd.}" below,"f" above] \ar[rru,"h",bend left=10] & \big(C_k(K';(F_e)')\big)_{k\,\geq\,0} \ar[r,"\text{loc.}" below,"g" above] & \big(C_k(K';(F_e)'_{e'})\big)_{k\,\geq\,0}
\end{tikzcd}
\end{equation*}
where $h$ is the canonical quasi-isomorphism announced in the statement of the proposition. Since $f$ is a subdivision morphism it is automatically a quasi-isomorphism. We only need to show that the localisation morphism $g$ is a quasi-isomorphism. This is a surjective morphism since localisations are quotients. By definition, we have the following short exact sequence of cosheaves:
\begin{equation*}
0\rightarrow (F_e)'\otimes \big[K'-e';\Z\big] \rightarrow (F_e)' \rightarrow (F_e)'_{e'} \rightarrow 0\,.
\end{equation*}
Hence $g$ is a quasi-isomorphism if and only if the kernel cosheaf has trivial homology. The cosheaf $(F_e)'$ equals $F'\otimes[K';K'-e;\Z]$ for $K'-e$ is the subdivision of the sub-complex $K-e$. Consequently:
\begin{equation*}
(F_e)'\otimes \big[K'-e';\Z\big]=F'\otimes\big[K';K'-e;\Z\big]\otimes \big[K'-e';\Z\big]=F'\otimes\big[K'-e';K'-e;\Z\big]\,.
\end{equation*}
By a process similar to excision, we see that:
\begin{equation*}
\big[K'-e';K'-e;\Z\big]=\big[K'(e)-e';K'(e)-e;\Z\big]\,.
\end{equation*}
Indeed, $K'(e)$ is the smallest sub-complex of $K'$ containing $e$ in its support and any cell not in it falls into $K'-e$. We note that $[K'(e)-e';K'(e)-e;\Z]$ can only be non-zero on the cells of $K'$ contained in $e$. Therefore:
\begin{equation*}
(F_e)'\otimes \big[K'-e';\Z\big]= F'\otimes\big[K'-e';K'-e;\Z\big] = \big[K'(e)-e';K'(e)-e;F(e)\big]\,.
\end{equation*}
Now because $K$ is regular, $|K'(e)|$ is a closed ball $B$, $|K'(e)-e|$ is its boundary $\partial B$, and $|K'(e)-e'|$ sits somewhere between $\partial B$ and $B$ punctured at one point. Since the latter retracts by deformation on the former, the homology of ${[K'(e)-e';K'(e)-e;F(e)]}$ is trivial.
\end{proof}

\subsection{Hodge Theory and Tropical Homology}

Let $P$ be a full dimensional integral polytope\footnote{The convex hull of a finite number of vertices.} in a finite dimensional real vector space $\t^*(\R)$ endowed with a lattice $\t^*(\Z)$. Its corresponding toric variety is a projective algebraic variety defined over the integers. The tropical locus $Y$ of such a toric variety can be seen as a compactification of the tropical torus $\t(\R)=\Hom_\R(\t^*(\R);\R)$. Moreover, the moment map provides an isomorphism between $Y$ and the polytope $P$ itself. The real or complex toric variety defined by an integral polytope comes equipped with an ample line bundle. The space of global sections of this line bundle is naturally isomorphic to the vector space of Laurent polynomials whose exponents are integer points of $P$. The ``tropical sections'' of this line bundle are likewise defined as tropical Laurent polynomials whose exponents are integer points of $P$ i.e. the convex piecewise affine functions of the following form:
\begin{equation*}
\begin{array}{rcl}
f:\t(\R) & \longrightarrow & \T=\R\cup\{-\infty\}\\%
v & \longmapsto & \max\limits_{\alpha\,\in \,P\,\cap\,\t^*(\Z)}(a_\alpha + \alpha(v))\;,
\end{array}
\end{equation*}
where the $a_\alpha$'s are tropical numbers. The tropical hypersurface $X$ of $Y$ defined by this equation is the topological closure of the non-differentiability locus of $f$ in $Y$. As usual the Newton polytope of $f$ is the convex hull of the $\alpha\in P\cap\t^*(\Z)$ whose associated coefficient $a_\alpha$ is different from $-\infty$. By~\cite[Theorem~2.3.7 p.~44]{Mik-Rau_tro_geo} of G.~Mikhalkin and J.~Rau,
any tropical hypersurface $X$ of $Y$ defined by an equation $f$ whose Newton polytope is $P$, is dual to an integer convex polyhedral subdivision $K$ of $P$. This means that $X$ is homeomorphic to the sub-complex of the dihomologic subdivision\footnote{Here $K$ is polyhedral so its dihomologic pseudo-subdivision is an actual regular subdivision of~$K$.} of $K$ consisting of the union of the closed dihomologic cells indexed the adjacent pairs $e^1\leq e^p$. The situation is illustrated in Figure~\ref{fig:trop_hyp}. The theory of tropical homology defined by I. Itenberg, L. Katzarkov, G. Mikhalkin, and I. Zharkov for both $X$ and $Y$ can be expressed as the homology of some dihomologic cosheaves on $K$. Moreover, the data of the polyhedral subdivision $K$ alone is enough to define these cosheaves. We will adopt this point of view and state our result in terms of cosheaves associated to an integral polyhedral subdivision $K$ of an integral polytope~$P$. As noted by E. Brugallé, L. L\'opez de Medrano and J.~Rau in~\cite{Brug-LdM-Rau_Comb_pac}, the tropical cosheaves can be associated to any integral polyhedral subdivision regardless of its convexity and most of the results about the tropical homology of tropical hypersurfaces apply to them. Theorem~\ref{thm:Lefschetz_hyp_sec} is no exception to that observation. We will not assume the subdivision to be convex and therefore the theorem will not be stated in the framework of tropical hypersurfaces.
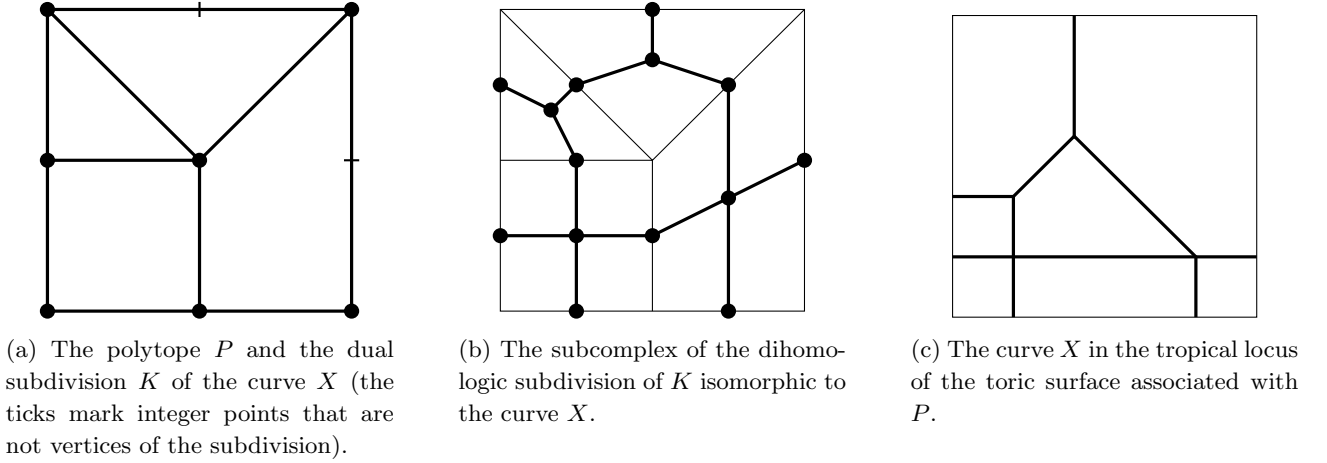
\begin{figure}[h]
\centering
\begin{subfigure}[t]{0.3\textwidth}
\centering
\begin{tikzpicture}[scale=2,very thick]
\fill (0,0) circle (.05);
\fill (1,0) circle (.05);
\fill (2,0) circle (.05);
\fill (0,1) circle (.05);
\fill (1,1) circle (.05);
\fill (0,2) circle (.05);
\fill (2,2) circle (.05);
\draw[thick](1.95,1) -- (2.05,1);
\draw[thick](1,1.95) -- (1,2.05);
\draw (0,0) rectangle (2,2);
\draw (0,1) -- (1,1) -- (1,0);
\draw (0,2) -- (1,1) -- (2,2);
\end{tikzpicture}
\caption{The polytope $P$ and the dual subdivision $K$ of the curve~$X$.}
\end{subfigure}
\hfill
\begin{subfigure}[t]{0.3\textwidth}
\centering
\begin{tikzpicture}[scale=2]
\fill (1,2) circle (.05);
\fill (2,1) circle (.05);
\draw[thin] (0,0) rectangle (2,2);
\draw (0,1) -- (1,1) -- (1,0);
\draw (0,2) -- (1,1) -- (2,2);
\fill (.5,.5) circle (.05);
\fill ($1/3*(0,1)+1/3*(1,1)+1/3*(0,2)$) circle (.05);
\fill ($1/3*(2,2)+1/3*(1,1)+1/3*(0,2)$) circle (.05);
\fill ($1/4*(1,0)+1/4*(2,0)+1/4*(2,2)+1/4*(1,1)$) circle (.05);
\fill ($1/2*(0,1)+1/2*(1,1)$) circle (.05);
\fill ($1/2*(0,2)+1/2*(1,1)$) circle (.05);
\fill ($1/2*(2,2)+1/2*(1,1)$) circle (.05);
\fill ($1/2*(1,0)+1/2*(1,1)$) circle (.05);
\fill (.5,0) circle (.05);
\fill (1.5,0) circle (.05);
\fill (0,1.5) circle (.05);
\fill (0,.5) circle (.05);

\begin{scope}[very thick]
\draw (0,.5) -- (.5,.5) -- (1,.5) -- ($1/4*(1,0)+1/4*(2,0)+1/4*(2,2)+1/4*(1,1)$) -- (2,1);
\draw (.5,0) -- (.5,1) -- ($1/3*(0,1)+1/3*(1,1)+1/3*(0,2)$) -- ($1/2*(0,2)+1/2*(1,1)$);
\draw ($1/2*(0,2)+1/2*(1,1)$) -- ($1/3*(2,2)+1/3*(1,1)+1/3*(0,2)$) -- ($1/2*(2,2)+1/2*(1,1)$) -- (1.5,0);
\draw ($1/3*(0,1)+1/3*(1,1)+1/3*(0,2)$) -- (0,1.5);
\draw ($1/3*(2,2)+1/3*(1,1)+1/3*(0,2)$) -- (1,2);
\end{scope}
\end{tikzpicture}
\caption{The subcomplex of the dihomologic subdivision of $K$ isomorphic to the curve~$X$.}
\end{subfigure}
\hfill
\begin{subfigure}[t]{0.3\textwidth}
\centering
\begin{tikzpicture}[scale=2,very thick]
\draw[thin] (0,0) rectangle (2,2);
\draw (0,.4) -- (2,.4);
\draw (0.4,0) -- (.4,.8) -- ++(.4,.4) -- ++(.8,-.8) -- ++(0,-.4);
\draw (0,.8) -- (.4,.8);
\draw (.8,1.2) -- (.8,2);
\end{tikzpicture}
\caption{The curve $X$ in the tropical locus of the toric surface associated to $P$.}
\end{subfigure}
\caption{A singular curve $X$ of bidegree $(2\,;\,2)$ in the tropical locus of $\bbP^1\times\bbP^1$.}\label{fig:trop_hyp}
\end{figure}

\begin{nota}\leavevmode
\begin{enumerate}
\item $\t^*(\Z)$ is a lattice of finite rank $n\in\N$ with dual lattice $\t(\Z)$;
\item Let $R$ be a commutative ring with unit, $\t^*(R)$ (resp. $\t(R)$) denotes the associated free $R$-module $\t^*(\Z)\otimes R$ (resp. $\t(\Z)\otimes R$);
\item Let $V$ be a subspace of $\t(\R)$ (or $\t^*(\R)$) that is rational with respect to the lattice $\t(\Z)$ (resp. $\t^*(\Z)$). We denote its lattice $V\cap\t(\Z)$ (resp. $V\cap\t^*(\Z)$) by $V(\Z)$;
\item $P$ is a full dimensional polytope of $\t^*(\R)$ whose vertices lie in the lattice $\t^*(\Z)$;
\item $K$ is an integral polyhedral subdivision of $P$, i.e. a polyhedral subdivision with $K^{(0)}\subset \t^*(\Z)$. (Note that every cell $e^q$ of $K$ is the relative interior of a $q$-dimensional polytope whose tangent space $Te^q$ is rational relatively to $\t^*(\Z)$, i.e. $Te^q(\Z)$ is free of rank $q$ and, in particular, the quotient of $\t^*(\Z)$ by this sub-group is free of rank $n-q$);
\item We denote by $X$ the \emph{dual hypersurface} of $K$. It is made of all the dihomologic cells of type $(p;q)$ with $p$ at least equal to $1$.
\end{enumerate}
\end{nota}

\begin{defi}[{Tropical Cosheaves, \cite[Definition~13 p.~10]{Ite-Kat-Mik-Zha_tro_hom}.}]\label{definition:trop_cosheaves}\leavevmode
\begin{enumerate}
\item The first cosheaf we define is called the \emph{sedentarity}\index{Tropical Cosheaves!Sedentarity@$\Sed$}. It represents the stabilisators of the action of the different loci of the algebraic torus of the toric variety of $P$. We denote it by $\Sed$. It is defined on the CW-complex associated with the polytope $P$. If $Q$ is a face of $P$ we set:
\begin{equation*}
\Sed(Q):= TQ^\perp(\Z)=\{v\in\t(\Z)\,|\,\alpha(v)=0,\,\forall \alpha\in TQ\}\subset\t(\Z)\,.
\end{equation*}
The cosheaf $\Sed(Q)$ consists of the integral vectors orthogonal to $TQ$. So, whenever $Q'$ is a face of $Q$, $\Sed(Q)$ is a sub-module of $\Sed(Q')$. The extension morphisms are simply given by these inclusions. We denote all its subdivisions by the same symbol $\Sed$;
\begin{figure}[!h]
\centering
\begin{subfigure}[t]{0.45\textwidth}
\centering
\begin{tikzpicture}[scale=1.90]
\coordinate (u) at (-.5,0,-.5);

\begin{scope}[>=Triangle]
\draw[very thick,->] (u) -- ($(u)+(0,0,.5)$) node[above=5pt, left=.1pt]{$\df x$};
\draw[very thick,->] (u) -- ($(u)+(.5,0,0)$) node[above=7pt, left=4pt]{$\df y$};
\draw[very thick,->] (u) -- ($(u)+(0,.5,0)$) node[left=4pt]{$\df z$};
\end{scope}

\begin{scope}[canvas is xz plane at y=0]
\draw[step=.5cm] (-.9,-.9) grid (1.9,2.4);
\end{scope}

\draw[dashed,very thick](0,0,0) -- (0,-1,0);
\fill (0,-1,0) circle (.05);
\draw[dashed,very thick] (0,-1,0) -- (0,0,1);
\draw[dashed,very thick] (0,-1,0) -- (1,0,0);
\fill[white] (0,0,0) -- (0,0,1) -- (1,0,0) -- cycle;
\filldraw[pattern={Lines[angle=22.5, yshift=4pt, line width=.5pt]},very thick] (0,0,0) -- (0,0,1) -- (1,0,0) -- cycle;
\fill (0,0,0) circle (.05) -- (0,0,1) circle (.05) -- (1,0,0) circle (.05);
\draw[-,thick] (2,1.5,2) node[right]{$Q^2$}.. controls (0,0,-1) and (.5,.5,0).. (.25,0,.25);
\draw[-,thick] (-1,.5,.5) node[left]{$Q^1$}.. controls (-.5,.2,1) and (0,.175,.5).. (0,0,.5);
\end{tikzpicture}
\caption{The polytope in $\t^*(\R)$.}
\end{subfigure}
\begin{subfigure}[t]{0.45\textwidth}
\centering
\begin{tikzpicture}[scale=1.90]
\coordinate (v) at (4,0,0);

\begin{scope}[canvas is xy plane at z=0]
\draw[step=.5cm] ($(v)+(-.9,-.9)$) grid ($(v)+(.9,.9)$);
\foreach \p in {-.5,0,.5}{
\foreach \q in {-.5,0,.5}{
\fill ($(v)+(\p,\q,0)$) circle (.05);}}
\draw[dashed]($(v)+(-.03,-0.9)$) -- ($(v)+(-.03,0.9)$);
\draw[dashed]($(v)+(.03,-0.9)$) -- ($(v)+(.03,0.9)$);
\draw ($(v)+(0,.9)$) node[above]{$\Sed(Q^2)$};
\draw[dotted, thick] ($(v)+(-.7,-.9)$) node[left]{$\Sed(Q^1)$} arc[start angle=0, end angle=90, radius=9pt] -- ++(-90:5pt);
\end{scope}

\begin{scope}[canvas is xz plane at y=0]
\draw[step=.5cm] ($(v)+(-.9,-.9)$) grid ($(v)+(.9,.9)$);
\end{scope}

\coordinate (u') at (5,-0.5,0);

\begin{scope}[>=Triangle]
\draw[very thick,->] (u') -- ($(u')+(0,0,.5)$) node[below=1pt]{$\partial_x$};
\draw[very thick,->] (u') -- ($(u')+(.5,0,0)$) node[below=1pt]{$\partial_y$};
\draw[very thick,->] (u') -- ($(u')+(0,.5,0)$) node[right=1pt]{$\partial_z$};
\end{scope}
\end{tikzpicture}
\caption{The sedentarity groups in $\t(\R)$.}
\end{subfigure}
\caption{An edge $Q^1$ of a $2$-dimensional face $Q^2$ and their respective sedentarity.}
\end{figure}
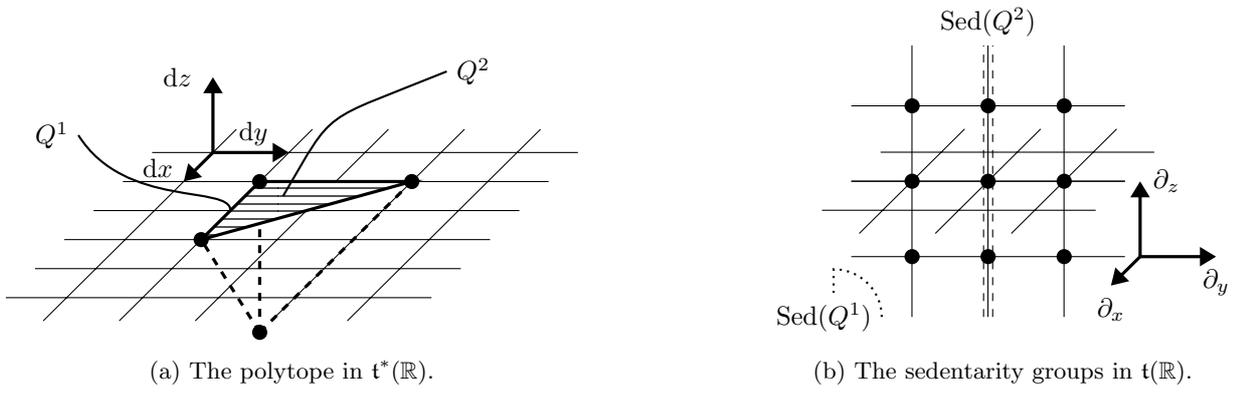

\item Let $p\geq 0$ be an integer. The $p^{\rm th}$ \emph{tropical cosheaf}\index{Tropical Cosheaves!FPP@$F_p^P$} $F_p^P$ associated to $P$ is defined by the following formula:
\begin{equation*}
F^{P}_p:= \bigwedge^p \bigslant{\t(\Z)}{\Sed}\;.
\end{equation*}
We denote all its subdivisions by the same symbol $F^{P}_p$;
\begin{figure}[h!]
\centering
\begin{tikzpicture}[scale=2]

\foreach \x in {0,1}{

\foreach \y in {0,1}{

\foreach \z in {0,1}{
\fill (\x,\y,\z) circle (.05);}}}

\foreach \p in {(0,0,1),(1,0,1),(1,0,0)}{\draw[very thick] \p -- ++(0,1,0);}
\foreach \p in {(0,1,1),(1,1,1),(1,0,1)}{\draw[very thick] \p -- ++(0,0,-1);}
\foreach \p in {(0,1,0),(0,1,1),(0,0,1)}{\draw[very thick] \p -- ++(1,0,0);}
\foreach \p in {(0,0,1),(0,1,0),(1,0,0)}{\draw[dashed,very thick] (0,0,0) -- ++\p;}
\fill[pattern={Lines[angle=22.5, yshift=4pt, line width=.5pt]}] (0,0,0) -- (0,0,1) -- (0,1,1) -- (0,1,0) -- cycle;
\draw[thick] (-1,.5,.5) node[left]{$\displaystyle \bigwedge^2 \bigslant{\Z^3}{\Z\partial_y}$}.. controls (-0.5,0.5,0) and (-0.5,0.5,2).. (0,.5,.5) ;
\draw[thick] (-1,0,1) node[left]{$\displaystyle \bigwedge^2 \bigslant{\Z^3}{\Z^3}=0$}.. controls (-0.05,0.5,3) and (-0.05,0.5,3).. (0,0,1) ;
\draw[thick] (2,0.75,0.5) node[right]{$\displaystyle \bigwedge^2 \Z^3$}.. controls (0.75,0.75,0.5) and (.75,0.5,0.5).. (0.5,0.5,0.5) ;
\draw[thick] (2,0.2,0.5) node[right]{$\displaystyle \bigwedge^2 \bigslant{\Z^3}{\Z\partial_y+\Z\partial_z}=0$}.. controls (1.5,-0.25,0.5) and (1.5,-0.5,0.5).. (1,0,0.5) ;
\end{tikzpicture}
\caption{A cube and the groups associated by $F_2^{P}$ to some of its faces.}
\end{figure}
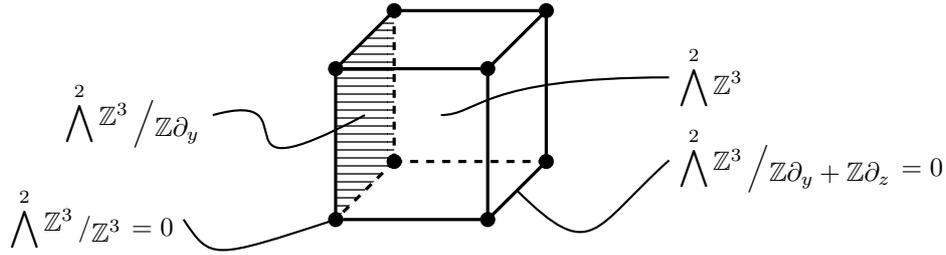
\item Let $p\geq 0$ be an integer. The $p^{\rm th}$ \emph{tropical cosheaf} associated to the dual hypersurface $X$ is denoted by $F_p^{X}$. It is a dihomologic cosheaf on $K$. For all pairs $e^q<e^{q'}$ of cells of $K$, the group $F_p^{X}(e^q;e^{q'})$ is given by the following formula:
\begin{equation}\label{eq:F1}
\tag{F}
F_p^{X}(e^q;e^{q'}):= \sum_{e^1\,\leq\,e^q} \bigwedge^p \bigslant{(Te^1)^\perp(\Z)}{\Sed(e^{q'})}\subset F_p^P(e^q;e^{q'})\;.
\end{equation}
As pointed out in Definition~\ref{definition:dih_cosheaves}, the extension morphisms need only be defined on elementary adjacencies. Let $e^{q_1} \leq e^{q_2} \leq e^{q_3} \leq e^{q_4}$ be four cells of $K$ (maybe with repetitions). The elementary extension morphisms are depicted in the following diagram:
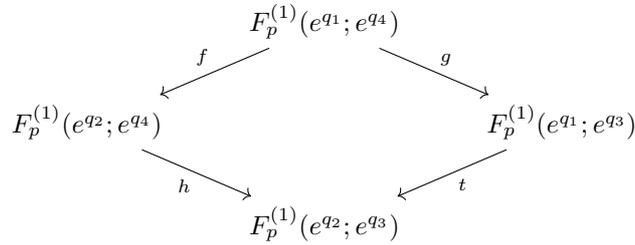
\begin{equation*}
\begin{tikzcd}
& F^{X}_p(e^{q_1};e^{q_4}) \ar[dr,"g" above right] \ar[dl,"f" above left] & \\%
F^{X}_p(e^{q_2};e^{q_4}) \ar[dr, "h" below left] & & F^{X}_p(e^{q_1};e^{q_3}) \ar[dl, "t" below right] \\%
& F^{X}_p(e^{q_2};e^{q_3}) &
\end{tikzcd}
\end{equation*}
where:
\begin{enumerate}
\item The morphisms $f$ and $t$ are the inclusions coming from the definition of the groups, c.f. \eqref{eq:F1};
\item The morphisms $g$ and $h$ are the reductions modulo $\Sed(e^{q_3})$. More precisely they correspond to the canonical projection:
\begin{equation*}
\bigwedge^p \bigslant{(Te^1)^\perp(\Z)}{\Sed(e^{q_4})} \rightarrow \bigwedge^p \bigslant{(Te^1)^\perp(\Z)}{\Sed(e^{q_3})},
\end{equation*}
\noindent on every summand of~\eqref{eq:F1}.
\end{enumerate}
By the nature of the morphisms involved, the diagram is commutative. It is a sub-cosheaf of $F_p^P$ and we denote the inclusion by $i_p:F_p^X\rightarrow F_p^P$. Even if it is defined on the dihomologic subdivision of $K$ it is supported on $X$. The Figure~\ref{fig:F1} illustrates the values taken by the cosheaf $F_1^{X}$ on a triangle with trivial subdivision $K$.
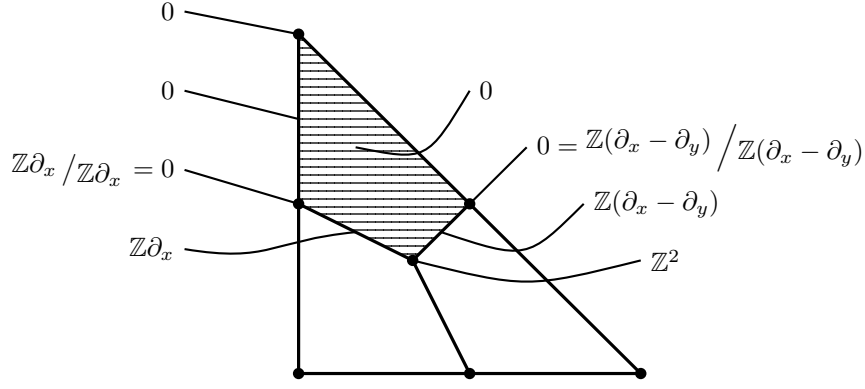
\begin{figure}[h]
\centering
\begin{tikzpicture}[scale=1.5,thick]
\foreach \p in {(0,0), (3,0), (0,3)}{
\fill \p circle (.05);}

\draw[very thick] (0,0) -- (3,0) -- (0,3) -- cycle;
\fill (1,1) circle (.05);
\fill ($3/2*(1,1)$) circle (.05);
\fill ($1/2*(3,0)$) circle (.05);
\fill ($1/2*(0,3)$) circle (.05);
\draw[very thick] ($3/2*(1,1)$) -- (1,1) -- ($1/2*(3,0)$);
\draw[very thick] ($1/2*(0,3)$) -- (1,1);
\fill[pattern={Lines[angle=22.5, yshift=4pt, line width=.5pt]}] (0,3) -- ($3/2*(0,1)$) -- (1,1) -- ($1/2*(3,3)$) -- cycle;
\draw (-1,3.2) node[left]{$0$} -- (0,3);
\draw (-1,2.5) node[left]{$0$} -- (0,2.25);
\draw (-1,1.8) node[left]{$\displaystyle \bigslant{\Z\partial_x}{\Z\partial_x}=0$} -- (0,1.5);
\draw (-1,1.1) node[left]{$\Z\partial_x$}.. controls ($1/2*(-1,1.1)+1/4*(1,1)+1/4*(0,1.5)+(0,-.2)$) and ($1/2*(-1,1.1)+1/4*(1,1)+1/4*(0,1.5)$).. ($1/2*(1,1)+1/2*(0,1.5)$) ;
\draw ($1/2*(3,0)+10/4*(0,1)$) node[right]{$0$}.. controls ($(0,3)!.75!(1.5,1.5)$) and ($(0,3)!.75!(1.5,1.5)$).. (0.5,2);
\draw ($1/2*(3,0)+10/4*(0,1)+(.5,-.5)$) node[right]{$\displaystyle 0=\bigslant{\Z(\partial_x-\partial_y)}{\Z(\partial_x-\partial_y)}$} -- ($3/2*(1,1)$);
\draw ($1/2*(3,0)+10/4*(0,1)+(1,-1)$) node[right]{$\Z(\partial_x-\partial_y)$}.. controls (2,1) and (2,1).. ($3/2*(1,1)+(-.25,-.25)$);
\draw ($1/2*(3,0)+10/4*(0,1)+(1.5,-1.5)$) node[right]{$\Z^2$}.. controls (2,.75) and (2,.75).. (1,1);
\end{tikzpicture}
\caption{A triangle and the groups associated by $F_1^{X}$ to some of its dihomologic cells.}\label{fig:F1}
\end{figure}
\end{enumerate}
\end{defi}

\begin{defi}[Contraction]\label{definition:contraction}
Let $M$ be a free module of finite rank over a ring $R$, and $k,l\in\N$ be integers. For all $\alpha\in \bigwedge^l \Hom_R(M;R)$ and $v\in\bigwedge^{l+k}M$ the contraction $\alpha\cdot v$ is the only element of $\bigwedge^{k} M$ satisfying:
\begin{equation*}
\beta(\alpha\cdot v)=(\beta\wedge\alpha)(v)\,,
\end{equation*}
for all $\beta\in \bigwedge^k\Hom_R(M;R)$. This construction is dual to the interior product.
\end{defi}

\begin{defi}\label{definition:saturated_co_sheaves}
Let $P$ be an integral polytope of $\t^*(\R)$, $K$ be an integral subdivision of $P$, and $p\geq 0$ be an integer. The \emph{$p^{\rm th}$ saturated tropical cosheaf} of the dual hypersurface $X$ is the dihomologic cosheaf denoted by $\Fbar^X_p$ that is defined for all pairs $e^q \leq e^{r}$ by the following formula:
\begin{equation*}
\Fbar^X_p\left(e^q ;e^{r}\right):= \left\{ v\in F^P_p\left(e^q ;e^{r}\right)\;\Big|\; \omega\cdot v =0,\,\forall \omega \in \bigwedge^qTe^q(\Z)\right\}\,.
\end{equation*}
The contraction $\omega\cdot v$ is well defined since $\Sed\left(e^{r}\right)$ is included in the orthogonal of $Te^q(\Z)$. Let $e^{q}\leq e^{q+1}\leq e^{r}$ be a triple of cells of $K$. There is a form $\alpha\in\t^*(\Z)$ for which $[\omega\mapsto \alpha\wedge\omega]$ is an isomorphism from $\bigwedge^{q}Te^{q}(\Z)$ to $\bigwedge^{q+1}Te^{q+1}(\Z)$. Hence $\Fbar^X_p(e^{q+1};e^{r})$ is included in $\Fbar^X_p(e^{q};e^{r})$ and $\Fbar^X_p$ is a sub-cosheaf of the dihomologic subdivision of $F^P_p$. We denote by $\ibar_p$ the inclusion of $\Fbar^X_p$ in $F^P_p$. A direct computation shows that $F^X_p$ is included in $\Fbar^X_p$.
\end{defi}

\begin{lemm}\label{lemm:finite_index}
Let $M$ be a free Abelian group of finite rank and some linear forms $\alpha_1,\dots,\alpha_k\in\Hom(M;\Z)\setminus\{0\}$. If $\omega$ is a generator of the last exterior power of the group spanned by the $\alpha_i$'s then for all $p\in \N$, the quotient:
\begin{equation*}
G:= \bigslant{\displaystyle \left\{v\in\bigwedge^pM\;\Big|\;\omega\cdot v=0\right\}}{\displaystyle \sum_{i=1}^k\bigwedge^p\ker(\alpha_i)}\,,
\end{equation*}
\noindent is a finite group.
\end{lemm}

\begin{proof}
We notice that for every $1\leq i\leq k$, the form $\alpha_i$ divides a non-zero multiple of $\omega$ in the exterior algebra of $\Hom(M;\Z)$. It follows that all $p$-elements of $\ker(\alpha_i)$ contract to $0$ against $\omega$. We assume now that $r\geq 1$ is the rank of the sub-group spanned by the $\alpha_i$'s and that $\alpha_1\wedge\dots\wedge\alpha_r=m\omega$ with $m\neq 0$. Then we have:
\begin{equation*}
\left\{v\in\bigwedge^pM\;\Big|\;\omega\cdot v=0\right\}=\left\{v\in\bigwedge^pM\;\Big|\;(\alpha_1\wedge\dots\wedge\alpha_r)\cdot v=0\right\}\,.
\end{equation*}
We can complete the set $\{\alpha_1,\dots,\alpha_r\}$ with a set $\{\beta_1,\dots,\beta_s\}\subset\Hom(M;\Z)$ to form a basis of the rational vector space $\Hom(M;\Q)$. For all $p$-elements $v\in\bigwedge^p M$, seen as $p$-vectors in $\bigwedge^p_\Q(M\otimes\Q)$, we have:
\begin{equation*}
(\alpha_1\wedge\dots\wedge\alpha_r)\cdot v = \sum_{|J|=p-r} (\alpha_1\wedge\dots\wedge\alpha_r\wedge\beta_J)(v)f_J\,,
\end{equation*}
where $\{e_1,\dots,e_r,f_1,\dots,f_s\}$ denotes the dual basis of $\{\alpha_1,\dots,\alpha_r,\beta_1,\dots,\beta_s\}$. As a consequence, we have:
\begin{equation*}
\left\{v\in\bigwedge^pM\;\Big|\;\omega\cdot v=0\right\}\otimes\Q=\Big\langle e_I\wedge f_J\colon |I|<r \text{ and }|I|+|J|=p\Big\rangle_\Q\,.
\end{equation*}
It implies the equality of $\ker(\alpha_i)\otimes \Q$, $\ker(\alpha_i\otimes 1)$, and $\langle e_j,f_k\colon j\neq i \text{ and } 1\leq k\leq s\rangle_\Q$. Hence:
\begin{equation*}
\left\{v\in\bigwedge^pM\;\Big|\;\omega\cdot v=0\right\}\otimes\Q=\left(\sum_{i=1}^k\bigwedge^p\ker(\alpha_i)\right)\otimes\Q\,,
\end{equation*}
and $G\otimes \Q$ vanishes. Since $G$ is finitely generated, it is finite.
\end{proof}

\begin{prop}\label{prop:saturation}
Let $P$ be an integral polytope of $\t^*(\R)$, $K$ be an integral subdivision of $P$, and $p\geq 0$ be an integer. The cosheaf $\Fbar_p^X$ is the saturation of $F^X_p$ in $F^P_p$, i.e. for all pairs $e^q\leq e^r$, we have:
\begin{equation*}
\Fbar^X_p(e^q;e^r)=\left\{ v\in F^P_p(e^q;e^r)\;\big|\; \exists k\in\Z,\, kv\in F^X_p(e^q;e^r)\right\}\,.
\end{equation*}
\end{prop}

\begin{proof}
Let $e^q\leq e^r$ be two cells of $K$. Lemma~\ref{lemm:finite_index} implies that $F_p^X(e^q;e^r)$ is of finite index in $\Fbar_p^X(e^q;e^r)$. Hence, these two groups have same rank and same saturation in $F^P_p(e^q;e^r)$. The module $\Fbar_p^X(e^q;e^r)$ is the kernel of a morphism between two free Abelian groups, thus it is saturated.
\end{proof}

\begin{defi}\label{definition:theta}
Let $P$ be an integral polytope of $\t^*(\R)$ and $K$ be an integral polyhedral subdivision of $P$. We denote by $\theta(K)$ the least common multiple of the exponents\footnote{We recall that the exponent of an Abelian group $M$ is the smallest, if any, $e\in\N\setminus\{0\}$ for which $em$ vanishes for all $m\in M$.} of the quotients of $\Fbar^X_p(e^q;e^r)$ by $F^X_p(e^q;e^r)$, for all integers $p\geq 0$ and all pairs of cells $e^q\leq e^r$ of $K$. If $P$ is endowed with its trivial subdivision $\Pi$ we denote $\theta(\Pi)$ by $\theta(P)$.
\end{defi}

\begin{prop}
Let $P$ be an integral polytope endowed with an integral polyhedral subdivision $K$. Let $R$ be a ring in which $\theta(K)$ is invertible. Then $F_p^X\otimes R$ equals $\Fbar_p^X\otimes R$, for all integers $p\geq 0$.
\end{prop}

\begin{proof}
We have the following exact sequence of cosheaves of $R$-modules:
\begin{equation*}
0 \rightarrow \Tor\left(\bigslant{\Fbar^X_p}{F^X_p}; R\right) \rightarrow F_p^X\otimes R \rightarrow \Fbar^X_p\otimes R \rightarrow \bigslant{\Fbar^X_p}{F^X_p}\otimes R \rightarrow 0\,.
\end{equation*}
Since $\theta(K)$ is invertible in $R$, both $\Tor(\smallslant{\Fbar^X_p}{F^X_p}; R)$ and $\smallslant{\Fbar^X_p}{F^X_p}\otimes R$ vanish.
\end{proof}

\begin{defi}[Tropical Homology Groups]\label{definition:trop_homology}
Let $R$ be a commutative ring, $P$ be an integral polytope, $K$ be an integral polyhedral subdivision of $P$, and $X$ be the dual hypersurface of $K$. The \emph{tropical homology groups} of $P$ are defined, for all $p,q\in\N$, by:
\begin{equation*}
H_{p,q}(P;R):=H_q\left(K;F_p^P\otimes R\right).
\end{equation*}
Likewise, the \emph{tropical homology groups} of $X$ are given for all $p,q\in\N$, by:
\begin{equation*}
H_{p,q}(X;R):=H_q\left(K;F_p^X\otimes R\right).
\end{equation*}
The \emph{saturated tropical homology groups} of $X$ are given for all $p,q\in\N$, by:
\begin{equation*}
\Hbar_{p,q}(X;R):=H_q\left(K;\Fbar_p^X\otimes R\right).
\end{equation*}
Moreover, the inclusions $F_p^X\subset \Fbar_p^X \subset F_p^P$ for all $p\in\N$, induce morphisms in homology:
\begin{equation*}
\begin{tikzcd}
H_{p,q}(X;R) \ar[rd, "i_{p,q}" above] \ar[d]& \\
\Hbar_{p,q}(X;R) \ar[r, "\ibar_{p,q}" below]& H_{p,q}(P;R)
\end{tikzcd}
\end{equation*}
\end{defi}

\subsection{A Lefschetz Hyperplane Section Theorem in Simple Polytopes}

\begin{defi}
A polytope $P$ is \emph{simple} if every face of codimension $q$ of $P$ is the intersection of exactly $q$ faces of codimension~$1$.
\end{defi}

\begin{defi}
Let $P$ be a simple integral polytope, we denote by $\Sed_{(1)}$ the cosheaf:
\begin{equation*}
\Sed_{(1)}:= \bigoplus_{\substack{Q\,<\,P \\ \codim Q =1}} \Big[\, Q\, ;\, \Sed(Q) \, \Big].
\end{equation*}
If $Q$ is a codimension~$1$ face of $P$ there is a natural injective cosheaf morphism $[\, Q\, ;\, \Sed(Q) \, ]\rightarrow \Sed$. If $Q'$ is a face of $P$ the group $[\, Q\, ;\, \Sed(Q) \,](Q')$ is either $\Sed(Q)$ or $0$. The former only happens when $Q'$ is a face of $Q$. In this case, the morphism is given by the inclusion. Summing all these morphisms yields a morphism:
\begin{equation*}
\Sed_{(1)}\rightarrow \Sed.
\end{equation*}
It is injective because the polytope $P$ is simple: if $Q'<P$ has codimension $q$ it is the intersection of exactly $q$ faces of codimension~$1$ and:
\begin{equation*}
\Big(TQ'\Big)^\perp=\bigoplus_{\substack{Q\,<\,P \\ \codim Q =1}} TQ^\perp,
\end{equation*}
so the $\Sed(Q)$'s are in direct sum inside $\Sed(Q')$. For the same reason the quotient $\Delta$ of $\Sed$ by $\Sed_{(1)}$ is a cosheaf of finite groups. We denote by $\delta(P)$ the least common multiple of the exponents of the groups $\Delta(Q)$ for all $Q\leq P$.
\end{defi}

\begin{remas}
Let $P$ be a simple integral polytope:
\begin{enumerate}
\item The cosheaf of finite groups $\Delta$ encodes the singularities of the toric variety associated with $P$. The set of complex points of the affine open set associated with the face $Q$ of $P$ is the quotient of $(\C^\times)^k\times\C^{n-k}$ by an algebraic action of the group $\Delta(Q)$ (c.f. \cite[Section~2.2 p.~34]{Ful_tor_var}).
\item Let $Q_1<Q_2$ be two faces of $P$, the Snake Lemma implies that the following commutative diagram of exact sequences:
\begin{equation*}
\begin{tikzcd}
0 \ar[r] & \Sed_{(1)}(Q_2) \ar[d,"f"] \ar[r] & \Sed(Q_2) \ar[d,"g"] \ar[r] & \Delta(Q_2) \ar[d,"h"] \ar[r] & 0 \\%
0 \ar[r] & \Sed_{(1)}(Q_1) \ar[r] & \Sed(Q_1) \ar[r] & \Delta(Q_1) \ar[r] & 0
\end{tikzcd}
\end{equation*}
gives rise to the following exact sequence:
\begin{align*}
0 \rightarrow \underset{\ker(f)}{\underbrace{0}} \rightarrow \underset{\ker(g)}{\underbrace{0}} \rightarrow \ker(h) \rightarrow \underset{\coker(f)}{\underbrace{\bigoplus_{\substack{Q_1\,<\,Q\,<\,P \\ Q \,\ngtr\,Q_2 \\ \codim Q=1}}\Sed(Q)}} &\rightarrow \underset{\coker(g)}{\underbrace{\Sed(Q_1)/\Sed(Q_2)}}\\
&\rightarrow \coker(h) \rightarrow 0\;.
\end{align*}
However, $\ker(h)$ is a finite group and $\coker(f)$ is free, so $\ker(h)=0$ and $h$ is injective. As a consequence, $\delta(P)$ is the least common multiple of the exponents of the groups $\Delta(V)$ for all vertices $V$ of $P$.
\item Let $R$ be a commutative ring, the classification of finite Abelian groups implies that the vanishing of both $\Delta(V)\otimes R$ and $\Tor(\Delta(V);R)$ is equivalent to the invertibility of the exponent of the group $\Delta(V)$ in $R$. When $R$ is a field this is equivalent to the coprimality of this exponent with the characteristic of $R$.
\end{enumerate}
\end{remas}

\begin{prop}\label{prop:preparation_F0_1st_position}
Let $P$ be a simple integral polytope endowed with an integral polyhedral subdivision $K$. For all rings $R$ in which $\delta(P)$ is invertible, the dihomologic subdivisions of the cosheaves $(F^{P}_p\otimes R)_{p\,\geq\,0}$ satisfy the hypotheses of Proposition~\ref{prp:degenerate_E2_lef} i.e. the sheaves $H_k((F^{P}_p\otimes R)_*)$ vanish for all integers $p$ and all integers $k$ different from $\dim P$.
\end{prop}

\begin{proof}
Let $F$ denote one of the cosheaves $(F^P_p)_{0\,\leq \,p\,\leq\,n}$ on $P$. The polyhedral subdivision $K$ of $P$ first and then its dihomologic subdivision give rise to two consecutive subdivisions of cosheaves $F\mapsto F'$ and $F'\mapsto F''$. Let $e$ be a cell of~$K$, the localisation process for dihomologic cosheaves consisting in ``fixing the first coordinate'' $G\mapsto G_e$ applied to a subdivided cosheaf is equivalent to the localisation of the initial cosheaf, c.f. diagram~\eqref{diag:comm_rel_loc}. Consequently, the cosheaf $F''_e$ is the same as~$F'_e$. In the light of Lemma~\ref{lemm:double_loc} on double localisation, showing that $F'_e$ can only have non-trivial homology in degree $n=\dim P$ for all cells $e\in K$ is equivalent to showing that $F_Q$ can only have non-trivial homology in degree $n$ for all faces $Q$ of~$P$. Since $P$ is simple, we have the following exact sequence of cosheaves of abelian groups:
\begin{equation*}
0\rightarrow \Sed_{(1)} \rightarrow \Sed \rightarrow \Delta \rightarrow 0,
\end{equation*}
which is tensorised to the following exact sequence of cosheaves of $R$-modules:
\begin{equation*}
0 \rightarrow \Tor(\Delta;R) \rightarrow \Sed_{(1)}\otimes R \rightarrow \Sed \otimes R \rightarrow \Delta \otimes R\rightarrow 0.
\end{equation*}
By hypothesis on $\delta(P)$, $\Sed\otimes R$ is isomorphic to $\Sed_{(1)}\otimes R$ and $F^{P}_p\otimes R$ is isomorphic to:
\begin{equation*}
\bigwedge^p \bigslant{\t(R)}{(\Sed_{(1)}\otimes R)}.
\end{equation*}
The cosheaf $\Sed_{(1)}$ can be described as follows: we choose, for every codimension\,$1$ face ${Q_{(1)}}$ of $P$, a generator $g_{Q_{(1)}}$ of $\Sed({Q_{(1)}})\cong \Z$, and associate to all faces $Q$ of~$P$ the set $G(Q)$ of the $g_{Q_{(1)}}$'s for which $Q<{Q_{(1)}}$. The association $G:Q\mapsto G(Q)$ is a cosheaf of sets with inclusions as extension maps. In these notations, $\Sed_{(1)}$ is the cosheaf that associates to $Q$ the sub-module of $\t(\Z)$ spanned by $G(Q)$ and with extension morphisms given by inclusions. We can consider the cosheaf of chain complexes $Q\mapsto C(\t(\Z);G(Q);p)$ of Definition~\ref{definition:resolution_ext_alg}. It has an augmentation morphism:
\begin{equation*}
0\leftarrow \bigwedge^p \bigslant{\t(\Z)}{\Sed_{(1)}} \leftarrow C(\t(\Z);G;p)\;,
\end{equation*}
and becomes a resolution once we tensorise every group by $R$. Indeed, for all $Q$ the set $G(Q)\otimes R$ is linearly independent in $\t(R)$ and $\Sed_{(1)}(Q)\otimes R=\Sed(Q)\otimes R$ is a free summand of $\t(R)$ since $\t(R)/(\Sed(Q)\otimes R)$ is isomorphic to $\Hom(TQ_\Z\,;R)$, thus free. Hence Lemma~\ref{lemm:resolution_ext_alg} applies. For this reason, we set $\Res_p$ to be $C(\t(\Z);G;p)\otimes R$ and we have a resolution of cosheaves:
\begin{equation}\label{eq:resolution_F_xxx}
\tag{R}
0\leftarrow F^P_p\otimes R \leftarrow \Res_p\;.
\end{equation}
Localising at a face $Q$ amounts to tensorisation by a cosheaf of free modules, therefore it is an exact endofunctor of the category of cellular cosheaves. To avoid confusion we will denote, from now on, the CW-complex defined by the polytope $P$ by $\Pi$ and every face $Q\leq P$ is meant to be open. In particular, $\Pi(Q)$ is the smallest sub-complex of $\Pi$ containing the open face $Q$, that is to say the collection of its faces. Let $Q$ be a face of $P$, we have the following local resolution:
\begin{equation*}
0\leftarrow \big(F^{P}_p\otimes R\big)_Q \leftarrow \big(\Res_p\big)_Q\;,
\end{equation*}
and for all integers $q\geq 0$:
\begin{multline*}
\big(\Res_{p,q}\big)_Q\\
\begin{aligned}
&= \bigoplus_{\substack{Q_{(q)}\,<\,P \\ \codim Q_{(q)} =q}} \left[\, \Pi(Q_{(q)})\,;\, \left(\bigwedge^{p-q}\bigslant{\t(\Z)}{\Sed_{(1)}(Q_{(q)})}\right)\otimes R \right]\otimes_R \Big[\Pi \,;\,\Pi-Q\,;R\Big]\\%
&= \bigoplus_{\substack{Q\,<\,Q_{(q)}\,<\,P \\ \codim Q_{(q)} =q}} \left[\, \Pi(Q_{(q)})\,;\,\Pi(Q_{(q)})-Q\,;\, \left(\bigwedge^{p-q}\bigslant{\t(\Z)}{\Sed_{(1)}(Q_{(q)})}\right)\otimes R \right]\;.
\end{aligned}
\end{multline*}
If $Q$ is a proper face of ${Q_{(q)}}$ then $|\Pi(Q_{(q)})|$, the closure of ${Q_{(q)}}$, retracts on $|\Pi(Q_{(q)})-Q|$ and the cosheaf:
\begin{equation*}
\left[\, \Pi(Q_{(q)})\,;\,\Pi(Q_{(q)})-Q\,;\, \left(\bigwedge^{p-q}\bigslant{\t(\Z)}{\Sed_{(1)}(Q_{(q)})}\right)\otimes R \right]\;,
\end{equation*}
has trivial homology. When $Q_{(q)}$ equals $Q$, we are computing the homology of a closed $(n-q)$-ball relatively to its boundary. The homology is concentrated in top dimension $n-q$. In application of Lemma~\ref{lemm:hom_shi}, the homology of $\big(F^{P}_p\otimes R\big)_Q$ is the shift by $q$ of $(\Res_{p,q})_Q$ and therefore has homology concentrated in dimension $(n-q)+q=n$.
\end{proof}

\begin{rema}\label{rmk:resolution_F0}
In the previous proof we have defined a resolution~\eqref{eq:resolution_F_xxx} of the cosheaf $F^P_p\otimes R$, denoted by:
\begin{equation*}
0\leftarrow F^P_p\otimes R \leftarrow \Res_p\;.
\end{equation*}
\end{rema}

\begin{prop}\label{prop:preparation_F1_1st_position}
Let $P$ be a simple integral polytope endowed with an integral polyhedral subdivision $K$, and $X$ be the dual hypersurface of $K$. For all rings $R$ in which $\delta(P)$ is invertible, for all integers $p$, and all integers $k$ different from $\dim P$, the sheaves of local homology $H_k((\Fbar_p^X\otimes R)_*)$ vanish.
\end{prop}

\begin{proof}
Let $p$ be a non-negative integer. Proposition~\ref{prop:preparation_F1_1st_position} follows from Proposition~\ref{prop:preparation_F0_1st_position} and the fact that $\Fbar_p^{X}\otimes R$ is locally a direct summand of $F_p^{P}\otimes R$. More precisely, we will show that for all cells $e^q$ of $K$, the cosheaf $(\Fbar_p^{X}\otimes R)_{e^q}$ is a direct summand of the cosheaf $(F_p^{P}\otimes R)_{e^q}$. This statement implies that the homology of the former is a direct summand of the latter. By Proposition~\ref{prop:preparation_F0_1st_position}, the latter can only have non-vanishing homology in dimension $n=\dim P$ and therefore so can the former.

Let $e^q$ be a cell of $K$ and $A$ denote $(Te^q)^\perp(\Z)$. Now let us choose $B$ a supplementary sub-module\footnote{It can be done for $(Te^q)^\perp$ is a rational sub-space of $\t(\R)$.} of $A$ in $\t(\Z)$. We consider the two cosheaves $\bigslant{A}{\Sed}$ and $\bigslant{B}{\Sed}$ that are the respective images of the constant cosheaves $[K;A]$ and $[K;B]$ in the quotient $\smallslant{\t(\Z)}{\Sed}$. We note that if $e^r$ is a cell containing $e^q$ then $\Sed(e^r)\subset A$ and the projection~$B\rightarrow \smallslant{B}{\Sed(e^r)}$ is an isomorphism. These projections induce an isomorphism of cosheaves:
\begin{equation*}
\big[K;K-e^q;B\big]\overset{\cong}{\longrightarrow} \left(\bigslant{B}{\Sed}\right)_{e^q}\;.
\end{equation*}
Therefore, the cosheaf $\bigslant{\t(\Z)}{\Sed}$ splits around $e^q$ into the following direct sum:
\begin{equation*}
\left(\bigslant{\t(\Z)}{\Sed}\right)_{e^q}=\Big(\bigslant{A}{\Sed}\Big)_{e^q}\oplus \Big(\bigslant{B}{\Sed}\Big)_{e^q}\;.
\end{equation*}
For all $p\geq0$, we have the decomposition:
\begin{equation*}
\left(\bigwedge^p\bigslant{\t(\Z)}{\Sed}\right)_{e^q}=\bigoplus_{p_A+p_B=p}\left(\bigwedge^{p_A}\bigslant{A}{\Sed}\right)_{e^q}\otimes \left(\bigwedge^{p_B}\bigslant{B}{\Sed}\right)_{e^q}\;.
\end{equation*}
All the cosheaves involved are made of free groups. Thus the decomposition remains valid after tensorisation by $R$:
\begin{equation*}
\left(\bigwedge^p\bigslant{\t(\Z)}{\Sed}\right)_{e^q}\otimes R=\bigoplus_{p_A+p_B=p}\left(\bigwedge^{p_A}\bigslant{A}{\Sed}\right)_{e^q}\otimes \left(\bigwedge^{p_B}\bigslant{B}{\Sed}\right)_{e^q}\otimes R\;.
\end{equation*}
Let $e^r\geq e^q$ be a pair of cells of $K$. Definition~\ref{definition:saturated_co_sheaves} implies that if $\omega$ is a generator of $\bigwedge^qTe^q(\Z)$, the group $\Fbar^{X}_p(e^q;e^r)$, consists of the $p$-elements of $\bigwedge^p\smallslant{\t(\Z)}{\Sed(e^r)}$ whose contraction against $\omega$ vanishes. Thus the localisation $(\Fbar^{X}_p)_{e^q}\otimes R$ can be expressed as:
\begin{equation*}
\left(\Fbar^{X}_p\right)_{e^q}\otimes R=\bigoplus_{\substack{p_A+p_B=p \\ p_B\,<\,q}}\left(\bigwedge^{p_A}\bigslant{A}{\Sed}\right)_{e^q}\otimes \left(\bigwedge^{p_B}\bigslant{B}{\Sed}\right)_{e^q}\otimes R\;,
\end{equation*}
and is a direct summand of $(F^{P}_p)_{e^q}\otimes R$.
\end{proof}

\begin{prop}
Let $P$ be a simple integral polytope and $R$ be a principal ideal domain in which $\delta(P)$ is invertible. For all $p\in\N$, the only non-trivial homology group of the cosheaf $F^{P}_p\otimes R$ is in dimension $p$. Moreover, this $R$-module is free and its rank is given by the following formula:
\begin{equation*}
\rk_R\, H_{p,p}(P;R) = \sum_{k=0}^p(-1)^{p-k} \binom{n-k}{p-k}f_{n-k}(P)\,,
\end{equation*}
where $f_k(P)$ denotes the number of $k$-faces of $P$.
\end{prop}

\begin{proof}
Since $F^{P}_p\otimes R$ is originally defined on the CW-structure induced by the faces of $P$, we can compute its homology on $\Pi$ the CW-complex induced on $P$ by its faces. We note, on the one hand, that from the Definition~\ref{definition:trop_cosheaves} the groups $F^{P}_p(Q)$ vanish for all faces $Q$ of dimension $q<p$. Therefore, $H_q(\Pi;F^{P}_p\otimes R)$ vanishes for all $q<p$. On the other hand, we have the resolution of Remark~\ref{rmk:resolution_F0}:
\begin{equation*}
0\leftarrow F^{P}_p\otimes R \leftarrow \Res_p\;.
\end{equation*}
This is an acyclic\footnote{A cosheaf is acyclic if it has trivial homology in dimension at least 1.} resolution\footnote{This is even a projective resolution, c.f. \cite{She_cel_des} in the dual setting of cellular sheaves.}. Indeed, the cosheaves $(\Res_{p,q})$ are sums of elementary cosheaves $[\Pi(Q);M]$ where $Q$ is a face of $P$ and $M$ a free $R$-module. All the sub-complexes $\Pi(Q)$ are contractible. Hence, these cosheaves only have non-trivial homology in dimension~$0$. The cosheaf resolution becomes a resolution of chain complexes:
\begin{align*}
0\leftarrow \left(C_k(\Pi;F^{P}_p\otimes R);\partial\right)_{k\,\geq\,0} &\leftarrow \big(C_k(\Pi;\Res_{p,0});\partial\big)_{k\,\geq\,0}\\
&\leftarrow \big(C_k(\Pi;\Res_{p,1});\partial\big)_{k\,\geq\,0}\\
&\leftarrow \cdots \;.
\end{align*}
Since the complexes $(C_k(\Pi;\Res_{p,q});\partial)_{k\,\geq\,0}$ are acyclic, the homology of the complex ${(C_k(\Pi;F^{P}_p\otimes R);\partial)_{k\,\geq\,0}}$ is isomorphic to the homology of the complex:
\begin{equation*}
0 \leftarrow H_0 \left(\Pi;\Res_{p,0}\right)\leftarrow H_0\left(\Pi;\Res_{p,1}\right)\leftarrow \cdots \;.
\end{equation*}
The cosheaves $\Res_{p,q}$ vanish for all $q>p$. Indeed, we have:
\begin{equation*}
\Res_{p,q}=\bigoplus_{\substack{Q_{(q)}\,<\,P \\ 
\codim Q_{(q)} =\,q}} \Bigg[\, \Pi(Q_{(q)})\,;\, \underset{=\,0 \text{ if }q\,>\,p}{\underbrace{\left(\bigwedge^{p-q}\bigslant{\t(\Z)}{\Sed_{(1)}(Q_{(q)})}\right)}}\otimes R \Bigg]\,.
\end{equation*}
Therefore, the only possibly non-trivial homology group of $F^{P}_p\otimes R$ is in dimension~$p$. Using again the resolution to compute this group, we see that it coincides with a group of cycles:
\begin{equation*}
H_{p}\left(\Pi;F^{P}_p\otimes R\right)\cong \ker\big(\partial\colon H_0\left(\Pi;\Res_{p,p}\right)\rightarrow H_0\left(\Pi;\Res_{p,p-1}\right) \big)\,.
\end{equation*}
Hence $H_{p}(\Pi;F^{P}_p\otimes R)$ is free and its rank equals the dimension of $H_{p}(\Pi;F^{P}_p\otimes R_0)$ where $R_0$ is the fraction field of $R$. Because $F^{P}_p\otimes R_0$ can only have non-trivial homology in dimension $p$, we have:
\begin{equation*}
\begin{split}
\dim_{R_0} H_p\left(\Pi;F^{P}_p\otimes R_0\right) &=(-1)^p\sum_{k=0}^n(-1)^k\dim_{R_0} H_0\left(\Pi;\Res_{p,k}\otimes_R R_0\right) \\
& = \sum_{k=0}^p(-1)^{p-k} \binom{n-k}{p-k}f_{n-k}(P)\,.\qedhere
\end{split}
\end{equation*}
\end{proof}

\begin{defi}\label{definition:Global_orientation}
Let $e^n$ be a cell of $K$. Since the closure of $e^n$ is a full dimensional polytope of $\t^*(\R)$ there is a canonical isomorphism $f$ between $\bigwedge^n\t(\Z)$ and $\Z(e^n)$. Let $\ori$ be a generator of $\bigwedge^n\t(\Z)$ and $\sigma^n$ be a barycentric simplex $[v_0;\dots;v_n]$ of the closure of $e^n$. Since $e^n$ is a polytope, its barycentric simplices define singular simplices of its support by convex interpolation\footnote{the vertices of a barycentric simplex are canonically ordered.}. We denote by $\ori(\sigma^n)$ the sign of the number $(\bigwedge_{i=1}^n(v_i-v_0))\ori\in\R$. The morphism $f$ sends the generator $\ori$ to the orientation of $e^n$ defined by the class of the singular chain $\sum_{\sigma^n\,\in\,\Sd(K(e^n))}\ori(\sigma^n)\sigma^n$. We denote by $\ori_{e^n}$ the image $f\ori$.
\end{defi}

\begin{lemm}\label{lemm:comparison_Hn(FX)_Hn(FP)}
Let $P$ be an $n$-dimensional integral polytope endowed with an integral polyhedral subdivision $K$, $X$ be the dual hypersurface of $K$, $0\leq p\leq n$ be an integer, and $R$ be a ring in which $\delta(P)$ is invertible.
\begin{enumerate}
\item The morphism of cellular sheaves ${j_p:H_n((\Fbar_p^X\otimes R)_*)\rightarrow H_n((F_p^P\otimes R)_*)}$ induced from ${\ibar_p:\Fbar_p^X\rightarrow F_p^P}$ is injective;
\item For all cells $e^q$ of $K$ there is an isomorphism:
\begin{equation*}
\psi_p^R(e^q):H_n\big(K;\left(F_p^P\otimes R\right)_{e^q}\big) \overset{\cong}{\longrightarrow} \bigwedge^n\t(\Z)\otimes\left(\bigwedge^s\Sed(e^q)\wedge\bigwedge^{p-s}\t(\Z)\right)\otimes R\,,
\end{equation*}
where $s$ denotes the rank of $\Sed(e^q)$;
\item The image of the composition $\psi_p^R(e^q)\circ j_p(e^q)$ is the tensor product of $\bigwedge^n\t(\Z)$ with the kernel of the contraction against $\omega$, a generator of $\bigwedge^qTe^q(R)$, restricted to $(\bigwedge^s\Sed(e^q)\wedge\bigwedge^{p-s}\t(\Z))\otimes R$. In particular, $j_p(e^q)$ is an isomorphism whenever $q$ is greater that $p$.
\end{enumerate}
\end{lemm}

\begin{proof}
\begin{enumerate}
\item The morphism $j_p(e^q)$ is injective because $K$ has dimension $n$. Since it is induced in homology by the injective morphism $i_p$, the associated long exact sequence in homology ensures that $j_p$ is injective.
\item Let $e^q$ be a cell of $K$. Since $K$ has dimension $n$, ${H_n(K;(F_p^P\otimes R)_{e^q})}$ is the same as $Z_n(K;(F_p^P\otimes R)_{e^q})$. Let $c\in C_n(K;(F_p^P\otimes R)_{e^q})$ and $[e^{n-1}]$ be an oriented cell adjacent to $e^q$. If $e^{n-1}$ does not belong to $\partial K$, there are two $n$-cells $e^n_\pm$ adjacent to $e^{n-1}$. In this case, we have:
\begin{equation*}
\left\langle \partial c;[e^{n-1}]\right\rangle = \left\langle c;[e^{n-1}][e^{n-1};e^n_+]\right\rangle + \left\langle c;[e^{n-1}][e^{n-1};e^n_-]\right\rangle\,,
\end{equation*}
\noindent since the extension morphism ${(F_p^P\otimes R)_{e^q}(e^n_\pm)\rightarrow (F_p^P\otimes R)_{e^q}(e^{n-1})}$ is the identity. Any two $n$-cells adjacent to $e^q$ can be joined by a sequence $e^n_1,\dots,e^n_k$, with $k\geq 1$, of $n$-cells such that $e^n_i$ and $e^n_{i+1}$ are adjacent to the same cell $e^{n-1}_i\geq e^q$, for all $1\leq i\leq k-1$. Therefore, if $\partial c$ vanishes we can find a $p$-element $v\in\bigwedge^p\t(R)$ such that $c=\sum_{e^n\,\geq\,e^q} v\otimes \ori_{e^n}$ where $\ori$ is a choice of generator of $\bigwedge^n\t(\Z)$. If now $e^{n-1}$ belongs to $\partial K$ we have:
\begin{equation*}
\left\langle \partial c;[e^{n-1}]\right\rangle = \pm v \;\left(\mod \left(\Sed(e^{n-1})\wedge\bigwedge^{p-1}\t(\Z)\right)\otimes R\right)\,.
\end{equation*}
It follows that $\partial c$ vanishes if and only if $v$ is divisible in the exterior algebra by the generators of ${\Sed(e^{n-1})\otimes R}$, for all $e^q\leq e^{n-1}\in\partial K$. Since $\delta(P)$ is invertible in $R$, $\Sed(e^q)\otimes R$ is equal to $\Sed_{(1)}(e^q)\otimes R$. Thus the latter condition is equivalent to $v$ being divisible by the generators of $\bigwedge^s\Sed(e^q)\otimes R$, where $s$ denotes the dimension of $\Sed(e^q)$. The inverse of the isomorphism $\psi^R_p(e^q)$ is given by the following formula:
\begin{equation}\label{eq:formula_psi-1}
\tag{$\Psi$}
\ori\otimes v\in \bigwedge^n\t(\Z)\otimes\left(\left(\bigwedge^s\Sed(e^q)\wedge\bigwedge^{p-s}\t(\Z)\right)\otimes R\right) \longmapsto \sum_{e^n\,\geq\,e^q} v\otimes\ori_{e^n}\,.
\end{equation}
\item The groups $\Fbar^X_p(e^q;e^n)$ are all equal. They are the kernel of the contraction $\omega\cdot-:\bigwedge^p\t(R) \rightarrow \bigwedge^{p-q}\t(R)$, where $\omega$ is a generator of $\bigwedge^qTe^q(R)$, c.f. Definition~\ref{definition:saturated_co_sheaves}. Combining this fact with the formula~\eqref{eq:formula_psi-1} we find the desired description of the image of $\psi^R_p(e^q)\circ j_p(e^q)$.\qedhere
\end{enumerate}
\end{proof}

\begin{theo}\label{thm:Lefschetz_hyp_sec}
Let $P$ be an $n$-dimensional integral polytope endowed with an integral polyhedral subdivision $K$, $X$ be the dual hypersurface of $K$, and $R$ be a ring in which $\delta(P)$ is invertible. The homological morphisms:
\begin{equation*}
\ibar_{p,q}\colon \Hbar_{p,q}(X;R)\rightarrow H_{p,q}(P; R)\,,
\end{equation*}
induced by the inclusions $\ibar_p\colon \Fbar_p^{X}\rightarrow F_p^{P}$ are isomorphisms for all $p+q<n-1$, and surjective for all $p+q=n-1$.
\end{theo}

\begin{proof}
In the light of Theorem~\ref{thm:cell_poinca_lef}, Proposition~\ref{prop:preparation_F0_1st_position}, and Proposition~\ref{prop:preparation_F1_1st_position} we can write:
\begin{equation*}
\begin{alcases}
H_{p,q}(P; R)&\cong H^{n-q}_c\big(K;H_n\left(\left(F^{P}_p\otimes R\right)_*\right)\big) \\ \\ \Hbar_{p,q}(X; R)&\cong H^{n-q}_c\big(K;H_n\left(\left(\Fbar^{X}_p\otimes R\right)_*\right)\big)
\end{alcases}\,.
\end{equation*}
Since $K$ is finite ($P$ being compact), cohomology with compact support is the same as cohomology. Let us denote by $G^{P}_p$ (resp. $G^{X}_p$), the sheaf $H_n((F^{P}_p\otimes R)_*)$ (resp. $H_n((\Fbar^{X}_p\otimes R)_*)$). We have the following commutative square relating the homological and cohomological counterparts of $\ibar_p$ and $j_p$:
\begin{equation*}
\begin{tikzcd}
H_{q}\big(K;\Fbar^{X}_p\otimes R\big) \arrow[r,"\ibar_p"] & H_{q}\big(K;F^{P}_p\otimes R\big) \\%
H^{n-q}\big(K;G^{X}_p\big) \arrow[r,"j_p" below] \ar[u,"\cong" left] & H^{n-q}\big(K;G^{P}_p\big) \ar[u,"\cong" right]
\end{tikzcd}
\end{equation*}
where the vertical isomorphisms are induced by the quasi-isomorphisms given by the second part of Theorem~\ref{thm:cell_poinca_lef}. We will prove that the commutativity is already satisfied on the level of chain and cochain complexes. Let us consider the following diagram:
\begin{equation*}
\begin{tikzcd}
\Omega_q\big(K;\Fbar^{X}_p\otimes R\big) \ar[r,"\ibar_p"] & \Omega_{q}\big(K;F^{P}_p\otimes R\big) \\%
\Omega_{n-q,n}\big(K;\Fbar^{X}_p\otimes R\big)\cap\ker(\partial_2) \arrow[u,hook,"\text{q.i.}" left] \ar[r,"\text{rest.}" above,"\ibar_p" below] & \Omega_{n-q,n}\big(K;F^{P}_p\otimes R\big)\cap\ker(\partial_2) \arrow[u,hook,"\text{q.i.}" right] \\%
E^1_{n-q,n} \arrow[u,equal] & E^1_{n-q,n} \arrow[u,equal] \\%
C^{n-q}\big(K;G^{X}_p\big) \ar[u,"\Phi" left] \ar[r,"j_p" below] & C^{n-q}\big(K;G^{P}_p\big) \ar[u,"\Phi" right]
\end{tikzcd}
\end{equation*}
We will show that its two squares are commutative. The morphisms $\Phi$ and the quasi-isomorphic inclusions come from Proposition~\ref{prp:degenerate_E2_lef}. Let $F$ denote either $\Fbar_p^{X}\otimes R$ or $F_p^{P} \otimes R$ and $G$ respectively designating $G_p^{X}$ or $G_p^{P}$. We have:
\begin{equation*}
\begin{alcases}
\Omega_{n-q,n}(K;F)& =\displaystyle\bigoplus_{e^{n-q}\,\leq\,e^{n~~}}F\left(e^{n-q};e^{n}\right)\otimes\Z\left(e^{n-q};e^{n}\right),\\%
\Omega_{n-q,n-1}(K;F)& =\displaystyle\bigoplus_{e^{n-q}\,\leq\,e^{n-1}}F\left(e^{n-q};e^{n}\right)\otimes\Z\left(e^{n-q};e^{n-1}\right)\,.
\end{alcases}
\end{equation*}
In this description, both these groups have a direct sum decomposition indexed by the $(n-q)$-cells of $K$. Both the morphisms $i_p$ and $\partial_2$ respect these decompositions which explains the commutativity of the upper square of the last diagram. Also in that setting, the value $\Hom(\Z(e^{n-q});G(e^{n-q}))$, on some cell $e^{n-q}$, is, modulo the action of the isomorphism $\Phi$, the kernel of $\partial_2$ restricted to the $(e^{n-q})$-component of $\Omega_{n-q,n}(K;F)$. By construction, $j_p$ is the restriction of $\ibar_p$ to the $(e^{n-q})$-component of the kernel of $\partial_2$, so the bottom square also commutes.

Because of Lemma~\ref{lemm:comparison_Hn(FX)_Hn(FP)}, the morphism $j_p(e^r)\colon G_p^{X}(e^{r})\rightarrow G_p^{P}(e^{r})$ is an isomorphism for all $r$-cells $e^{r}$ of dimension $r>p$. This implies that the cokernel of $j_p$ is trivial in dimension greater than $p$. By means of the long exact sequence induced in cohomology by the injective morphism $j_p$, we find that ${j_p^r\colon H^r(K;G^{X}_p)\rightarrow H^r(K;G^{P}_p)}$ is surjective when $r=p+1$ and invertible for all $r>p+1$. Theorem~\ref{thm:Lefschetz_hyp_sec} follows after performing the change of variables $r=n-q$.
\end{proof}

We would like to conclude with few remarks on the numbers $\delta(P)$ and $\theta(K)$ and the definition of the cosheaves $(F^{X}_p)_{p\,\in\,\N}$. Let us assume that $K$ is convex so we can speak in the terms of Tropical Geometry. From its definition the number $\delta(P)$ equals $1$ if and only if the toric variety associated with $P$ is smooth. Therefore, assuming $\delta(P)=1$ puts us closer to the \emph{Tropical Lefschetz Hyperplane Section Theorem} of C.~Arnal, A.~Renaudineau and K.~Shaw~\cite[Theorem~1.2 p.~1349]{Arn-Ren-Sha_Lef_sec}.
However, assuming $\delta(P)=\theta(K)=1$ does not implies that the dual hypersurface $X$ is smooth. The hypersurface $X$ is said to be smooth when $K$ is an unimodular triangulation. In this case $\theta(K)=1$. Therefore, assuming both $\delta(P)$ and $\theta(K)$ to be $1$ already generalises the statement of C. Arnal, A. Renaudineau and K. Shaw to some singular tropical hypersurfaces.

For a general polytope $Q$, the number $\theta(Q)$ seems difficult to compute. However, it seems computable for simplices. For segments it is always $1$. For a triangle $T$, a direct computation yields:
\begin{equation*}
\theta(T)=\frac{2\cdot\vol_\Z(T)\cdot\text{GCD}\{\vol_\Z(E) \colon E \text{ edge of }T\}}{\displaystyle\prod_{E\,\leq\,T} \vol_\Z(E)},
\end{equation*}
where $\vol_\Z(Q)$ denotes the integral volume of an integral polytope $Q$, i.e. its Lebesgue measure in the affine sub-space it spans normalised so that a parallelogram on a basis of the induced lattice has measure $1$.

When $\theta(K)$ equals $1$, Lemma~\ref{lemm:finite_index} and Definition~\ref{definition:theta} describe the cosheaf $\bigoplus_{p\,\in\,\N} F^{X}_p$ as the kernel of a contraction. When $\theta(K)$ is greater than $1$ the latter is the saturation of the former. Theorem~\ref{thm:Lefschetz_hyp_sec} suggests that if we alternatively defined the cosheaf $\bigoplus_{p\,\in\,\N} F^{X}_p$ as the kernel of a contraction then every tropical hypersurface dual to a polyhedral subdivision (combinatorially ample in the terminology of~\cite{Arn-Ren-Sha_Lef_sec}) in a projective non-singular tropical toric variety would satisfy the Tropical Lefschetz Hyperplane Section Theorem with integral coefficients.

\bibliographystyle{alpha}
\bibliography{Jules_Chenal_A_Poincare_Lefschetz_Theorem_for_Cellular_Cosheaves_and_an_Application_to_the_Tropical_Homology_of_Orbifold_Toric_Varieties}
\end{document}